\documentclass[11pt,reqno]{amsart}

\usepackage[normalem]{ulem}
\usepackage{cancel}
\usepackage{hyperref}

\usepackage{amsmath,amsthm,amssymb,latexsym,amsfonts,wrapfig,yfonts,mathrsfs,mathtools}
\usepackage[latin1]{inputenc}
\usepackage{graphicx,xcolor,bm,bbm,esint}
\usepackage{enumitem}

\usepackage{cleveref}

\newcommand{\norm}[1]{\left\lVert#1\right\rVert}
\newcommand{\sub}[1]{\mathfrak{#1}}

\allowdisplaybreaks

\usepackage{upgreek}
\def \N{\mathbb N}
\def \Z{\mathbb Z}

\def \R{\mathbb R}

\def \P{{\mathbb P}}
\def \pa{{\partial}}
\def \O{\mathcal{O}}
\def \T{\mathbb{T}}

\def \E{\mathbb{E}}

\newcommand{\Ac}{\mathcal A}

\newcommand{\Ic}{\mathcal I}
\newcommand {\Jc}{\mathcal{J}}
\newcommand{\Kc}{\mathcal K}
\newcommand{\Lc}{\mathcal L}

\newcommand{\Nc}{\mathcal N}

\newcommand{\Pc}{\mathcal P}

\newcommand{\Rc}{\mathcal R}
\newcommand{\Sc}{\mathcal S}
\newcommand{\Tc}{\mathcal T}

\newcommand{\Wc}{\mathcal W}

\newcommand{\Ss}{\mathscr S}

\newcommand{\Kf}{\mathfrak K}

\numberwithin{equation}{section}

\theoremstyle{plain}
\newtheorem{thm}{Theorem}[section]

\newtheorem{claim}[thm]{Claim}
\newtheorem{lem}[thm]{Lemma}
\newtheorem{prop}[thm]{Proposition}
\newtheorem{cor}[thm]{Corollary}
\theoremstyle{definition}
\newtheorem{defn}[thm]{Definition}

\newtheorem{rk}[thm]{Remark}

\theoremstyle{plain}

\theoremstyle{remark}

\theoremstyle{plain}

\title{Rigorous derivation of damped-driven wave turbulence theory}

\author[R. Grande]{Ricardo Grande}
\address{International School for Advanced Studies (SISSA), Via Bonomea 265, 34136, Trieste, Italy}
\email{rgrandei@sissa.it} 

\author[Z. Hani]{Zaher Hani}
\address{Department of Mathematics, University of Michigan, Ann Arbor, MI, USA}
\email{zhani@umich.edu} 

\begin{document}
	
	\begin{abstract}
	We provide a rigorous justification of various kinetic regimes exhibited by the nonlinear Schr\"{o}dinger equation with an additive stochastic forcing and a viscous dissipation. The importance of such damped-driven models stems from their wide empirical use in studying turbulence for nonlinear wave systems. The force injects energy into the system at large scales, which is then transferred across scales, thanks to the nonlinear wave interactions, until it is eventually dissipated at smaller scales. The presence of such scale-separated forcing and dissipation allows for the constant flux of energy in the intermediate scales, known as the inertial range, which is the focus of the vast amount of numerical and physical literature on wave turbulence.

Roughly speaking, our results provide a rigorous kinetic framework for this turbulent behavior by proving that the stochastic dynamics can be effectively described by a deterministic damped-driven kinetic equation, which carries the full picture of the turbulent energy dynamic across scales (like cascade spectra or other flux solutions). The analysis extends previous works in the unperturbed setting \cite{DengHani}--\cite{DengHani3} to the above empirically motivated damped driven setting. Here, in addition to the  size $L$ of the system, and the strength $\lambda$ of the nonlinearity, an extra thermodynamic parameter has to be included in the kinetic limit ($L\to \infty, \lambda\to 0$), namely the strength $\nu$ of the forcing and dissipation. Various regimes emerge depending on the relative sizes of $L$, $\lambda$ and $\nu$, which give rise to different kinetic equations. Two major novelties of this work is the extension of the Feynman diagram analysis to additive stochastic objects, and the sharp asymptotic development of the leading terms in that expansion.
		\end{abstract}

	\maketitle
	
  \tableofcontents

\section{Introduction}

In an informal manner of expression, wave turbulence theory can be described as \emph{Boltzmann meets Kolmogorov in a wavy setting.} This concise slogan surprisingly captures the essence of the matter, as wave turbulence lies at an intersection of kinetic theory, a legacy of Boltzmann, and Kolmogorov's vision of turbulence in the context of nonlinear interacting waves. This striking intersection of ideas was first uncovered by Zakharov in 1960's \cite{Zakharov65}, when he showed that the kinetic theory for wave systems, a theory that has been developed since the late 1920s to parallel Boltzmann's kinetic theory for particle systems, can be used to give a rather systematic formulation of Kolmogorov's ideas of turbulence, such as energy cascades and their power-type spectra, in the context of nonlinear wave systems. 

\medskip

Like in Boltzmann's particle kinetic theory, the central object in wave kinetic theory is a kinetic PDE, known as the \emph{wave kinetic equation}, which plays the role of the Boltzmann equation to describe the lowest order statistics of the wave interactions, in the limit where the number of waves goes to infinity and the interaction strength goes to zero. More precisely, the microscopic system in this wave setting is given by a wave-type nonlinear Hamiltonian PDE posed, for example, on a large box of size $L$, which we shall denote by $\T^d_L$, and with a nonlinearity of strength $\lambda$. The parameter $L$ encodes the number of interacting waves (which is $\O(L^d)$) and is sent to infinity in the thermodynamic limit, whereas $\lambda$ is sent to zero. The wave kinetic equation appears as the effective equation for the second order moments of the solution $u(x, t)$, that can be described in Fourier space by $\mathbb E |\widehat u(k, t)|^2$. Just like in the classical Boltzmann theory, the higher order statistics are supposed to reduce to products of the lowest order ones (here second order), a phenomenon that is known as \emph{propagation of chaos} \cite{ZLF,Nazarenko}. 

\medskip

Zakharov discovered that such wave kinetic equations sustain stationary in time, power-type, solutions that exhibit a \emph{constant flux} of energy across scales, like forward or backward cascades. Such solutions are exact analogs to Kolmogorov's spectra in hydrodynamic turbulence, and as such they offer a deep insight into the out-of-equilibrium statistical physics of a dispersive system including its turbulent features. It is safe to say that the vast majority of the physical and empirical literature on wave kinetic theory has focused on exhibiting, either numerically, experimentally, or empirically, such power-type solutions.  While Zakharov's solutions can be formally obtained from the unperturbed wave kinetic equation, they are only observed in practice (such as in numerical simulations or empirical experiments) in the presence of forcing and dissipation in the system that are separated\footnote{While it is permissible for the forcing and dissipation to coexist on certain scales, the effect of each should be negligible on the scales where the other is significant.} by a large stretch of intermediate scales, known as the \emph{inertial range}, over which the system is left basically unperturbed. The presence of such scale-separated forcing and dissipation allows for the constant flux of energy in the \emph{inertial range}, on which the wave kinetic equation is supposed to hold. Of course, such constant flux of energy across scales cannot be sustained in a conservative system (i.e. without forcing and dissipation), at least not for finite energy solutions. 

\medskip 

Mathematically speaking, a lot needs to be done to get to, or even make sense of, this exciting combination of ideas presented by the Kolmogorov-Zakharov spectra. Naturally, the first step would be to justify the wave kinetic theory and establish its rigorous mathematical foundations. This has been an active avenue of research over the past ten years. We will review this progress in more detail later in the introduction, but that justification has been achieved successfully for some semilinear models (like the nonlinear Schr\"odinger equation), and even over arbitrarily long time intervals \cite{DengHani5}, which was the first result of its kind, even compared to the particle setting which only recently featured a similar progress \cite{DengHaniMa24}. Nonetheless, all the recent progress on rigorous wave kinetic theory has been restricted to situations where either the equation is left conservative or in the presence of phase-randomizing forces that don't introduce energy into the system, and act on all scales. In particular, this literature does not cover the setting that is most relevant to observing the Kolmogorov-Zakharov spectra, namely the one we described above in which an additive force injects energy into the system which is eventually dissipated by damping. This is precisely the setting we consider in this manuscript, where we study a forced-dissipated nonlinear Schr\"odinger model introduced by Zakharov and L'vov \cite{ZakharovLvov} for that same purpose.

Roughly speaking, our results provide a rigorous kinetic framework for the turbulent behavior of the damped-driven stochastic dynamics. More precisely, we show that the stochastic dynamics can be effectively described by a deterministic damped-driven kinetic equation, in a similar spirit to the recently established results in the unperturbed setting by Deng and the second author \cite{DengHani}--\cite{DengHani5}. This is despite the different randomness structures in the two settings: in the unperturbed setting, the randomness is coming from the initial data, whereas in our setting here, it is mainly coming from the stochastic force (especially when the initial data is zero). We do not attempt in this manuscript to match the progress achieved by the most recent results in the unperturbed setting, namely \cite{DengHani2}--\cite{DengHani5}, but only restrict to the subcritical regime of the problem (precisely, the parallel results to \cite{DengHani}). This is meant to provide a proof of concept that the same results should carry over to the damped-driven setting by combining the ideas in this paper and those used in the unperturbed setting (along with some technical developments; see Section \ref{intro.extension} below). 

Finally, we iterate that the importance of our results derives mainly from the fact that they show that \emph{one can reduce the formulation and proof of practically any turbulence conjecture for waves to the study of a single deterministic damped-driven kinetic equation}. We believe that this equation is the right theoretical framework to understand the, up until now formal and non-rigorous, turbulent statistics (like cascade spectra) of nonlinear wave systems. Moreover, it allows to draw parallels, and exchange insights, with hydrodynamic turbulence theory, where such kinetic formulation of turbulence is not present. The latter was the aim of the recent work of Bedrossian \cite{Bed24}, which explained the strong analogy between wave and hydrodynamic turbulence and formulated various conjectures based on this analogy. As we explain in \Cref{into.analysisofwke}, those conjectures can now be studied, numerically and rigorously, through our derived kinetic equation. 

\subsection{Setup}

We consider the following forced and damped NLS equation:
\begin{equation}\label{eq:intro_stochasticNLS}
 \pa_t u + i \Delta u = i \lambda\, \left(|u|^2- \frac{2}{L^d} \int_{\T_L^d} |u|^2 \right)\, u - \nu\, (1-\Delta)^r u + \sqrt{\nu} \, \dot{\beta}^{\omega}, \qquad u(t=0, x)=u_{\mathrm{in}}(x),
\end{equation}
on the torus $\T_L^d=\R^d/(L\Z^d)$, where $\lambda,\nu>0$ and $r\in (0,1]$.

\smallskip

Let $\Z_L^d=(L^{-1}\Z)^d$ be the dual of $\T_L^d$, with the Fourier transform and its inverse defined by
\begin{equation}\label{eq:defFourier}
\widehat{u}(t,k)=u_k(t) = {L^{-d/2}}\, \int_{\T_L^d} u(x)\, e^{-2\pi i  k\cdot x}\, dx\qquad \mbox{and} \qquad u(t,x)=L^{-d/2}\, \sum_{k\in \Z^d_L} u_k(t)\, e^{2\pi i k\cdot x} .
\end{equation}

Randomness comes into the system through the additive stochastic force $\beta^\omega$, and possibly the initial data $u_{\mathrm{in}}$ if the latter is nonzero. More precisely, the initial data and the forcing are taken as follows:
\begin{equation}\label{eq:intro_data}
\begin{split}
u^{\omega}(t,x)|_{t=0}&=u_{\mathrm{in}}^{\omega}(x)=L^{-d/2} \, \sum_{k\in\Z_L^d} c_k \, \eta_k^{\omega} \, e^{2\pi i k\cdot x}\\
 \beta^{\omega}(t,x)&= L^{-d/2}\, \sum_{k\in\Z_L^d} b_k\, \beta_k^{\omega} (t) \, e^{2\pi i k\cdot x}.
 \end{split}
\end{equation}
Here, the family $\{\eta_k^{\omega}\}_{k\in\Z_L^d}$ are independent complex standard Gaussian random variables, while\\ $\{\beta_k^{\omega} (t)\}_{k\in\Z_L^d}$ are independent complex standard Wiener processes, which are also independent of  $\{\eta_k^{\omega}\}_{k\in\Z_L^d}$. The coefficients $c_k=c(k)$ and $b_k =b(k)$, $c,b:\R^d \rightarrow \R$, are sufficiently smooth and decaying functions of $k$, which we take to be Schwartz functions to simplify the statement of our main results (although a finite amount of decay and regularity given by the first $100d$ Schwartz seminorms is sufficient for all our results).

In Fourier space, the dissipation operator appearing in \eqref{eq:intro_stochasticNLS} is a multiplier with coefficients given by:
\begin{equation}\label{eq:dissipation}
\gamma_k := (1+ |k|^2)^r,\quad \mbox{where}\quad  |k|^2=\sum_{j=1}^d |k^{(j)}|^2\quad \mbox{for}\quad k=(k^{(1)},\ldots,k^{(d)})\in\R^d.
\end{equation}

\subsubsection{The model} 
The normalization prefactor $L^{-d/2}$ in \eqref{eq:intro_data} guarantees that, for each $x\in\T_L^d$, the initial datum and the forcing have average size which is \emph{asymptotically independent of $L$}. Indeed,
\begin{equation}
\mathbb{E} |u_{\mathrm{in}}(x)|^2 = L^{-d} \sum_{k\in\mathbb{Z}_L^d} c_k^2 \longrightarrow \int_{\mathbb{R}^d} c(k)^2 \, \mathrm{d}k = \mathcal{O} (1)\qquad \mbox{as } L\rightarrow\infty.
\end{equation}
As a result, the nonlinearity in \eqref{eq:intro_stochasticNLS} has size $\lambda$, and the \emph{kinetic timescale} (or Van-Hove timescale) is given by 
\begin{equation}\label{eq:intro_Tkin}
T_{\mathrm{kin}}:=\lambda^{-2}.
\end{equation}

The parameters $\nu$ in front of the dissipation and $\sqrt{\nu}$ in the forcing balance input/output of energy in the system\footnote{While other choices of parameters are covered in \Cref{rk:timescales}, we focus our presentation on the most interesting choice.}.
Indeed, an application of the It\^{o} formula for \eqref{eq:intro_stochasticNLS} (cf. \Cref{subsec:bal_energy}) yields the balance of energy equation:
\begin{equation}\label{eq:intro_BalanceEnergy}
\E \norm{u(t)}_{L^2(\T_L^d)}^2 - \sum_{k\in\Z_L^d} c_k^2 = -2 \nu \, \E\int_0^t  \norm{(1-\Delta)^{r/2} u (t')}_{L^2(\T_L^d)}^2\, dt' + 2\nu \,  t \, \sum_{k\in\Z_L^d} b_k^2
\end{equation}
which, by the Gr\"onwall inequality, implies that 
\begin{equation}\label{eq:intro_BalanceEnergy2}
\E \norm{u(t)}_{L^2(\T_L^d)}^2 \leq (1-e^{-2\nu t})\,\sum_{k\in\Z_L^d} b_k^2 + e^{-2\nu t}\, \sum_{k\in\Z_L^d} c_k^2.
\end{equation}
It follows that the timescale for the forcing to input enough energy into the system, known as \emph{forcing timescale}, is given by
\begin{equation}\label{eq:intro_Tfor}
T_{\mathrm{for}}=\nu^{-1}.
\end{equation}

We remark here that this choice of additive stochastic forcing that is balanced by dissipation is standard in the study of turbulent systems, both wave-type and hydrodynamic. It allows for the necessary scale separation between the needed source and sink of energy, in a way that leaves an inertial range, where Hamiltonian interactions dominate, to form. For the sake of giving some representative examples, in the setting of wave turbulence, we cite the works \cite{MMT1, CMMT, Onorato, ZakharovLvov}; and in the setting of hydrodynamic turbulence we mention the mathematical works \cite{HaiMatt06, Bed1, Bed2}.

\subsubsection{Scaling laws}\label{intro.scalinglaws} Our goal is to study the dynamics of \eqref{eq:intro_stochasticNLS} as $\nu\rightarrow 0$, $\lambda\rightarrow 0$ and $L\rightarrow\infty$. These parameters are related by \emph{scaling laws} that specify the relative rates of their convergence to the limits. In particular, we assume that $\lambda = L^{-\kappa_1}$ and $\nu \sim L^{-\kappa_2}$ for some $\kappa_1, \kappa_2>0$. In the absence of forcing ($\kappa_2=\infty$), there are natural restrictions on $\kappa_1$ for one to obtain a useful kinetic description in the thermodynamic limit. This was precisely explained in \cite[Section~1.3]{DengHani4} (see also \cite{DengHaniExpo}), where it is shown that the optimal range of $\kappa_1$ on a general torus is the interval $(0, 1)$ (for example, this is the optimal range on the square torus), which can be extended to $(0, d/2)$ in the case of a generically irrational torus. The same limitations carry over to our forced-dissipated setting, so we shall restrict ourselves to the same ranges of $\kappa_1$. Note that $\kappa_1$ dictates the size of the kinetic timescale $T_{\mathrm{kin}}=L^{2\kappa_1}$ in terms of the system size $L$, and the mentioned range of $\kappa_1$ limits $T_{\mathrm{kin}}$ to the range $1\ll T_{\mathrm{kin}}\ll L^2$ on a general torus, and $1\ll T_{\mathrm{kin}}\ll L^d$ for a generically irrational torus.

\subsubsection{Heuristics}\label{intro.heuristics} Heuristically, one expects that the limiting dynamics of \eqref{eq:intro_stochasticNLS} to depend on whether the ratio
\begin{equation}\label{eq:intro_varrho}
\varrho := \frac{T_{\mathrm{kin}}}{T_{\mathrm{for}}}=\nu \lambda^{-2} \sim L^{2\kappa_1-\kappa_2}
\end{equation}
 tends to a fixed constant in $(0, \infty)$, to 0, or to infinity in the limit. 
 This trichotomy leads to the following asymptotic regimes:
\begin{enumerate}[label=(\roman*)]
\item If $2\kappa_1-\kappa_2=0$, nonlinear wave interactions and forcing/dissipation are comparable, and we expect a forced/damped wave kinetic equation in the kinetic limit.
\item If $2\kappa_1-\kappa_2<0$, nonlinear interactions dominate over the forcing and dissipation, and thus we expect similar behavior as in the case of the unforced and undamped cubic Schr\"odinger equation.
\item If $2\kappa_1-\kappa_2>0$, the forcing and dissipation dominate over nonlinear interactions, and we expect a kinetic limit entirely determined by forcing and dissipation.
\end{enumerate}
 
Roughly speaking, our results confirm these heuristics by deriving the adequate kinetic limit in each regime. But before stating the result, let us review the kinetic theory for the NLS equation in the absence of forcing and dissipation.

\subsubsection{The Wave Kinetic Equation for the NLS equation}\label{intro.NLSlit}

In the second case where $\varrho\rightarrow 0$, one expects the effective dynamics to be similar to those in the unforced and undamped NLS equation\footnote{This equation is equivalent to \eqref{eq:intro_stochasticNLS} with $\nu=0$, as one can go from one to the other using the phase change $u \to e^{\frac{2i\lambda t}{L^d}\int |u|^2 dx}\cdot u$ and noting that $\int |u|^2 dx$ is a conserved quantity for either equation.}:
 \begin{equation}
 i\pa_t v + \Delta v = \lambda |v|^2\, v \qquad x \in \T_L^d=\R^d/(L\Z^d).
 \end{equation}
For $\lambda = L^{-\kappa_1}$, and as $L\rightarrow\infty$, the average energy density of the $k$-th Fourier mode at the kinetic time is given by $\mathbb E |\widehat v(t, k)|^2$. This captures the full lowest (second) order statistics of the solution. The kinetic theory postulates that this energy density is described by the solution $n(t, k)$ of the \emph{Wave Kinetic Equation} given by:

\begin{equation}\label{eq:intro_WKE}
\pa_t n (t, k) = \Kc \left(n(t, k)\right) \qquad k\in\R^d,
\end{equation}
where 
\begin{equation}\label{eq:def_K}
\begin{split}
\Kc (\phi)(k)=\int_{\substack{k=k_1-k_2+k_3\\ k_1,k_2,k_3\in\R^d}}\  & 4\pi\,\delta_{\R}(|k_1|^2 - |k_2|^2 + |k_3|^2 -|k|^2) \left( \frac{1}{\phi} - \frac{1}{\phi_1}+\frac{1}{\phi_2} -\frac{1}{\phi_3}\right) \, \phi\, \phi_1\, \phi_2\, \phi_3 \,dk_1\, dk_2\, dk_3
\end{split}
\end{equation}
and where $\phi_j = \phi(k_j)$, $\phi=\phi(k)$ and $\delta$ is the Dirac delta. More precisely, one expects that 
\begin{equation}\label{approx0}
\mathbb E |\widehat v(t, k)|^2\approx n(\frac{t}{T_{\mathrm{kin}}}, k), \qquad as \; L \to \infty.
\end{equation}

\subsection{Statement of results}

The aim of this manuscript is to extend the above picture to the forced-dissipated equation \eqref{eq:intro_stochasticNLS}, which is a problem of particular scientific relevance from the point of view of turbulence as we explained earlier. Given that this is the first paper proving such an approximation in this setting, we decided to keep it at reasonable length and only prove the subcritical results, i.e. ones that hold for times $\leq T_{\mathrm{kin}}^{1-\epsilon}$ for some $\epsilon>0$. This corresponds to the progress in the unperturbed case (i.e. when $\nu=0$) that preceded the results in \cite{DengHani2}--\cite{DengHani4}. More precisely, we will be content with proving the analogous results to \cite{DengHani}, and focusing the presentation on the new ingredients needed to address the more challenging stochastic input. 
As we shall remark in Section \ref{intro.extension}, we contend that the ideas in this manuscript can be combined with the more sophisticated combinatorial and analytic methodology in \cite{DengHani2, DengHani4} in order to obtain similar results at the kinetic timescale. 

\medskip

The main result of this paper makes the intuition described in Section \ref{intro.heuristics} above rigorous, at least over some subcritical time intervals. Our first theorem is stated for a specific scaling law for which we can cover the full subcritical regime for the problem. The following theorem will cover other scaling laws.

\begin{thm}\label{thm:main} Consider equation \eqref{eq:intro_stochasticNLS} for $d\geq 2$, and $r\in (0,1]$. Assume that the initial data and the stochastic force are taken as in \eqref{eq:intro_data} with $b,c\in \Sc (\R^d; [0,\infty))$ being Schwartz functions. Fix $0<\varepsilon\ll 1$ to be sufficiently small.

Suppose that $\kappa_1=1+\frac{\varepsilon}{2}-$ so that $T_{\mathrm{kin}}=L^{2\kappa_1}=L^{2+\varepsilon-}$ (here, $C-$ represents a number strictly smaller than and sufficiently close to $C$), and that $\nu \sim L^{-\kappa_2}$, $\kappa_2>0$. Then there exists some $0<\delta=\delta(d, r, \varepsilon, \kappa_2)\ll 1$ such that, as $L\rightarrow\infty$, we have that the solution to the initial value problem given by \eqref{eq:intro_stochasticNLS}-\eqref{eq:intro_data} satisfies:
\begin{equation}\label{eq:main}
\E |\widehat{u}(t,k)|^2 = n_{\mathrm{app}} \left(\frac{t}{T_{\mathrm{kin}}},k\right) + \O_{\ell^{\infty}_k} \left( L^{-\delta}\, \frac{t}{T_{\mathrm{kin}}}\right)
\end{equation}
for all $L^{\delta}\leq t\leq L^{-\varepsilon} T_{\mathrm{kin}}$. Here, $n_{\mathrm{app}}$ is given according to the regime of $2\kappa_1-\kappa_2$ as follows:

\

\begin{enumerate}[label=(\roman*)]
\item if $2\kappa_1-\kappa_2=0$ (i.e. $T_{\mathrm{kin}}\sim T_{\mathrm{for}}$), then
\begin{equation}
\begin{split}
n_{\mathrm{app}} (t,k) & = f(t,k) + \int_0^t  e^{-2\varrho \gamma_k (t-s)} \Kc (f) (k,s)\, ds,\\
f(t,k) & =c_k^2\, e^{-2\varrho \gamma_k t}+\int_0^t e^{-2\varrho \gamma_k (t-s)} 2\varrho b_k^2\, ds
\end{split}
\end{equation}
is the first iterate of the damped/driven Wave Kinetic Equation (WKE):
\begin{equation}\label{eq:forced_WKE}
\pa_t n = \Kc (n) - 2\varrho \gamma \, n +2\varrho\, b^2,\qquad n(0,k)=c_k^2,
\end{equation}
for $\Kc$ defined in \eqref{eq:def_K} and $\varrho$ defined in \eqref{eq:intro_varrho}.

\

\item if $2\kappa_1-\kappa_2<0$ (i.e. $T_{\mathrm{kin}}\ll T_{\mathrm{for}}$), then
\begin{equation}\label{eq:WKE_cubic}
n_{\mathrm{app}} (t,k)=c_k^2 + t\, \Kc (c^2_{\cdot}) (t,k)
\end{equation}
is the first iterate of the WKE:
\begin{equation}
\pa_t n = \Kc (n), \qquad n(0,k)=c_k^2,
\end{equation}

\

\item if $2\kappa_1-\kappa_2<0$ (i.e. {$T_{\mathrm{kin}}\gg T_{\mathrm{for}}$}), then we have that 
\begin{equation}
n_{\mathrm{app}} (t,k)=c_k^2\, e^{-2\varrho \gamma_k t}+ \frac{b_k^2}{\gamma_k} \left( 1- e^{-2\varrho \gamma_k t}\right) +\O(L^{2\kappa_1-\kappa_2})
\end{equation}
is the solution to the equation:
\begin{equation}\label{eq:trivial_WKE}
\pa_t n = - 2\varrho \gamma \, n +2\varrho\, b^2,\qquad n(0,k)=c_k^2.
\end{equation}
\end{enumerate}
\end{thm}

We summarize this result in  \Cref{fig:kinetic_table} below.

Our second theorem covers more general scaling laws and aspect rations of the torus. The latter is important to expose the effects that different boundary conditions can have on the result:

\begin{thm}\label{thm:main2} 
Consider the equation:
\begin{equation}\label{eq:intro_stochasticNLS_rescaledT}
 \pa_t u + i \Delta_{\zeta} u = i \lambda\, \left(|u|^2- \frac{2}{L^d} \int_{\T_L^d} |u|^2 \right)\, u - \nu\, (1-\Delta_{\zeta})^r u + \sqrt{\nu} \, \dot{\beta}^{\omega}, \qquad u(t=0, x)=u_{\mathrm{in}}(x),
\end{equation}
on the torus $\T_L^d=\R^d/(L\Z^d)$. The torus can be rational or irrational, which is captured by the (rescaled) Laplacian:
\begin{equation}
\Delta_{\zeta} := \sum_{j=1}^{d} \zeta_j \, \pa_j^2 , \qquad \zeta=(\zeta_1,\ldots,\zeta_d)\in [1,2]^d ,
\end{equation}
where $\zeta_j$ represent the aspect ratios of the torus. Under the assumptions of \Cref{thm:main}, fix $\updelta>0$ and suppose that $T$ satisfies:
\begin{equation}\label{eq:main2_T}
L^{\delta}\leq T\leq 
\begin{cases}
L^{2-\updelta} & \mbox{if}\ \zeta\ \mbox{arbitrary},\\
L^{d-\updelta} & \mbox{if}\ \zeta\ \mbox{generically irrational},\\
\end{cases}
\end{equation}
\begin{equation}\label{eq:main2_Tkin}
 T_{\mathrm{kin}} \geq 
\begin{cases}
L^{2\updelta}\, T^2 & \mbox{if}\ T\leq L,\\
L^{2+2\updelta} & \mbox{if}\ L\leq T\leq L^2,\\
L^{-2+2\updelta}\, T^2 & \mbox{if}\ L^2\leq T.
\end{cases}
\end{equation}
Then \eqref{eq:main} and the rest of the results in \Cref{thm:main} hold.
\end{thm}

\begin{figure}[h!]
\includegraphics[scale=0.24]{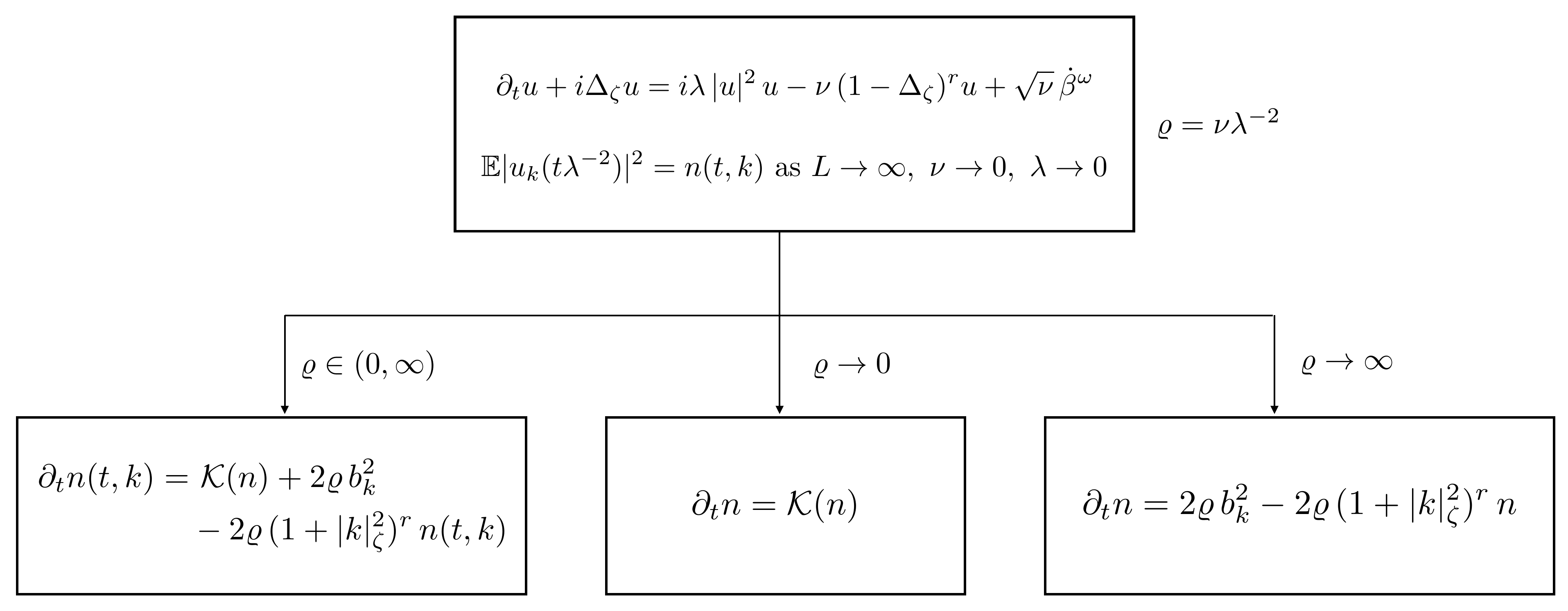}
\caption{Limiting dynamics depending on relative size between nonlinearity and forcing.}
\label{fig:kinetic_table}
\end{figure}

\begin{rk}
The main novelty of the above theorems lies in the regime $\varrho \in (0, \infty)$ where both the kinetic collisional dynamics and the forcing terms are balanced. We note here that in this case, the initial data can be well taken to be zero, in which case, the randomness is only coming from the stochastic force. The effective forced-damped kinetic equation has nontrivial solutions in this case if $b_k \neq 0$.
\end{rk}

\begin{rk}\label{rk:timescales}
Our techniques actually allow us to study the more general equation with three parameters in the nonlinearity:
\begin{equation}\label{eq:intro_stochasticNLS_alt}
 \pa_t u + i \Delta u = i \lambda\, \left(|u|^2- \frac{2}{L^d} \int_{\T_L^d} |u|^2 \right)\, u - \nu_1\, (1-\Delta)^r u + \nu_2 \, \dot{\beta}^{\omega}, \qquad u(t=0, x)=u_{\mathrm{in}}(x),
\end{equation}
for arbitrary $\lambda, \nu_1,\nu_2$. A scaling law in the form of $(\lambda,\nu_1,\nu_2)=( L^{-\kappa_1}, \widetilde{\nu_1} L^{-\kappa_2}, \widetilde{\nu_2} L^{-\kappa_3})$ with $\widetilde{\nu_1}, \widetilde{\nu_2} \in (0,\infty)$ would have to be imposed. To keep the number of different regimes more tractable, we have restricted attention in this paper to the case when $\kappa_2=2\kappa_3$, which gives the three regimes mentioned in Theorem \ref{thm:main}. The most interesting case, where $2\kappa_1=\kappa_2=2\kappa_3$ which is covered by our result, leads to the kinetic equation featuring all the different elements in the nonlinearity, namely
\begin{equation}\label{eq:forced_WKE_alt}
\pa_t n = \Kc (n) - 2\widetilde{\nu_1} \gamma \, n +2\widetilde{\nu_2}^2\,  b^2,\qquad n(0,k)=c_k^2.
\end{equation}
This is exactly \eqref{eq:forced_WKE} after a simple rescaling. If this condition $2\kappa_1 = \kappa_2=2\kappa_3$ is not satisfied, certain terms would be missing from the kinetic equation depending on the regime, similar to what we saw in Theorem \ref{thm:main}. It is worth pointing out that \eqref{eq:forced_WKE_alt}, especially in the case when $\widetilde{\nu_1} \ll \widetilde{\nu_2}$, gives a direct and rigorous route to checking the turbulence conjectures for the damped-driven (NLS) equation formulated in \cite{Bed24}, based on drawing analogies with hydrodynamic turbulence. We shall remark on this further in Section \ref{into.analysisofwke}.
\end{rk}

\subsection{Some background literature}
Equation \eqref{eq:intro_stochasticNLS} was introduced\footnote{Technically, the model proposed by Zakharov and L'vov allows for more general wave propagation and interaction. In this paper we have chosen to work with the cubic NLS equation, but in general one would write \eqref{eq:intro_stochasticNLS} in terms of the Fourier coefficients $u_k$ and replace the terms $i\Delta u$ and $|u|^2\, u$ by $\delta H/\delta \overline{u}_k$ for a different Hamiltonian $H$. See \cite{ZakharovLvov} and \cite{Falkovich} for full details. For the purpose of this paper, we have chosen the Wick ordered Hamiltonian 
\[
H(u)=\frac{1}{2}\, \int_{\T_L^d} |\nabla u|^2 + \frac{\lambda}{4} \int_{\T_L^d} :\!|u|^4\!: \ ,
\]
which gives rise to \eqref{eq:intro_stochasticNLS} above. See also \Cref{sec:setup}.} 
by Zakharov and L'vov in 1975 \cite{ZakharovLvov}. This model was first studied by Dymov-Kuksin \cite{DymovKuksin,DymovKuksin2}, and Dymov-Kuksin-Maiocchi-Vl\u{a}du\c{t} \cite{DymovKuksin3}. In \cite{DymovKuksin,DymovKuksin2}, the authors consider the case when $\varrho\in (0,\infty)$ and with zero initial data, and study the $d$-th order Picard iterates of \eqref{eq:intro_stochasticNLS} (which they call quasi-solutions). They show that the latter are well-approximated by solutions to \eqref{eq:forced_WKE}. In \cite{DymovKuksin3}, the authors focus on the limit $\nu\rightarrow 0$ (with fixed $L$) of the second order Picard iterates of \eqref{eq:intro_stochasticNLS} with zero initial data, and then they investigate the limit as $L\rightarrow \infty$.

This paper extends the results in \cite{DymovKuksin,DymovKuksin2} in three ways. Firstly and most importantly, our results pertain to the actual solution of \eqref{eq:intro_stochasticNLS} rather than its $d$-th order Pircard iterates. Secondly, we give a complete picture of the different asymptotic regimes $(i)-(iii)$ described in \Cref{thm:main}, and thirdly, we allow initial data to be random rather than being zero. 
All this requires estimating all the Picard iterates of \eqref{eq:intro_stochasticNLS} up to an arbitrarily high (but finite) threshold, and an analysis of the remainder term. This will require proving new sharp asymptotics of the main terms, and obtaining combinatorial results for Feynmann diagrams in the spirit of those in \cite{DengHani}. 

Following some earlier foundational works focused on some related problems in linear or equilibrium settings \cite{ESY08, LukSpohn}, the first attempt at rigorously deriving a kinetic equation from a dispersive wave system was in the joint work of the second author with Buckmaster, Germain, and Shatah \cite{BGHS}. There, the authors prove \eqref{approx0} on generically irrational tori but only on \emph{subcritical\footnote{These are intervals that are of the form $[0, T_{\mathrm{kin}}^{1-\updelta}]$ for some $\updelta>0$. On such intervals, the diagrammatic expansion of the NLS equation is subcritical in the sense that every subsequent iterate is better behaved than the previous one by a negative power of $L$.} time intervals}, like $t\leq T_{\mathrm{kin}}^{1/2}$. This was later improved by Deng and the second author in \cite{DengHani}, and independently by Collot and Germain in \cite{CollotGermain}, to reach much longer, but still subcritical, times intervals of the form $t\leq T_{\mathrm{kin}}^{1-\epsilon}$, at least for some scaling laws. The full justification of the approximation \emph{at the kinetic timescale}, which corresponds to proving \eqref{approx0} for time intervals of the form $[0, \delta \cdot T_{\mathrm{kin}}]$ where $\delta>0$ is independent of $L$ and $\lambda$, was first done in a series of works \cite{DengHani2, DengHani3,DengHani4} (with \cite{DengHani3} describing the asymptotics of the higher order statistics). This covered the case of general tori with their corresponding full range of scaling laws $\kappa_1 \in (0,1)$. We refer to the expository note \cite{DengHaniExpo} for more details about this progress. The rigorous derivation was recently extended to arbitrarily long times in \cite{DengHani5} (i.e. for as long as the solution of the wave kinetic equation exists).

In the stochastic setting, in addition to the works mentioned earlier, we should mention the recent works of Staffilani-Tran \cite{StaffilaniTran} and Hannani-Rosenzweig-Staffilani-Tran \cite{HRST}, who studied the kinetic limit of a discrete KdV-type equation with a time-dependent Stratonovich stochastic forcing. That forcing is constructed to effectively randomize the phases of the Fourier modes without injecting energy into the system.

Finally, we would like to draw a contrast between the kinetic and non-equilibrium paradigm explored in this paper, and the vast literature in stochastic PDE focused on constructing invariant measures for damped-driven equations. We won't make any attempt at being exhaustive in citing the literature there, but we single out the highly influential and representative work of Hairer and Mattingly \cite{HaiMatt06} in the context of the 2D Navier-Stokes equation with additive forcing, in which the authors construct the unique invariant measure for the dynamics. The existence of such measures is not always known or guaranteed; for example this seems to be open for equation \eqref{eq:intro_stochasticNLS}. Moreover, the turbulent features of such invariant measures (e.g. energy cascade behavior) is something that is not a priori evident, and usually very hard to establish. At a philosophical level, one can think of the kinetic paradigm derived in this paper as describing an explicit (in terms of its description of the turbulent features) transient regime that would hold on the way to the convergence of the dynamics to its unique invariant measure, if such measure exists. At the point of this supposed convergence, the solution of the kinetic equation would converge to one of its stationary solutions. 

\subsection{Strategy of proof}\label{intro.proof} We explain briefly here some of the ideas in the proof, focusing mainly on the novelties and differences with respect to the $\nu=0$ case. 

\subsubsection{Picard iterates with stochastic terms}

The main strategy consists of expanding the solution in terms of Picard iterates (also referred to as Dyson series expansion)
\begin{equation}\label{eq:intro_Picard}
u = u^{(0)} +u^{(1)} + \ldots +u^{(N)} + R_{N+1},
\end{equation}
up to some large order $N$ to be determined. The key is to obtain effective estimates for the iterates as well as the error, which will require taking $N$ large enough in order for the error contribution to be negligible.

Each $u^{(n)}$ can be written as a sum over ternary trees of order $n$, which admit similar combinatorics to those in \cite{DengHani}. The main difference with the case of the cubic NLS equation ($\nu=0$) is the fact that $u^{(0)}$ contains both the linear evolution of the initial data and a stochastic term, namely
\begin{equation}\label{eq:intro_first_iterate}
u^{(0)}(t) = e^{-it\Delta-\nu\,t(1-\Delta)^r} u_{\mathrm{in}}^{\omega} + \sqrt{\nu}\, \int_0^t e^{-i(t-s)\Delta-(t-s)\, \nu\,(1-\Delta)^r} d\beta^{\omega}(s)
\end{equation}
This means that the solution profile $e^{it\Delta}u^{(0)}(t)$ now has a nontrivial time dependence compared to the $\nu=0$ case, which leads to new difficulties at the level of estimating the higher iterates $u^{(n)}$, which are basically iterative self-interactions of the zeroth iterate $u^{(0)}$. In the case of the cubic NLS equation ($\nu=0$ case), this means that $u^{(n)}(t)$ can be written as a sum over expressions involving a product of $2n+1$ Gaussian random variables $(\eta^{\omega}_k)_{k\in\Z_L^d}$, which are \emph{time-independent}. In our setting, however, the stochastic terms in $u^{(n)}$ require a different treatment, due to their nontrivial time dependence coming from  the stochastic integral in \eqref{eq:intro_first_iterate}. In fact, $u^{(n)}$ can be seen as an integral expression involving a product of $2n+1$ copies of $u^{(0)}(t_j)$ for different times $t_j\in [0,t]$. 
As such, obtaining sharp bounds on $u^{(n)}$ requires precise bounds on the two-point correlations $\E[ u^{(0)}(t) \, \overline{u^{(0)}(t')}]$, or rather, their Fourier transform in time (to capture their time oscillation). To make this comparison explicit, it suffices to compare the contributions of the two terms in \eqref{eq:intro_first_iterate}. Let us rescale time by setting $t=sT_{\mathrm{kin}}$, let $\chi(s)$ be a smooth cut-off in time supported in $[0,2]$ such that $\chi=1$ in $[0,1]$,  and define 
\[
\begin{split}
v^{(0)}(s) & = \chi(s) \, e^{-\varrho\,s(1-\Delta)^r}  u_{\mathrm{in}}^{\omega} \\
w^{(0)}(s) & = \chi(s)\, \int_0^s \sqrt{\varrho}\,e^{-(s-s')\, \varrho\,(1-\Delta)^r}d\beta^{\omega}(s'),
\end{split}
\]
so that $u^{(0)}(t)= e^{-it\Delta} [ v^{(0)}(\frac{t}{T_{\mathrm{kin}}})  + w^{(0)}(\frac{t}{T_{\mathrm{kin}}}) ]$ for times $t\in [0,T_{\mathrm{kin}}]$. Denoting by $\widetilde {f_k}(\tau)$ the spacetime Fourier transform at $(k, \tau)$, one can show that the two-point correlations of the Fourier transform of the terms coming from the initial data, i.e. $v^{(0)}(s)$, admit bounds of the form,
\[
\E[ \widetilde{v_k}^{(0)}(\tau) \, \overline{\widetilde{v_{k'}}^{(0)}(\tau')}] \lesssim_N c_k^2\, \langle \tau\rangle^{-N}\, \langle \tau'\rangle^{-N}\, \delta_{k-k'} \qquad \qquad \forall N\in\N.
\]
On the other hand, those coming from the stochastic forcing are only small when they they are sufficiently separated in temporal frequency, e.g.
\[
\E[ \widetilde{w_k}^{(0)}(\tau) \, \overline{\widetilde{w_{k'}}^{(0)}(\tau')}] \lesssim_N b_k^2 \, \langle \tau -\tau'\rangle^{-N}\, \langle \tau \rangle^{-1} \, \langle \tau'\rangle^{-1}\, \delta_{k-k'} \qquad \qquad \forall N\in\N \ \mbox{and for}\ \varrho=1.
\]

We note that the two-point correlations between $\widetilde{v}^{(0)}$ and $\widetilde{w}^{(0)}$ vanish on account of the independence between forcing and initial datum. This is a crucial property in the computation of correlations between Picard iterates in \eqref{eq:intro_Picard}:
\begin{equation}\label{eq:intro_ncorrelations}
\E [ u^{(n_1)}_k(t) \overline{u^{(n_2)}_k(t)}], \qquad n_1,n_2\in \N\cup \{0\}.
\end{equation}
For instance, beyond terms coming exclusively from the initial datum or the forcing, $\E|u^{(1)}_k(t)|^2$ is also made of cross-terms given by integrals over $s_1,s_2\in [0,t]$ of
\[
\E[ v_{k_1}^{(0)}(s_1) \overline{w^{(0)}_{k_2}(s_1)} v^{(0)}_{k_3}(s_1)  \overline{v^{(0)}_{k_1'} (s_2)} w^{(0)}_{k_2'}(s_2) \overline{v^{(0)}_{k_3'} (s_2)} ] \stackrel{\rm{ind.}}{=} 
\E[ v_{k_1}^{(0)}(s_1)  v^{(0)}_{k_3}(s_1)  \overline{v^{(0)}_{k_1'} (s_2)}  \overline{v^{(0)}_{k_3'} (s_2)} ]\, \E[ \overline{w^{(0)}_{k_2}(s_1)} w^{(0)}_{k_2'}(s_2)].
\]
Whenever the initial datum is nonzero\footnote{When $c_k=0$ in \eqref{eq:intro_data}, $v^{(0)}\equiv 0$ and the only terms that give rise to the kinetic integral come from the forcing $w^{(0)}$.}, all such terms give rise to top-order contributions towards the kinetic equation \eqref{eq:forced_WKE}, as shown in \Cref{lem:towardsWKE}.

As can be seen from this example, the new input coming from the stochastic forcing --which is decoupled from the input coming from the initial datum via independence-- requires a new analysis to manage its contribution to the Dyson series. 
In the case of top-order contributions such as \eqref{eq:intro_ncorrelations} with $n_1+n_2\leq 2$, this input can be explicitly computed, cf. \Cref{lem:towardsWKE}-\Cref{thm:integral_to_delta}. Terms of the form \eqref{eq:intro_ncorrelations} with $n_1+n_2\geq 3$, as well as remainder terms, are then shown to be lower order in the kinetic limit. In order to do so, we derive a novel approach to bound all the new terms coming from the stochastic forcing --see \Cref{thm:FT_higher_iterates} and \Cref{thm:n_iterate}--, which we combine with some counting bounds for the number of Feynman diagrams similar to those used in \cite{DengHani}, which we use as a black box. See \Cref{sec:iterates} for full details.

\subsubsection{Convergence to WKE}

As a consequence of the analysis mentioned above, one obtains that, with overwhelming probability, the Picard iterates $u^{(n)}$ in \eqref{eq:intro_Picard} decay like a geometric sequence for $n \leq N$. An estimate on the remainder $R_{N+1}$ can be obtained by analyzing the linearization of the (NLS) equation around the first $N$ iterates; this is done in Section \ref{sec:error}. From this, it follows that
\begin{equation}
\E |u_k (s T_{\mathrm{kin}})|^2  = \E \left[ |u_k^{(0)} (sT_{\mathrm{kin}})|^2 +  |u_k^{(1)} (sT_{\mathrm{kin}})|^2  +2 \mbox{Re}\, \overline{u_k^{(0)} (sT_{\mathrm{kin}})} u_k^{(2)} (sT_{\mathrm{kin}}) \right] + o\left(s\right).
\end{equation}

The terms in the braces should coverge to the term $n_{\mathrm{app}}$ in \eqref{eq:main} which give rise to the kinetic kernel $\Kc$ in \eqref{eq:def_K}. This is another place where there is a major departure with the $\nu=0$ setting. Indeed, the convergence to the kinetic kernel is very different in the regimes $\varrho \ll 1$ and $\varrho \gtrsim 1$. In order to illustrate this, let us describe the computation in the case of $\E|u_k^{(1)} (t)|^2$ for $t=s\, T_{\mathrm{kin}}$ with $0\leq s\leq 1$. As shown in \Cref{sec:WKE}, 
\[
\begin{split}
\E|u_k^{(1)} (sT_{\mathrm{kin}})|^2   = 4\, \left(\frac{1}{\lambda\,L^{d}} \right)^2\! \sum_{\substack{k=k_1-k_2+k_3\\}}  \int_0^s\int_{s_2}^s  &\ e^{-\varrho\gamma_k (2s-s_1-s_2)-\varrho (s_1-s_2) \sum_{j=1}^{3} \gamma_{k_j}} \\ &   \cos \left( T_{\mathrm{kin}} \Omega\, (s_1-s_2)\right)\, \prod_{j=1}^{3} \E |u_{k_j}^{(0)}(s_2 T_{\mathrm{kin}})|^2 \,ds_2\, ds_1  +o (s).
\end{split}
\]
where $\Omega = |k_1|_{\zeta}^2 - |k_2|_{\zeta}^2+|k_3|_{\zeta}^2 -|k|_{\zeta}^2$. 

When $\varrho \rightarrow 0$, one may exploit oscillations by integrating by parts the cosine in order to recover the top order terms which give rise to the Wave Kinetic Equation \eqref{eq:intro_WKE}. This critically exploits the fact that we gain powers of $\varrho$ in the remainder terms from integrating by parts.

For $\varrho\gtrsim 1$, however, such terms are at least as important, and they even dominate as $\varrho\rightarrow\infty$. For that reason, a different approach is needed. Upon integration in $s_1$, we obtain
\[
\begin{split}
\E|u_k^{(1)} (sT_{\mathrm{kin}})|^2  = L^{-2d} &\ \sum_{\substack{k=k_1-k_2+k_3\\ }}   \left( \int_0^s e^{-2\varrho\gamma_k (s-s')} \,\nu^{-1}\, h_0(\nu^{-1}\Omega, \Gamma_{-}) \, \prod_{j=1}^{3} \E |u_{k_j}^{(0)}(s'T_{\mathrm{kin}})|^2\, ds' \right.\\
& - \left. \int_0^s  e^{-2\varrho\Gamma_{+} (s-s')} \, \nu^{-1}\, [h_1(\nu^{-1} \Omega, \Gamma_{-})-h_2(\nu^{-1}\Omega, \Gamma_{-})]\, \prod_{j=1}^{3} \E |u_{k_j}^{(0)}(s'T_{\mathrm{kin}})|^2\, ds'\right),
\end{split}
\] 
where $\Gamma_{\pm}=\pm \gamma_k + \sum_{j=1}^3 \gamma_{k_j}$ and
\[
\begin{split}
h_0(x,y)= \frac{4y}{y^2+x^2}, \qquad h_1(x,y) =\frac{4y\,\cos(\varrho (s-s') x)}{y^2+x^2}, \qquad h_2(x,y) = \frac{4x\,\sin(\varrho (s-s') x)}{y^2+x^2}.
\end{split}
\]

In the limit as $L \to \infty$, and hence $\nu \to 0$, all three kernels $h_0, h_1, h_2$ should behave like an approximation to the identity in their first variable and lead to $\delta(\Omega)$ factors as a leading contribution. While the term coming from $h_0$ leads to a key contribution towards the WKE, the contribution of the terms in $h_1$ and $h_2$ should vanish in the limit. This vanishing is the result of an exact cancellation of their respective contributions in the limit, which have to be computed precisely. 
Interestingly, in the complementary regime $\varrho\rightarrow 0$ we described above, the delta function coming from $h_2$ is the one leading to the WKE, while the deltas arising from $h_0$ and $h_1$ cancel in this asymptotic regime. This sharp asymptotic development of the oscillatory kernels in $h_1$ and $h_2$ as $\nu\rightarrow 0$ in the whole regimes of $\varrho \gtrsim 1$ constitutes one of the main novelties of this work, as they do not fall under the umbrella of the kernels studied in the previous works (e.g. Proposition 6.1 in \cite{DengHani3}). Note that $h_2$ is not even absolutely integrable in the $x$ variable, which requires exploiting its oscillations to reveal its approximation to the identity behavior.

\subsection{Future perspectives}

\subsubsection{Getting to the kinetic timescale}\label{intro.extension} It should be clear from our discussion in Section \ref{intro.proof} that the bulk of the analysis in this paper goes on two fronts: The first is reducing the combinatorial estimates on the Feynman diagrams to their analogs in the $\nu=0$ case, by proving and effectively utilizing the time correlations estimates of the stochastic integral terms. The second front is obtaining the exact asymptotics of the first iterates (here $u^{(n)}$ for $n\leq 2$) that converge to the corresponding iterate of the wave kinetic equation. Here, the analysis was quite different and more involved compared to the $\nu=0$ case. 

We believe that those same ideas can be extended, in combination to the more sophisticated analytical and combinatorial machinery in \cite{DengHani3, DengHani4}, to obtain a justification of the approximation \eqref{approx0} all the way to time intervals that reach the kinetic timescale, say of the form $[0, \delta \cdot T_{\mathrm{kin}}]$, or even longer time intervals as was done recently in \cite{DengHani5}.

\subsubsection{Analysis of the damped/driven WKE}\label{into.analysisofwke}

As we explained earlier, one of the main postulates in wave turbulence theory is that of scale separation: the random force operates at large scales, the dissipation acts on small scales, and there are intermediate scales (known as inertial range) in which the effect of the viscosity and the force are negligible. In this inertial range, physical arguments, as well as various experiments, suggest that $\E |u_k (t)|^2$ should resemble a power law in $k$, in the spirit of the stationary-in-time power-law solutions to \eqref{eq:intro_WKE} that were discovered by Zakharov. Unfortunately, the highly singular behavior of those power-law solutions and the very formal fashion in which they satisfy the equation pose substantial blocks to their rigorous study, and as a result, our understanding of the turbulence aspects of wave turbulence theory. 

This was a major motivation for us to study the forced-dissipated equation, as we believe that the analogous turbulent solutions to \eqref{eq:forced_WKE} should be relatively better behaved. This raises the highly important open problem of investigating the existence of such solutions to \eqref{eq:forced_WKE} that resemble the Kolmogorov-Zakharov power-law spectra at least within the inertial range.

Clearly, the analogy with hydrodynamic turbulence is quite strong here. In a recent manuscript, Bedrossian explored this analogy and formulated some conjectures on turbulence and anomalous dissipation for the damped-driven NLS equation \cite{Bed24}. While these conjectures were formulated in terms of invariant measures and averages of solutions to the damped-driven NLS, they can now be rigorously and concretely checked at the level of solutions of the \eqref{eq:forced_WKE} (or more generally \eqref{eq:forced_WKE_alt} with $\widetilde \nu_2 \gg \widetilde \nu_1$). This provides yet another motivation to study the dynamics of the latter kinetic equations.

\subsubsection{Admissible dissipations}

One of the basic questions in turbulence is that of universality: do the statistics of the problem depend on the external forcing or the internal friction? Which features are common to different turbulent systems? 

A key property satisfied by our chosen dissipation function $\gamma_k$ in \eqref{eq:dissipation} is that
 \begin{equation}\label{eq:intersection_mfolds}
\left\lbrace |k_1|^2 - |k_2|^2 + |k_3|^2 -|k|^2=0\right\rbrace \cap \lbrace -\gamma_k + \sum_{j=1}^3 \gamma_{k_j}=0\rbrace = \emptyset,
 \end{equation}
 see \Cref{VIP_claim}. From a technical viewpoint, this condition plays an important role in our derivation of the kinetic equation. However, we believe that our argument also allows to deal with other dissipation functions that satisfy weaker versions of \eqref{eq:intersection_mfolds}. For example, in \Cref{rk:degen_diss}, we explain how one can modify the argument in \Cref{sec:WKE} to deal with the case of the dissipation given by the Laplacian $\gamma_k=|k|_{\zeta}^2$, which violates condition \eqref{eq:intersection_mfolds}. More precisely, we show that it suffices to prove that the number of points in a neighborhood of the intersection \eqref{eq:intersection_mfolds} is sufficiently small (see Lemma \ref{lem:degen_diss}). Interestingly enough, our analysis suggests that one might always need a condition in the spirit of \eqref{eq:intersection_mfolds} (or a quantitative version of it in the spirit of Lemma \ref{lem:degen_diss}) where the equality in the second set is replaced by $<$.
 
Further study on the necessity of such criteria, as well as deriving quantitative estimates on the number of points in a neighborhood of the intersection of these manifolds for more general dissipations $\gamma_k$ are interesting open problems.

\subsection{Outline} In \Cref{sec:setup}, we present the setup of the problem, including the Picard iteration scheme and the main bounds on the iterates. In \Cref{sec:iterates}, we prove said bounds on the Picard iterates. \Cref{sec:WKE} is devoted to proving the convergence to the various wave kinetic equations depending in the three regimes presented in Theorem \ref{thm:main}. In \Cref{sec:error}, we prove the main bounds for the remainder term in the Picard iteration scheme. Finally, \Cref{sec:appendix} summarizes a few results from probability and combinatorics that are useful in our work.

\section{Setup of the problem}\label{sec:setup}

Consider the Zakharov-L'vov stochastic model for wave turbulence: 
\begin{equation}\label{eq:stochasticNLS}
 \pa_t u + i \Delta_{\zeta} u = i \lambda\, \left(|u|^2- 2M(u)\right)\, u - \nu\, (1-\Delta_{\zeta})^r u + \sqrt{\nu} \, \dot{\beta}^{\omega}
\end{equation}
where $(t,x)\in \R \times \T_L^d$, where $\T_L^d=\frac{\R^d}{L\Z^d}$ and $L\gg 1$ is the size of the large box, $\lambda,\nu>0$ are parameters that will be discussed later, $r\in (0,1]$ and 
\begin{equation}\label{eq:randomforcing}
 \beta^{\omega}(t,x):= L^{-d/2}\, \sum_{k\in\Z_L^d} b_k \beta_k^{\omega} (t) \, e^{2\pi i k\cdot x},
\end{equation}
where $\{ \beta_k (t)\}_{k\in \Z_L^d}$ are standard independent complex Wiener processes and $b_k=b(k)$ is a Schwartz function on $\R^d$. The normalized mass is defined as follows:
\begin{equation}\label{eq:mass}
\begin{split}
M(u)(t)& :=\fint_{\T_L^d} |u(t,x)|^2\, dx =  L^{-d}\,\int_{\T_L^d} |u(t,x)|^2\, dx.
\end{split}
\end{equation}
Note that, unlike in the case of the cubic NLS equation with random initial data, the mass is not conserved. However, we will show that $M(u)$ is uniformly bounded in time.

The rescaled Laplacian is given by:
\begin{equation}
\Delta_{\zeta} := \sum_{j=1}^{d} \zeta_j \, \pa_j^2 , \qquad \zeta=(\zeta_1,\ldots,\zeta_d)\in [1,2]^d ,
\end{equation}
where $\zeta$ controls the aspect ratios of the torus, and thus whether it is rational or irrational. The dissipation operator appearing in \eqref{eq:stochasticNLS} is a Fourier multiplier with coefficients given by:
\begin{equation}\label{eq:dissipation_rescaled}
\gamma_k := (1+ |k|_{\zeta}^2)^r,\quad \mbox{where}\quad  |k|^2_{\zeta}=\sum_{j=1}^d \zeta_j |k^{(j)}|^2\quad \mbox{for}\quad k=(k^{(1)},\ldots,k^{(d)})\in\R^d.
\end{equation}

The choice of normalization $L^{-d/2}$ in \eqref{eq:randomforcing} guarantees that that the size of the forcing at each point in the torus is asymptotically independent on the size of the torus $L$. Indeed, we define
\begin{equation}\label{eq:defB}
 B:=\norm{b}_{\ell^2_k (\Z^d_L)}^2= \sum_{k\in \Z^d_L} b_k^2,
 \end{equation}
which has size $L^d$ since
\begin{equation}\label{eq:aboutB}
L^{-d}\, B\longrightarrow \int_{\R^d} |b(\xi)|^2\, d\xi\sim 1\quad \mbox{as}\quad L\rightarrow \infty.
\end{equation}

Together with \eqref{eq:stochasticNLS}, let us prescribe random initial data at $t=0$ as follows:
\begin{equation}\label{eq:def_randomdata}
u(t,x)|_{t=0}= u_0(x)=L^{-d/2} \, \sum_{k\in\Z_L^d} c_k \, \eta_k^{\omega} \, e^{2\pi i k\cdot x}
\end{equation}
where $c_k=c(k)$ is a Schwartz function on $\R^d$ and $\{\eta_k\}_{k\in\Z_L^d}$ are iid complex Gaussian random variables with zero mean, $\E \eta_k\eta_j=0$, and $\E \eta_k\overline{\eta_j}=\delta_{kj}$. Moreover, we will assume that $\eta_k$ is independent of $\beta_k(t)$ for all $t\geq 0$. 

We may now interpret the stochastic PDE in \eqref{eq:stochasticNLS} in the integral sense, i.e.
\[ 
u(t)=u_0 + \int_0^t \left[ -i \Delta_{\zeta} u (t') +  i \lambda\, \left(|u(t')|^2-2M(u(t'))\right)\, u(t') - \nu\, (1-\Delta_{\zeta})^r u(t') \right] \, dt' +  \sqrt{\nu} \, \beta^{\omega}(t).
\]

\subsection{Balance of energy}\label{subsec:bal_energy} The choice of the coefficients $\nu$ and $\sqrt{\nu}$ for the dissipation and the forcing terms, respectively, is to balance the energy in the system. Using the standard notation for It\^o processes, we write 
\[ 
du = \left( -i \Delta_{\zeta} u  +  i \lambda\, \left(|u|^2-2 M(u)\right)\, u - \nu\, (1-\Delta_{\zeta})^{r} u \right) dt +\sqrt{\nu}\, d\beta(t):=\mu_t dt + \sqrt{\nu}\, d\beta(t).
\]
Then one can apply the It\^o formula, which yields:
\[ 
d(|u(t)|^2) = \left(\overline{u(t)}\,\mu_{t} + u(t)\,\overline{\mu_{t}} + 2 \nu\, L^{-d} B \right) \, dt + \sqrt{\nu}\, \overline{u(t)}\, d\beta(t) +  \sqrt{\nu}\, u(t)\, d\overline{\beta(t)}.
\]
By integrating in time and in space, we are led to the balance of energy equation:
\begin{equation}\label{eq:BalanceEnergy}
\E \norm{u(t)}_{L^2(\T_L^d)}^2 - \E \norm{u(0)}_{L^2(\T_L^d)}^2 = -2 \nu \, \E\int_0^t  \norm{(1-\Delta_{\zeta})^{r/2} u (t')}_{L^2(\T_L^d)}^2\, dt' + 2\nu \, B\, t .
\end{equation}
Here we see that our choice of coefficients $\nu$ and $\sqrt{\nu}$ led to the terms of the right-hand side having comparable size and opposite signs, in such a way that the dissipation and the forcing terms balance each other. 
Moreover, we find
\begin{equation}\label{eq:forcing_timescale_0}
 \E \norm{u(t)}_{L^2(\T_L^d)}^2 \leq \E \norm{u(0)}_{L^2(\T_L^d)}^2 + 2\nu B\, t,
 \end{equation}
which is a good estimate for times $t\lesssim \nu^{-1}$. In view of \eqref{eq:aboutB}-\eqref{eq:def_randomdata}, 
\[
\E \norm{u(0)}_{L^2(\T_L^d)}^2\sim L^d \ \mbox{and}\ B\sim L^d,\]
and therefore $\E \norm{u(t)}_{L^2(\T_L^d)}^2\sim L^d$ for times $t\lesssim \nu^{-1}$. This \emph{a priori} bound still holds for times $t\gtrsim \nu^{-1}$, but we must perform a more careful analysis of \eqref{eq:BalanceEnergy}.
Using the fact that 
\[
\norm{u(t)}_{L^2(\T_L^d)}\leq \norm{(1-\Delta_{\zeta})^{r/2} u (t)}_{L^2(\T_L^d)},
\]
equation \eqref{eq:BalanceEnergy} yields the following inequality:
\begin{equation}
\E \norm{u(t)}_{L^2(\T_L^d)}^2 - \E \norm{u(0)}_{L^2(\T_L^d)}^2 \leq  -2 \nu \, \int_0^t \, \E \norm{u (t')}_{L^2(\T_L^d)}^2\, dt' + 2\nu \, B\, t .
\end{equation}
An application of Gr\"onwall's inequality yields:
\begin{equation}\label{eq:BalanceEnergy2}
\E \norm{u(t)}_{L^2(\T_L^d)}^2 \leq B + (\E \norm{u_0}_{L^2(\T_L^d)}^2 - B) \, e^{-2\nu t}.
\end{equation}
Note that at times $t\gtrsim \nu^{-1}$ the effect of $\E \norm{u_0}_{L^2(\T_L^d)}^2$ fades off and the forcing takes over, while keeping $\E \norm{u(t)}_{L^2(\T_L^d)}^2\sim L^d$ at all times.

\subsection{Two timescales}

There are two natural timescales that are important in this problem. As suggested by \eqref{eq:forcing_timescale_0}, the timescale in which the forcing truly starts to operate is proportional to $\nu^{-1}$, and therefore we define
\begin{equation}\label{eq:forcing_timescale}
T_{\mathrm{for}} = \nu^{-1}.
\end{equation}
On the other hand, we have the kinetic timescale given by the typical size of the cubic nonlinearity in \eqref{eq:stochasticNLS}. For all times $t$ where the solution exists, we know that $\E \norm{u(t)}_{L^2(\T^d_L)}^2 \sim L^d$ thanks to \eqref{eq:BalanceEnergy2} and \eqref{eq:aboutB}. As a consequence, $|u|$ has size $1$ on average, so the strength of the cubic nonlinearity is $\lambda$. The kinetic timescale is the inverse of the square of this number:
\begin{equation}\label{eq:kinetic_timescale}
 T_{\mathrm{kin}} = \lambda^{-2}.
\end{equation}

Given the two natural timescales in this problem, there are three different scenarios depending on the relative size between $T_{\mathrm{kin}}$ and $T_{\mathrm{for}}$.
\begin{enumerate}
\item If $T_{\mathrm{kin}}\ll T_{\mathrm{for}}$, the forcing is so weak that it has no effect in the limit. The limiting dynamics will be governed by the WKE for the cubic NLS equation with random initial data, as in \cite{DengHani}.
\item If $T_{\mathrm{kin}}\sim T_{\mathrm{for}}$, there is a balance between forcing and the cubic nonlinearity. In the limit, we obtain a damped/driven wave kinetic equation.
\item If $T_{\mathrm{kin}}\gg T_{\mathrm{for}}$, the forcing is so strong that the limiting dynamics are dominated by it. This is a somewhat uninteresting scenario, where we expect the forcing and dissipation to entirely determine the dynamics. 
\end{enumerate}

Let us highlight that our techniques may handle a third timescale connected to the dissipation, should one consider different parameters $\nu_1$ and $\nu_2$ for the dissipation and the forcing. Such considerations lead to additional possible kinetic equations, see \Cref{rk:timescales} for more detailes.

We will study \eqref{eq:stochasticNLS} for times $t\in [0,T]$ where we wish to take $T$ to be as close to $T_{\mathrm{kin}}$ as possible. Given the heuristics above, we define
\begin{equation}\label{eq:def_vartheta}
\vartheta := \nu T = \frac{T}{T_{\mathrm{for}}}.
\end{equation}
Let us highlight that $ \vartheta\leq \varrho$ for $\varrho$ introduced in \eqref{eq:intro_varrho}.

\subsection{Fourier formulation and rescaled time} 

Given that we are interested in times $t\in [0,T]$ for some $T$ as close as possible to $T_{\mathrm{kin}}$, we rescale time so that we may work on the interval $[0,1]$. In order to do so, it is convenient to work at the level of the Fourier coefficients. Recall that
\begin{equation}
u(x)=L^{-d/2}\, \sum_{k\in \Z^d_L} u_k\, e^{2\pi i k\cdot x} \qquad \mbox{and} \qquad u_k = L^{-d/2}\,  \int_{\T_L^d} u(x)\, e^{-2\pi i  k\cdot x}\, dx.
\end{equation}
By the Plancherel theorem, we have that
\[ \norm{u}_{L^2(\T^d_L)}=\left(\int_{\T_L^d} |u(x)|^2\, dx \right)^{1/2} = \left( L^{-d}\, \sum_{k\in \Z^d_L} |u_k|^2\right)^{1/2}.\]

We rewrite \eqref{eq:stochasticNLS} in terms of the Fourier coefficients:
\begin{equation}
\begin{split}
 \pa_t u_k - i |k|_{\zeta}^2 u_k 
 & =  \frac{i\lambda}{L^{d}}\, \sum_{k_1-k_2+k_3=k} \epsilon_{k_1, k_2, k_3} \, u_{k_1} \overline{u_{k_2}} u_{k_3} - \nu \gamma_k u_k + \sqrt{\nu}\, b_k \dot{\beta}_k^{\omega}
 \end{split}
\end{equation}
where $\gamma_k := (1+|k|_{\zeta}^2)^{r}$, and 
\begin{equation}\label{eq:epsilon}
\epsilon_{k_1,k_2,k_3} = \begin{cases}
+1 & \mbox{if}\ k_2\notin \{ k_1,k_3\};\\
-1 & \mbox{if}\ k_1=k_2=k_3;\\
0 & \mbox{otherwise}.
\end{cases}
\end{equation}

We set $v(t):= e^{it\Delta_{\zeta}} u(t)$ so that $v_k = e^{-i t|k|_{\zeta}^2 } u_k$, which yields
\begin{equation}\label{eq:transformation}
 \pa_t v_k= \frac{i\lambda}{L^{d}}\, \sum_{k_1-k_2+k_3=k} \epsilon_{k_1, k_2, k_3} \, v_{k_1} \overline{v_{k_2}} v_{k_3} \, e^{it\Omega_{k2}^{13}} - \nu \gamma_k v_k + \sqrt{\nu}\, b_k \dot{\beta}_k^{\omega}
\end{equation}
 and 
\begin{equation}
 \Omega_{k2}^{13}=\Omega(k_1,k_2,k_3,k)= |k_1|_{\zeta}^2 - |k_2|_{\zeta}^2 + |k_3|_{\zeta}^2 - |k|_{\zeta}^2.
\end{equation}
Note that in \eqref{eq:transformation} we use the fact that $e^{-it |k|_{\zeta}^2} \,\dot{\beta}_k^{\omega}(t)$ is a white noise with the same distribution as $\dot{\beta}_k^{\omega}(t)$.

Finally, we rescale the time variable $t\mapsto Tt$. Using the fact that $T^{-1/2}\, \beta (T t)$ is a Wiener process with the same distribution as $\beta (t)$, we find that 
\begin{equation}\label{eq:Fouriercoeff}
\pa_t w_k+ \vartheta \, \gamma_k\, w_k = \frac{i\lambda\, T}{L^{d}}\, \sum_{k_1-k_2+k_3=k} \epsilon_{k_1, k_2, k_3} \,  w_{k_1} \overline{w_{k_2}} w_{k_3} \, e^{iTt\Omega_{k2}^{13}} + \sqrt{\vartheta}\, b_k\, \dot{\beta_k}
\end{equation}
with $t\in [0,1]$ and $\vartheta=\nu T$ as introduced in \eqref{eq:def_vartheta}.

\subsection{Picard iterates}

At this stage it is convenient to introduce some additional notation. For functions $\bm{a}(t)=(a_k (t))_{k\in\Z_L^d}$, $\bm{b}$ and $\bm{c}$ we define
\begin{equation}\label{eq:defW}
\mathcal{W}(\bm{a},\bm{b},\bm{c})_k (t) = \frac{i\lambda T}{ L^{d}}\, \sum_{k_1-k_2+k_3=k} \epsilon_{k_1, k_2, k_3} \, a_{k_1} \overline{b_{k_2}} c_{k_3} e^{iTt \Omega_{k2}^{13}}.
\end{equation}
Fix a smooth cutoff $\varphi\in C_c^{\infty}(\R)$ with $\mbox{supp}(\varphi)\subset (-2,2)$ such that $\varphi(t)=1$ if $t\in [-1,1]$. Then we define
\begin{equation}\label{eq:defI}
 (\mathcal{I} \bm{a})_k (t)  = \varphi(t) \, \int_{0}^{t} e^{-\vartheta \gamma_k (t-t')} a_k(t')\, \varphi(t')\,dt',
 \end{equation}
 with $\vartheta$ as in \eqref{eq:def_vartheta}. 
Writing $\bm{w}(t)=(w_k(t))_{k\in\Z_L^d}$, we may rewrite \eqref{eq:Fouriercoeff} as
\begin{equation}\label{eq:defW2}
 \dot{\bm{w}} + \vartheta\, \bm{\gamma} \cdot \bm{w} = \mathcal{W}(\bm{w}, \bm{w},\bm{w}) + \sqrt{\vartheta}\, \bm{b} \cdot \dot{\bm{\beta}},
 \end{equation}
which, in integral form, reads
\begin{equation}\label{eq:defW3}
 \bm{w}(t)-  e^{-\vartheta \bm{\gamma} t}\bm{c}\cdot \bm{\eta}=\mathcal{IW}(\bm{w}, \bm{w},\bm{w})(t) + \sqrt{\vartheta}\, \bm{b}\cdot \mathcal{I}\dot{\bm{\beta}},
 \end{equation}
for times $t\in [0,1]$. Note also that we need to properly interpret $\mathcal{I}\dot{\bm{\beta}}$ as an It\^{o} integral, which we detail in \eqref{eq:cutoff_w0} below.

Next we write an expansion for $\bm{w}$ based on Picard iteration: 
\begin{equation}\label{eq:Picarditeration}
 \bm{w} = \bm{w}^{(0)} + \bm{w}^{(1)} + \ldots +\bm{w}^{(N)} + \bm{\mathcal{R}}_{N+1}, 
 \end{equation}
so that 
\begin{align}
 \bm{w}^{(0)} & = e^{-\vartheta \bm{\gamma} t}\bm{c}\cdot \bm{\eta}+ \sqrt{\vartheta}\,\bm{b}\cdot \mathcal{I}\dot{\bm{\beta}},\\
 \bm{w}^{(1)} & = \mathcal{IW}(\bm{w}^{(0)}, \bm{w}^{(0)},\bm{w}^{(0)}),\\
  \bm{w}^{(N)} & = \sum_{n_1+n_2+n_3=N-1}\mathcal{IW}(\bm{w}^{(n_1)}, \bm{w}^{(n_2)},\bm{w}^{(n_3)}),\label{eq:Picard_iterates}\\
   \bm{\mathcal{R}}_{N+1} & = \bm{w}  - \sum_{n=0}^{N} \bm{w}^{(n)}.\label{eq:error}
\end{align}

In order to derive an equation for $\bm{\mathcal{R}}_{N+1}$, we plug \eqref{eq:error} into \eqref{eq:defW3}
\begin{equation}\label{eq:erroreq}
 \bm{\mathcal{R}}_{N+1} = \mathcal{L}(\bm{\mathcal{R}}_{N+1}) + \mathcal{Q}(\bm{\mathcal{R}}_{N+1}) +\mathcal{C}(\bm{\mathcal{R}}_{N+1})+ \sum_{\substack{N\leq n_1+n_2+n_3 \\ n_1,n_2,n_3\leq N}}\mathcal{IW}(\bm{w}^{(n_1)}, \bm{w}^{(n_2)},\bm{w}^{(n_3)})
\end{equation}
where 
\begin{align}
 \mathcal{L}(\bm{\mathcal{R}}_{N+1}) & = \sum_{0\leq n_1 , n_2 \leq N} 2\,\mathcal{IW}(\bm{\mathcal{R}}_{N+1}, \bm{w}^{(n_1)},\bm{w}^{(n_2)})+\mathcal{IW}(\bm{w}^{(n_1)},\bm{\mathcal{R}}_{N+1}, \bm{w}^{(n_2)}),\\
 \mathcal{Q}(\bm{\mathcal{R}}_{N+1})  & = \sum_{0\leq n_1\leq N} 2\,\mathcal{IW}(\bm{\mathcal{R}}_{N+1}, \bm{\mathcal{R}}_{N+1},\bm{w}^{(n_1)})+\mathcal{IW}(\bm{\mathcal{R}}_{N+1},\bm{w}^{(n_1)},\bm{\mathcal{R}}_{N+1}),\\
 \mathcal{C}(\bm{\mathcal{R}}_{N+1}) & = \mathcal{IW}(\bm{\mathcal{R}}_{N+1},\bm{\mathcal{R}}_{N+1},\bm{\mathcal{R}}_{N+1})
\end{align}

From now on, we let $w^{(n)}(t,x)=L^{-d/2}\,\sum_{k\in\Z_L^d} w_k^{(n)}(t) e^{2\pi i k\cdot x}$ be the function corresponding to the sequence $\bm{w}^{(n)}(t)$. Below we state our main results regarding the size of the Picard iterates and the remainder term.

\subsection{Main estimates}

For $s,b\geq 0$, consider the spaces
\[
\begin{split}
\norm{a}_{h^b} & = \left( \int_{\R} \langle \tau\rangle^{2b}\, |\widehat{a}(\tau)|^2\, d\tau\right)^{1/2},\\
\norm{\bm{a}}_{h^{s,b}} & = \left( L^{-d}\, \sum_{k\in\Z_L^d} \int_{\R}\langle \tau\rangle^{2b}\,\langle k\rangle^{2s}\, |\widehat{\bm{a}}_k(\tau)|^2\, d\tau\right)^{1/2}.
\end{split}
\]

In order to state our main estimates, let us introduce the quantity
\begin{equation}\label{eq:def_rho}
\rho:= \left\lbrace\begin{array}{ll}
T & \mbox{if } 1\leq T\leq L,\\
L & \mbox{if } L\leq T\leq L^2,\\
T\,L^{-1}  & \mbox{if } L^2\leq T \mbox{ and }\zeta\ \mbox{is generic}.
\end{array}\right.
\end{equation}

Note that the assumptions in \Cref{thm:main2} guarantee that $\rho\lesssim L^{-\delta} \sqrt{T_{\mathrm{kin}}}$. The definition of $\rho$ is motivated by a combinatorial result which we record in \Cref{sec:appendix}.

Our first result provides an upper bound on the size of the Picard iterates:

\begin{prop}\label{thm:n_iterate} 
Suppose that $n\geq 1$, $s\geq 0$, and $b>\frac{1}{2}$. For any $0<\theta\ll 1$, we have that
\begin{equation}\label{eq:n_iterate}
\begin{split}
\sup_{k\in\Z_L^d} \norm{\bm{w}_k^{(0)}}_{L^{\infty}_t} & \lesssim L^{\theta}\\
\sup_{k\in\Z_L^d} \langle k\rangle^{s} \norm{ \bm{w}_k^{(n)}}_{h^b} & \lesssim L^{\theta + c (b-1/2)} \, \left(\frac{\rho}{\sqrt{T_{\mathrm{kin}}}}\right)^{n-1}\, \sqrt{\frac{T}{T_{\mathrm{kin}}}}
\end{split}
\end{equation}
$L$-certainly (i.e. on an event with probability $\geq 1- C\, e^{-L^{\theta'}}$ for some $0<\theta'\ll 1$ and $C>0$ depending on $\theta'$).
\end{prop}

Our next result allows us to control the size of the remainder in \eqref{eq:error}.

\begin{prop}\label{thm:error}  Suppose that $T$ and $T_{\mathrm{kin}}$ are as in \Cref{thm:main2}, and let $s>d/2$, and $b>\frac{1}{2}$.
Then the remainder in \eqref{eq:error} satisfies
\[
\norm{\bm{\Rc}_{N+1}}_{h^{s,b}} \leq L^{-\delta N}
\]
$L$-certainly, where $\delta>0$ is as in \eqref{eq:main2_Tkin}.
\end{prop}

The proof of \Cref{thm:n_iterate}  will be given in \Cref{sec:iterates}, while \Cref{sec:error} is devoted to the proof of \Cref{thm:error}.

\section{Estimates on the iterates}\label{sec:iterates}

In this section we prove \Cref{thm:n_iterate}. We start our analysis with the case $n=0$ and then develop a general approach to handle any $n\in\N$.

\subsection{Zeroth iterate}

Our first result corresponds to the statement about $w^{(0)}$ in \Cref{thm:n_iterate}.

\begin{prop}\label{thm:supnorm_w0} 
For any $s\geq 0$ and any $0<\theta\ll 1$, we have that 
\[
 \sup_{0\leq t\leq 1} \sup_{k\in\Z_L^d} |w_k^{(0)}(t)| \leq L^{\theta}
 \]
$L$-certainly.
\end{prop}
\begin{proof}
We write $w^{(0)} = c^{(0)} + b^{(0)}$, where
\begin{equation}\label{eq:decomposition_w0}
\begin{split}
c^{(0)}(t,x) & = L^{-d/2} \sum_{k\in\Z_L^d} c_k \eta_k\, e^{2\pi ik\cdot x-\vartheta \gamma_k t},\\
 b^{(0)}(t,x) & =L^{-d/2} \sum_{k\in\Z_L^d} b^{(0)}_k (t) \, e^{2\pi i k\cdot x}, \quad \mbox{for}\qquad 
 b^{(0)}_k (t)   = \sqrt{\vartheta} \, b_k\, \int_0^{t} e^{-\vartheta \gamma_k (t-t')}\, d\beta_k (t').
 \end{split}
\end{equation}
We first study $c^{(0)}$. We claim that it suffices to show that, for $c_k\neq 0$, 
\begin{equation}\label{eq:ptw_tail_bound_c}
\P \left( \sup_{0\leq t\leq 1} |c_k^{(0)} (t)| > L^{\theta} \right) \leq C\, \exp\left(-\frac{1}{2\, c_k^2}\, L^{2\theta}\right).
\end{equation}
Indeed, assuming \eqref{eq:ptw_tail_bound_c}, and using the fact that $c_k^2\leq C\, \langle k\rangle^{-2}$, we have that
\[
\begin{split}
\P\left( \sup_{0\leq t\leq 1} \sup_{k\in\Z_L^d} |c_k^{(0)}(t)| > L^{\theta} \right) &  \leq \sum_{k\in\Z_L^d} \P\left( \sup_{0\leq t\leq 1} |c_k^{(0)}(t)| > L^{\theta} \right) \lesssim  \sum_{k\in\Z_L^d} e^{-C\, \langle k\rangle^2\, L^{2\theta}} \\
& \lesssim e^{-L^{2\theta} }  \sum_{k\in\Z_L^d} e^{-C\, \langle k\rangle^2} \lesssim e^{-L^{2\theta} }  \sum_{k\in\Z_L^d}  \langle k\rangle^{-2} \lesssim L^d \, e^{-L^{2\theta} } \lesssim e^{-L^{\theta} } ,
\end{split}
\]
where we allow the constant $C$ to change from line to line. 

Let us thus focus on proving \eqref{eq:ptw_tail_bound_c}. To do so, we note that $\sup_{0\leq t\leq 1} |c_k^{(0)}(t)| \leq c_k |\eta_k|$ and therefore \eqref{eq:ptw_tail_bound_c} follows from an elementary tail bound for a Gaussian random variable.

\medskip

Next consider $b^{(0)}$. By the same reasoning as before, we claim that it suffices to show that, for $b_k\neq 0$, 
\begin{equation}\label{eq:ptw_tail_bound}
\P \left( \sup_{0\leq t\leq 1} |b_k^{(0)} (t)| > L^{\theta} \right) \leq C\, \exp\left(-\frac{\gamma_k}{2\, b_k^2}\, L^{2\theta}\right).
\end{equation}
for some $C$ independent of $k$. In order to prove this, we first note that  we may write
$b_k^{(0)} (t)= \sqrt{\vartheta} \, b_k \, V_k(t)$, where $V_k(t)$ is an Ornstein-Uhlenbeck process satisfying
\[
dV_k(t) = - \vartheta \gamma_k V_k(t) + d\beta_k (t).
\]
By \cite[Theorem~2.5]{GravPes}, we have that 
\[
\E \left[ \sup_{0\leq t\leq 1} |V_k(t)| \right]\leq \frac{C}{\sqrt{\vartheta\, \gamma_k}} \, \log ( 1+ \vartheta \gamma_k),
\]
for a universal constant $C$. Since $\vartheta\leq L^d$, this yields
\[
\E \left[\sup_{0\leq t\leq 1} |b_k^{(0)} (t)| \right]\leq \frac{C\, b_k}{\sqrt{\gamma_k}} \, [\log(1+\gamma_k)+\log ( 1+ \vartheta) ] \leq L^{\theta} b_k .
\]
The Borell-TIS inequality \cite[Theorem 2.1.1]{AdTay} 
then implies
\[
\begin{split}
\P \left( \sup_{0\leq t\leq 1} |b_k^{(0)} (t)| > L^{\theta}(1+b_k) \right) & \leq 
\P \left( \sup_{0\leq t\leq 1} |b_k^{(0)} (t)| > L^{\theta} + \E \left[\sup_{0\leq t\leq 1} |b_k^{(0)} (t)| \right] \right) \\
& \leq \exp \left( - \frac{\gamma_k}{2b_k^2 (1-e^{-2\vartheta\gamma_k})} \, L^{2\theta}\right)\leq \exp \left( - \frac{\gamma_k}{2b_k^2} \, L^{2\theta}\right),
\end{split}
\]
as desired.
\end{proof}

Despite the fact that $w^{(0)}$ is relatively rough, we would like to place higher iterates in a better space such as $H^{b}([0,1],H^s(\T^d_L))$ for $b>1/2$, $s>d/2$. 

We begin by computing the Fourier transform of $w^{(0)}$. In order to do so, it is convenient to fix a smooth cutoff $\varphi\in C_c^{\infty}(\R)$ with $\mbox{supp}(\varphi)\subset (-2,2)$ such that $\varphi(t)=1$ if $t\in [-1,1]$. Next we redefine 
\[
w^{(0)}(t,x):= L^{-d/2}\, \sum_{k\in\Z^d_L} e^{2\pi ik\cdot x} w_k^{(0)}(t)
\]
where
\begin{equation}\label{eq:cutoff_w0}
w_k^{(0)}(t)  =   c_k^{(0)}(t) + b_k^{(0)}(t) := c_k \eta_k e^{-\vartheta \gamma_k t}\varphi (t) + \sqrt{\vartheta}\, b_k \, \varphi (t) \, \mathbbm{1}_{t\geq 0} \,\int_0^t e^{-\vartheta  \gamma_k (t-t')}\, d\beta_k (t').
\end{equation}
Given that $\varphi=1$ in the interval $[0,1]$, we can focus on studying this function as long as we restrict ourselves to such times. Moreover, \Cref{thm:supnorm_w0}  and the continuity of $\varphi$ guarantee that \eqref{eq:cutoff_w0} is continuous almost surely.

We compute the temporal Fourier transform of $w_k^{(0)}$. First of all,
\begin{equation}\label{eq:Fourier_Laplace}
\widetilde{c_k}^{(0)}(\tau) = c_k \eta_k \, \widetilde{\varphi}(\tau-i\vartheta\gamma_k),
\end{equation}
where $\widetilde{\varphi}$ is the Fourier-Laplace transform of $\varphi$. Similarly, 
\begin{equation}\label{eq:representation_w0_1}
\begin{split}
 \widetilde{b_k}^{(0)}(\tau) &= \sqrt{\vartheta} \, b_k\, \int_{0}^{\infty} e^{-it\tau} \varphi(t) \int_0^{t} e^{-\vartheta \gamma_k (t-t')} \, d\beta_k(t') \, dt \\
 & = \sqrt{\vartheta} \, b_k\, \int_{0}^{\infty} e^{\vartheta \gamma_k t'} \int_{t'}^{\infty} \varphi(t)\, e^{-(\vartheta\gamma_k+i\tau) t} \, dt \, d\beta_k(t')
 =:\sqrt{\vartheta} \, b_k\, Y_k(\tau).
 \end{split}
\end{equation}
For fixed $k$ and $\tau$, it is easy to show that $Y_k(\tau)$ is a normal random variable with mean zero and variance $\E |Y_k(\tau)|^2\lesssim 1$.

On the other hand, we may integrate by parts to obtain some decay in $\tau$:
\begin{equation}\label{eq:representation_Xk}
\begin{split}
 \widetilde{b_k}^{(0)}(\tau) & = \sqrt{\vartheta} \, b_k\, \int_{0}^{\infty} e^{-it\tau} \varphi(t) \int_0^{t} e^{-\vartheta \gamma_k (t-t')} \, d\beta_k(t') \, dt = \sqrt{\vartheta} \, b_k\, \int_{0}^{\infty} e^{\vartheta \gamma_k t'} \int_{t'}^{\infty} \varphi(t)\, e^{-(\vartheta\gamma_k+i\tau) t} \, dt \, d\beta_k(t')\\
& = \frac{ \sqrt{\vartheta} \,b_k}{\vartheta\gamma_k+i\tau}\, \int_{0}^{2} \left( e^{-i\tau t'} \varphi (t') + e^{\vartheta \gamma_k t'}\int_{t'}^{\infty}e^{-(\vartheta \gamma_k+i\tau) t} \varphi'(t)\,dt\right) \, d\beta_k (t') = \frac{ \sqrt{\vartheta} \,b_k}{\vartheta \gamma_k+i\tau}\, X_k(\tau) .
\end{split}
\end{equation}
For fixed $k$ and $\tau$, $X_k(\tau)$ is a Gaussian random variable with mean zero and variance:
\begin{equation}
\E | X_k(\tau)|^2 = 2\, \int_{0}^{2} \left | e^{-i\tau t'} \varphi (t') + e^{\vartheta\gamma_k t'}\int_{t'}^{\infty}e^{-(\vartheta \gamma_k+i\tau) t} \varphi'(t)\,dt \right |^2 \, dt'\lesssim 1.
\end{equation}

When $\vartheta<1$, the best possible bounds are:
\begin{equation}\label{eq:bdecay1}
\E | \widetilde{b_k}^{(0)}(\tau) |^2 \lesssim 
\begin{cases}
\vartheta \, b_k^2 & \mbox{if}\ \tau < 1\\
\vartheta\, |\tau|^{-2}\, b_k^2 & \mbox{if}\ \tau > 1
\end{cases} \lesssim \vartheta \, b_k^2 \,\langle \tau\rangle^{-2}.
\end{equation}
Note that $b^{(0)}$ cannot live in $H^b H^s$ for $b>1/2$ due to the decay $\langle \tau\rangle^{-2}$, which captures the roughness of the Wiener process.

When $\vartheta>1$, the bound coming from integration by parts is best and we have that:
\begin{equation}\label{eq:bdecay2}
\E | \widetilde{b_k}^{(0)}(\tau) |^2 \lesssim \vartheta^{-1} \, b_k^2\, \langle \tau / \vartheta\rangle^{-2} .
\end{equation}

Despite the fact that the variance of $\widetilde{b_k}^{(0)}(\tau)$ only decays like $|\tau|^{-2}$, we can show that covariances have additional decay. That is the content of the next result.

\begin{lem}\label{thm:covariance_w0}
Let $k_1,k_2\in \Z^d_L$ and $\tau_1, \tau_2\in \R$. Then 
\begin{enumerate}[label=(\roman*)]
\item If $k_1\neq k_2$, or if $k_1=k_2$ and $\imath=+$, we have that
\begin{equation}\label{eq:covariance_w01}
\E \left[ \widetilde{w_{k_1}}^{(0)}(\tau_1)\, \widetilde{w_{k_2}}^{(0)}(\tau_2)^{\imath} \right] =
0,
\end{equation}
where $\imath\in \{-,+\}$ indicates conjugation. 
\item If $\vartheta >2$,  we have that
\begin{equation}\label{eq:covariance_w02}
\begin{split}
\Big |\E \left[ \widetilde{w_k}^{(0)}(\tau_1) \, \overline{\widetilde{w_k}^{(0)}(\tau_2)}\, \right] \Big |\lesssim_N &\ \vartheta^{-1}\,b_k^2\, \langle \tau_1 -\tau_2\rangle^{-N} \, \langle \tau_1 / \vartheta\rangle^{-1} \, \langle \tau_2 /\vartheta\rangle^{-1}  + c_k^2\, \langle \tau_1\rangle^{-N}\, \langle \tau_2\rangle^{-N} 
\end{split}
\end{equation}
for any $k\in \Z^d_L$ and any $N\in\N$. 
\item If $\vartheta <2$, we have that 
\begin{equation}\label{eq:covariance_w03}
\begin{split}
\Big |\E \left[ \widetilde{w_k}^{(0)}(\tau_1) \, \overline{\widetilde{w_k}^{(0)}(\tau_2)}\, \right] \Big |\lesssim_N &\ \vartheta\, b_k^2 \,\langle \tau_1 -\tau_2\rangle^{-N} \, \left\langle \tau_1\right\rangle^{-1} \, \left\langle \tau_2\right\rangle^{-1}  + c_k^2\, \langle \tau_1\rangle^{-N}\, \langle \tau_2\rangle^{-N} 
\end{split}
\end{equation}
for any $k\in \Z^d_L$ and any $N\in\N$.
\end{enumerate}
\end{lem}
\begin{proof}
\textbf{(i)}  Writing $\widetilde{w_{k}}^{(0)}=\widetilde{c_k}^{(0)}+\widetilde{b_k}^{(0)}$ we have three types of terms in \eqref{eq:covariance_w01}. Note that 
\[
\E[\widetilde{c_{k_1}}^{(0)}(\tau_1)\widetilde{c_{k_2}}^{(0)}(\tau_2)^{\imath}]=0
\]
 thanks to the definition of the $\eta_k$. The fact that $\E[\widetilde{b_{k_1}}^{(0)}(\tau_1)\widetilde{b_{k_2}}^{(0)}(\tau_2)^{\imath}]=0$ follows from \\
 $\E [d\beta_{k_1} d\beta_{k_2}^{\imath}]=0$. Finally, $\E[\widetilde{c_{k_1}}^{(0)}(\tau_1)\widetilde{b_{k_2}}^{(0)}(\tau_2)^{\imath}]=0$ follows from the fact that $\eta_{k_1}$ is independent of $\widetilde{b_{k_2}}^{(0)}$ for all times and all $k_1,k_2$.

Next we account for the second summand in \eqref{eq:covariance_w02} and \eqref{eq:covariance_w03}.
The fact that $\widetilde{c_{k}}^{(0)}$ is independent from $\widetilde{b_{k}}^{(0)}$ implies that we need only estimate 
\[
\E [\widetilde{c_{k}}^{(0)}(\tau_1)\, \overline{\widetilde{c_{k}}^{(0)}(\tau_2)}],\qquad  \mbox{and}\qquad  
\E [\widetilde{b_{k}}^{(0)}(\tau_1)\, \overline{\widetilde{b_{k}}^{(0)}(\tau_2)}].
\]

Since $\varphi\in C_c^{\infty}(\R)$, \eqref{eq:Fourier_Laplace} implies that 
\[
|\E [\widetilde{c_{k}}^{(0)}(\tau_1)\, \overline{\widetilde{c_{k}}^{(0)}(\tau_2)}]|= c_k^2\, |\widetilde{\varphi}(\tau_1-i\vartheta\gamma_k)|\,|\widetilde{\varphi}(\tau_2-i\vartheta\gamma_k)|\lesssim  c_k^2\,\langle \tau_1\rangle^{-N}\, \langle \tau_2\rangle^{-N} .
\]

\

The rest of the proof consists of estimating $\E [\widetilde{b_{k}}^{(0)}(\tau_1)\, \overline{\widetilde{b_{k}}^{(0)}(\tau_2)}]$. 

\

\textbf{(ii)}  Suppose that $\vartheta\gtrsim 1$. Using \eqref{eq:representation_Xk}, we write:
\begin{equation}
\E \left[ \widetilde{b_k}^{(0)}(\tau_1) \, \overline{\widetilde{b_k}^{(0)}(\tau_2)}\, \right]  = \frac{\vartheta b_k^2}{(\vartheta\gamma_k +i \tau_1)\, (\vartheta \gamma_k -i \tau_2)} \, \E[ X_k(\tau_1) \overline{X_k(\tau_2)}].
\end{equation}
Given that $\E [d\beta_{k}(\tau) \overline{d\beta_{k}}(\tau')]=2\delta(\tau-\tau')\, d\tau\, d\tau'$, we have that
\begin{equation}\label{eq:covariance_X}
\frac{1}{2}\, \E[ X_k(\tau_1) \overline{X_k(\tau_2)}]  = \int_{0}^{2} \Psi (\tau_1,s) \, \overline{\Psi (\tau_2,s)} \, ds
\end{equation}
where 
\begin{align}
\Psi (\tau,s) 
& = e^{-i\tau s} \varphi (s) + e^{-i \tau s}\, \int_{0}^{\infty} e^{-(\vartheta \gamma_k+i\tau)s'} \varphi'(s+s')\,ds'.\nonumber
\end{align}
It is easy to show that 
\begin{align*}
 \Psi (\tau_1,s) \, \overline{\Psi (\tau_2,s)}  = & \ e^{i(\tau_2-\tau_1)s}\, \varphi(s)^2 + e^{i(\tau_2-\tau_1)s}\, \varphi (s) \, \int_{0}^{\infty} e^{-(\vartheta\gamma_k-i\tau_2) s'} \varphi'(s+s')\,ds' \\
 & + e^{i(\tau_2-\tau_1)s}\, \varphi (s) \, \int_{0}^{\infty} e^{-(\vartheta\gamma_k+i\tau_1) s'} \varphi'(s+s')\,ds'\\
 & +   e^{i(\tau_2-\tau_1)s}\,\left( \int_{0}^{\infty} e^{-(\vartheta\gamma_k-i\tau_2) s'} \varphi'(s+s')\,ds' \right) \, \left(\int_{0}^{\infty} e^{-(\vartheta \gamma_k+i\tau_1) s'} \varphi'(s+s')\,ds'\right)\\
 =: &\  e^{i(\tau_2-\tau_1)s}\, F(\tau_1,\tau_2; s).
\end{align*}
In the last step, we factor out $e^{i(\tau_2-\tau_1)s}$ and notice that the remaining function is supported in $s \in (0,2)$, it is smooth, and that for any $N\in\N$
\[ |\pa_{s}^{N} F(\tau_1,\tau_2; s)|\lesssim_N 1\]
independently of $\tau_1, \tau_2$. Plugging this into \eqref{eq:covariance_X}, and integrating by parts $N$ times yields:
\[ \int_{0}^{2} e^{i(\tau_2-\tau_1)s}\, F(\tau_1,\tau_2; s) \, ds  = \O_N  \left( \langle \tau_1-\tau_2\rangle^{-N}\right).\]

\textbf{(iii)}  Suppose that $\vartheta<2$. If $|\tau_1|,|\tau_2|\leq 2$ we exploit \eqref{eq:representation_w0_1} to write:
\[
\E \left[ \widetilde{b_k}^{(0)}(\tau_1) \, \overline{\widetilde{b_k}^{(0)}(\tau_2)}\, \right]  = \vartheta \, b_k^2 \, \E [Y_k(\tau_1) \overline{Y_k(\tau_2)}]\lesssim \vartheta \, b_k^2 \sim \vartheta \, b_k^2\, \langle \tau_1\rangle^{-1} \,\langle \tau_2\rangle^{-1} \, \langle \tau_1-\tau_2\rangle^{-N} .
\]
If $|\tau_1|,|\tau_2|>2$ we proceed as in Step 2 to gain a factor of $\langle \tau_1-\tau_2\rangle^{-N}$, and we note that
\[
\vartheta^{-1}\,\langle\tau_1 / \vartheta\rangle^{-1} \, \langle\tau_2 / \vartheta\rangle^{-1} \sim \vartheta \, \langle \tau_1\rangle^{-1} \,\langle \tau_2\rangle^{-1} .
\]
Finally, suppose that $|\tau_1|\leq 2$ and $|\tau_2|>2$. We write
\[
\E \left[ \widetilde{b_k}^{(0)}(\tau_1) \, \overline{\widetilde{b_k}^{(0)}(\tau_2)}\, \right] \leq b_k^2\, \langle \tau_2/\vartheta\rangle^{-1}\, \E [Y_k(\tau_1) \overline{X_k(\tau_2)}].
\]
Then we can argue as in Step 2 to show that 
\[
\E [Y_k(\tau_1) \overline{X_k(\tau_2)}] = \int_0^2 e^{-i(\tau_1-\tau_2) t} G(\tau_1,\tau_2;t)\, dt
\]
where $|\pa^N_{t} G(\tau_1,\tau_2;t)|\lesssim_N 1$ uniformly in $\tau_1,\tau_2$. Integration by parts and the fact that $|\tau_1|\leq 2$ and $|\tau_2|>2$ yields
\[
\E \left[ \widetilde{b_k}^{(0)}(\tau_1) \, \overline{\widetilde{b_k}^{(0)}(\tau_2)}\, \right] \lesssim_N \vartheta \, b_k^2\, \langle \tau_2\rangle^{-1}\, \langle \tau_1 -\tau_2\rangle^{-N} \sim \vartheta \, b_k^2\, \langle \tau_2\rangle^{-1}\, \langle \tau_1 -\tau_2\rangle^{-N} \, \langle \tau_1\rangle^{-1}.
\]
This concludes the proof of \eqref{eq:covariance_w03}.
\end{proof}

\subsection{Higher iterates}\label{sec:higher_iterates}

We now move on to higher iterates. Despite $w^{(0)}$ being somewhat rough in time, we want to show that higher iterates live in $H^b(\R, H^s(\T^d_L))$ $L$-certainly for any $s\geq 0$ and $b>1/2$. Recall that this space embeds into $C([0,1],H^s(\T^d_L))$.

The first step is to write the Fourier transform of $w^{(1)}$ in terms of $w^{(0)}$. Recall that
\begin{equation}
 w_k^{(1)} (t) = \mathcal{I}\mathcal{W} (w^{(0)},w^{(0)},w^{(0)})_k(t)
\end{equation}
where 
\[ (\mathcal{I}F)_k(t)= \varphi(t) \, \int_{0}^{t} e^{-\vartheta \gamma_k (t-t')} F(t')\, \varphi(t')\,dt' .\]

Our first result allows us to understand the temporal Fourier transform of this operator.

\begin{lem}\label{thm:FT_integration} We have that 
\[
(\widetilde{\mathcal{I}F})_k (\tau ) = \int_{\R} [ I_{0,k} (\tau,\sigma) + I_{1,k} (\tau,\sigma)] \, \widetilde{F}(\sigma)\, d\sigma
\]
where 
\[ 
\begin{split}
| I_{0,k}(\tau,\sigma)| & \lesssim_N \langle \tau+i \vartheta \gamma_k\rangle^{-N} \,\langle \sigma\rangle^{-1}\lesssim_N \langle\tau\rangle^{-N}\,\langle \sigma\rangle^{-1}\\
| I_{1,k}(\tau,\sigma)| & \lesssim_N \langle (\tau-\sigma)+i \vartheta \gamma_k\rangle^{-N} \,\langle \sigma\rangle^{-1}\lesssim_N \langle \tau-\sigma\rangle^{-N}\,\langle \sigma\rangle^{-1}
\end{split}
\]
for any $N\in \N$.
\end{lem}
\begin{proof}
We first write 
\begin{equation}\label{eq:sign_trick}
\int_{0}^{t} f(t')\, dt' = \frac{1}{2} \, (f\ast \mbox{sign} )(t) + \frac{1}{2} \int_{\R} f(t')\, \mbox{sign} (t') \, dt'.
\end{equation}
We multiply this equality by $e^{-\vartheta \gamma_k t}\, \varphi(t)$ and take the temporal Fourier transform. Next we substitute $f(t'):=e^{\vartheta \gamma_k t'}\,\varphi(t')\, F(t')$. The second term in \eqref{eq:sign_trick} becomes
\[
\frac{1}{2}  \int_{\R} f(t')\, \mbox{sign} (t') \, dt' \, \int_{\R} e^{-(\vartheta \gamma_k+i\tau) t} \varphi(t)\, dt
\]
which has the desired decay $\langle \vartheta \gamma_k+i\tau\rangle^{-N}$. Next note that the Plancherel theorem implies
\[
\int_{\R} f(t')\, \mbox{sign} (t') \, dt' = \int_{\R} e^{\vartheta \gamma_k t'} \varphi(t') \,  \mbox{sign} (t') \, F(t')\, dt'  = \int_{\R} \left(\int_{\R} \frac{1}{\sigma-s} \int_{\R} e^{-(is-\vartheta \gamma_k)s'} \varphi(s')\, ds' ds\right)\widetilde{F}(\sigma)\, d\sigma,
\]
where we interpret $(\sigma-s)^{-1}$ in the principal value sense. The integrals in parenthesis yield the decay $\langle \sigma\rangle^{-1}$. This completes the bound on $I_{0,k}(\tau,\sigma)$.

The term $I_{1,k}(\tau,\sigma)$ corresponds to the first term in \eqref{eq:sign_trick}. It admits a similar analysis, so we omit the details.
\end{proof}

Using \Cref{thm:FT_integration}, we may write the temporal Fourier transform of the first iterate as:
 \begin{equation}\label{eq:FT_first_iterate}
 \widetilde{w_k}^{(1)} (\tau) =  \int_{\R} [ I_{0,k} (\tau,\sigma) + I_{1,k} (\tau,\sigma)] \, \widetilde{F}(\sigma)\, d\sigma
\end{equation}
where
\[
\widetilde{F}(\sigma) = \frac{i\lambda T}{L^{d}}\, \sum_{k_1-k_2+k_3=k} \epsilon_{k_1,k_2,k_3}\, \int_{\sigma=\tau_1-\tau_2+\tau_3+T\Omega_{k2}^{13}}\  \prod_{j=1}^{3} \, \widetilde{w_{k_j}}^{(0)}(\tau_j)^{\imath_j}\, d\tau_j.
\]

The next step is to derive similar formulae for higher iterates $w^{(n)}$, $n\geq 1$. First of all, we introduce some notation to keep track of the indices over which we will sum. 

\begin{defn}
A \emph{ternary tree} $\mathcal{T}$ is a rooted tree where each non-leaf (or branching) node has exactly three children nodes.
Let  $\mathcal{L}$ the set of leaves of such a tree and $\mathcal{N}=\mathcal{T}-\mathcal{L}$ the set of branching nodes, with the convention that $\mathcal{N}=\emptyset$ when the tree is only a root $\mathcal{T}=\{\sub{r}\}$. The order of the tree, denoted by $n(\mathcal{T})$, is the number of branching nodes, i.e. $|\mathcal{N}|$. If $n(\mathcal{T})=n$, then 
\[ |\mathcal{N}|=n,\quad  |\mathcal{L}|=2n+1,\quad \mbox{and}\quad |\mathcal{T}|=3n+1.\]

For $\sub{l}\in\Lc$, we introduce the signs $\imath_{\sub{l}}\in\{-1,1\}$ to keep track of which terms are conjugated.
\end{defn}

\begin{defn}\label{def:decoration} A \emph{decoration} $\mathscr{D}$ of a tree $\mathcal{T}$ is a set of vectors $(k_{\sub{n}})_{n\in\mathcal{T}} \subset  \Z_L^d$ such that for each branching node $\sub{n}\in \mathcal{N}$ we have that 
 \begin{equation}\label{eq:tree1}
 k_{\sub{n}}=k_{\sub{n}_1}- k_{\sub{n}_2}+k_{\sub{n}_3}
 \end{equation}
where $\sub{n}_j$ are its children nodes. 
Given a decoration $\mathscr{D}$, we define
\begin{equation}\label{eq:epsilon_D}
\epsilon_{\mathscr{D}} := \prod_{\sub{n}\in\Nc} \epsilon_{k_{\sub{n}_1}, k_{\sub{n}_2}, k_{\sub{n}_3}} 
\end{equation}
following \eqref{eq:epsilon}, as well as
\begin{equation}\label{eq:tree2}
 \Omega_{\sub{n}}=\Omega (k_{\sub{n}_1},k_{\sub{n}_2},k_{\sub{n}_3},k_{\sub{n}}) = |k_{\sub{n}_1}|_{\zeta}^2 - |k_{\sub{n}_2}|_{\zeta}^2+|k_{\sub{n}_3}|_{\zeta}^2-|k_{\sub{n}}|_{\zeta}^2 \ .
 \end{equation}

Finally, we choose $d_{\sub{n}}\in \{0,1\}$ for each $\sub{n}\in \mathcal{T}$, and we define
\begin{equation}\label{eq:convolution}
 \tau_{\sub{n}}= d_{\sub{n}_1}\tau_{\sub{n}_1}- d_{\sub{n}_2}\tau_{\sub{n}_2}+d_{\sub{n}_3}\tau_{\sub{n}_3}+ T\Omega_{\sub{n}}
\end{equation}
for each branching node $\sub{n}\in\mathcal{N}$, where $\sub{n}_j$ are its children.
\end{defn}

Note that $(\tau_{\sub{n}})_{\sub{n}\in\Tc}$ is fully determined by $(\tau_{\sub{l}})_{\sub{l}\in\Lc}$. We will often use the following notation:
\[
\tau [\Lc] := (\tau_{\sub{l}})_{\sub{l}\in\Lc}, \qquad \Omega [\Nc]:= ( \Omega_{\sub{n}})_{\sub{n}\in\Nc}.
\]

We are ready to derive a formula for the Fourier transform of the higher iterates:

\begin{prop}\label{thm:FT_higher_iterates} Consider $n\geq 1$, then we have that the $n$-th iterate is given by
\[
\widetilde{w_k}^{(n)}(\tau ) = \sum_{n(\Tc)=n} \widetilde{\Jc_{\Tc}}(\tau)
\]
where we sum over all ternary trees $\Tc$ of order $n$ satisfying \eqref{eq:tree1}-\eqref{eq:convolution}, and 
\begin{equation}\label{eq:FT_higher_iterates}
 \widetilde{\Jc_{\Tc}}(\tau) = \left(\frac{i\lambda T}{L^{d}}\right)^n \, \sum_{\mathscr{D}}  \epsilon_{\mathscr{D}} \, \int_{\R^{2n+1}} \Kc (\tau, \mathscr{D},\tau [\Lc]) \, \prod_{\sub{l}\in\Lc} \widetilde{w_{k_{\sub{l}}}}^{(0)} (\tau_{\sub{l}}) \, d\tau_{\sub{l}}
\end{equation}
where we sum over all possible decorations $\mathscr{D}$ with $k_{\sub{r}}=k$. Moreover,
\begin{equation}\label{eq:FTbound_higher_iterates}
| \Kc (\tau, \mathscr{D},\tau [\Lc]) |\lesssim_N \sum_{(d_{\sub{n}}:\sub{n}\in\Nc)} \langle \tau-d_{\sub{r}} \tau_{\sub{r}}\rangle^{-N} \, \prod_{\sub{n}\in\Nc} \langle \tau_{\sub{n}}\rangle^{-1},
\end{equation}
for any $N\in\N$.
\end{prop}
\begin{proof}
We prove the proposition by induction. The case $n=1$ follows from \Cref{thm:FT_integration}. Next suppose that $n\geq 2$ and that the statement is true for all $\widetilde{w_k}^{(m)}$ with $1\leq m\leq n-1$. We can write $w_k^{(n)}$ as a sum of terms such as $\Ic \Wc (w^{(n_1)},w^{(n_2)},w^{(n_3)})_k$ where $n_1+n_2+n_3+1=n$, so it is enough to prove the statement for one such term.

Using the induction hypothesis each $w^{(n_j)}$ can be written as a sum over trees of order $n_j$, therefore it is enough to prove our result for $\Ic \Wc (\Jc_{\Tc_1},\Jc_{\Tc_2},\Jc_{\Tc_3})_k$ where $n(\Tc_j)=n_j$. These three trees uniquely determine a tree $\Tc$ of order $n$ where we attach $\Tc_j$ to its root, thus we denote this term by $\Jc_{\Tc}$.

By \Cref{thm:FT_integration} we have that the Fourier transform of $\Jc_{\Tc}$ can be written as
\[
\widetilde{\Jc_{\Tc}}(\tau)=\frac{i\lambda T}{L^{d}} \, \sum_{d\in\{0,1\}} \sum_{k=k_1-k_2+k_3} \epsilon_{k_1,k_2,k_3}\, \int_{\R^3} I_{d,k}(\tau,\tau_1- \tau_2+\tau_3 + T\Omega_{k2}^{13}) \, \prod_{j=1}^{3} \widetilde{\Jc_{\Tc_j}}(\tau_j)\, d\tau_j.
\]
Next we use the induction hypothesis on each $\widetilde{\Jc_{\Tc_j}}(\tau_j)$ to write them as in \eqref{eq:FT_higher_iterates}.
Setting $\mathscr{D}_j := (k_{\sub{n}})_{ \sub{n}\in \Tc_j}$, we find that 
\begin{multline*}
\widetilde{\Jc_{\Tc}}(\tau)=\left(\frac{i\lambda T}{L^{d}}\right)^{n}\, \sum_{d\in\{0,1\}}  \sum_{k=k_1-k_2+k_3} \epsilon_{k_1,k_2,k_3}\, \int_{\R^{2n+4}} I_{d,k}(\tau,\tau_1- \tau_2+\tau_3 + T\Omega_{k2}^{13}) \\ 
\prod_{j=1}^{3} \sum_{\mathscr{D}_j}  \epsilon_{\mathscr{D}_j}\, \Kc (\tau_j, \mathscr{D}_j, \tau[\Lc_j])\, \prod_{\sub{l}\in\Lc} \widetilde{w_{k_{\sub{l}}}}^{(0)} (\tau_{\sub{l}}) \, d\tau_{\sub{l}}  d\tau_j.
\end{multline*}
Let\footnote{Technically we would also need to sum over all $n_1,n_2,n_3$ such that $n_1+n_2+n_3+1=n$, but we omit this to keep the notation as simple as possible.}
\[
 \Kc (\tau, \mathscr{D},\tau [\Lc]) := \sum_{d\in\{0,1\}} \int_{\R^3} I_{d,k}(\tau,\tau_1- \tau_2+\tau_3 + T\Omega_{k2}^{13}) \, \prod_{j=1}^{3} \Kc (\tau_j, \mathscr{D}_j, \tau[\Lc_j])\, d\tau_j.
\]
Using \Cref{thm:FT_integration} and the induction hypothesis, we have that 
\[
\begin{split}
| \Kc (\tau, \mathscr{D},\tau [\Lc]) | \lesssim \sum_{d\in\{0,1\}} \int_{\R^3}  & \langle \tau-d (\tau_1- \tau_2+\tau_3 + T\Omega_{k2}^{13})\rangle^{-N} \langle \tau_1- \tau_2+\tau_3 + T\Omega_{k2}^{13}\rangle^{-1} \\
&\prod_{j=1}^{3} \left(\sum_{(d_{\sub{n}}:\sub{n}\in\Nc_j)} \langle \tau_j-d_{\sub{r}_j} \tau_{\sub{r}_j}\rangle^{-N}\prod_{\sub{n}\in\Nc_j}  \langle \tau_{\sub{n}}\rangle^{-1}\right)\, d\tau_j \lesssim \sum_{(d_{\sub{n}}:\sub{n}\in\Nc)} \langle \tau-d_{\sub{r}} \tau_{\sub{r}}\rangle^{-N}\,\prod_{\sub{n}\in\Nc} \langle \tau_{\sub{n}}\rangle^{-1}
\end{split}
\]
where $d_{\sub{r}}=d$ and $\tau_{\sub{r}}$ is given by 
\[
\tau_{\sub{r}}= d_{\sub{r}_1}\tau_{\sub{r}_1}- d_{\sub{r}_2}\tau_{\sub{r}_2}+d_{\sub{r}_3}\tau_{\sub{r}_3}+ T\Omega_{k2}^{13}.
\]
Note that the roots of the three trees $\Tc_j$, $\sub{r}_j$, are the children of the main root $\sub{r}$ of $\Tc$. This concludes the proof.
\end{proof}

We are ready to prove \Cref{thm:n_iterate}.

\begin{proof}[Proof of \Cref{thm:n_iterate}]
For simplicity, let us fix $(d_{\sub{n}} :\sub{n}\in\Tc)$ (since there is a finite number of options) and derive a uniform bound over such choices. By \Cref{thm:FT_higher_iterates}, let us prove the desired bound \eqref{eq:n_iterate} for a unique $\Jc_{\Tc}$ since a finite amount of such terms add up to $w^{(n)}$. 

 We start by computing the expectation of the square of \eqref{eq:FT_higher_iterates}. Letting $\mathscr{D}':= (k_{\sub{n}}')_{\sub{n}\in\Tc}$, we have that 
\begin{equation}\label{eq:variance_n}
\E |\widetilde{\Jc_{\Tc}}(\tau)|^2 = \left(\frac{\lambda T}{L^{d}}\right)^{2n} \int_{(\R^{2n+1})^2}  \sum_{\mathscr{D}, \mathscr{D}' } \epsilon_{\mathscr{D}}\, \epsilon_{\mathscr{D}'} \Kc (\tau, \mathscr{D}, \tau[\Lc])\, \overline{\Kc (\tau, \mathscr{D}', \tau'[\Lc])} \, \E\left[ \prod_{\sub{l}\in\Lc} \widetilde{w_{k_{\sub{l}}}}^{(0)} (\tau_{\sub{l}}) \, \overline{\widetilde{w_{k_{\sub{l}}'}}^{(0)} (\tau_{\sub{l}}')} \right]\, d\tau_{\sub{l}}\, d\tau_{\sub{l}}' .
\end{equation}
We can estimate such correlations using \Cref{thm:Iserlis} and \Cref{thm:covariance_w0}. We said that two leaves $\sub{l}_1$ and $\sub{l}_2$ are \emph{paired} if 
\[ 
k_{\sub{l}_1} = k_{\sub{l}_2}\quad \mbox{and} \quad \imath_{\sub{l}_1}+\imath_{\sub{l}_2}=0, 
\] 
or if 
\[ 
k_{\sub{l}_1} = k_{\sub{l}_2}'\quad \mbox{and} \quad \imath_{\sub{l}_1}-\imath_{\sub{l}_2}=0.
\] 
Let $\mathcal{P}$ be the set of such unordered pairs $\{\sub{l}_1,\sub{l}_2\}$, where $\sub{l}_1, \sub{l}_2 \in \mathcal{L}$. Note that this is equivalent to making two copies of the tree $\Tc$, which we will denote by $\widetilde{\Tc}$, and pairing all $4n+2$ leaves. As such, we will use $\tau_{\sub{n}}$ and $\tau_{\sub{n}}'$ indistinctly for those $\sub{n}$ in the copy of the tree whenever there is no ambiguity. Let us also denote by $\widetilde{\mathcal{N}}$ the set of branching nodes of $\widetilde{\mathcal{T}}$ and by $\widetilde{\mathcal{L}}$ the set of all $4n+2$ leaves of these trees. 

Some of the pairs of leaves in $\mathcal{P}$ will come from within a tree; let the number of such pairs be $p$. The rest of the pairings will involve a leaf from each tree; there are $2n+1-p$ such pairs. From now on we will fix such a pairing $\mathcal{P}$ since there is a finite number of them (depending on $n$), so this will only contribute in the form of a constant to our final bound. 

Finally, it will be convient to introduce the quantities $q_{\sub{n}}$, which are inductively defined according to
\begin{equation}\label{eq:tree3}
q_{\sub{n}} =0\ \mbox{if}\ \sub{n}\in\mathcal{L}; \quad q_{\sub{n}} = d_{\sub{n}_1} q_{\sub{n}_1}- d_{\sub{n}_2} q_{\sub{n}_2}+d_{\sub{n}_3} q_{\sub{n}_3}+ \Omega_{\sub{n}},
\end{equation}
where $\sub{n}_j$ are the children of the branching node $\sub{n}$.

\medskip

\noindent {\bf Case 1.} Suppose that $\vartheta<2$. Then \Cref{thm:Iserlis} and \Cref{thm:covariance_w0} yield
\[
\begin{split}
\E |\widetilde{\Jc_{\Tc}}(\tau)|^2 \lesssim \left(\frac{\lambda T}{L^{d}}\right)^{2n} &\ \int_{(\R^{2n+1})^2}  \sum_{\mathscr{D}} \,\sum_{\mathscr{D}'}    |\epsilon_{\mathscr{D}}\, \epsilon_{\mathscr{D}'} |\, \Kc (\tau, \mathscr{D}, \tau[\Lc])\, \overline{\Kc (\tau, \mathscr{D}', \tau'[\Lc])} \\
&  \prod_{\{\sub{l}_1,\sub{l}_2\}\in\Pc} \left( \vartheta\, b_{k_{\sub{l}_1}}^2 \langle \tau_{\sub{l}_1} -\tau_{\sub{l}_2}\rangle^{-N}  \left\langle \tau_{\sub{l}_1}\right\rangle^{-1}  \left\langle \tau_{\sub{l}_2}\right\rangle^{-1}  +  c_{k_{\sub{l}_1}}^2\, \langle \tau_{\sub{l}_1}\rangle^{-N} \langle \tau_{\sub{l}_2}\rangle^{-N} \right)\, d\tau_{\sub{l}_1}d\tau_{\sub{l}_2}
\end{split}
\]
where we sum over all decorations $\mathscr{D}, \mathscr{D}'$ that respect the pairings given by $\mathcal{P}$. Next we use \eqref{eq:FTbound_higher_iterates} to estimate the kernels $\Kc$,
\[
\begin{split}
\E |\widetilde{\Jc_{\Tc}}(\tau)|^2 \lesssim \left(\frac{\lambda T}{L^{d}}\right)^{2n} &\ \int_{(\R^{2n+1})^2}  \sum_{\mathscr{D}} \,\sum_{\mathscr{D}'}   | \epsilon_{\mathscr{D}}\, \epsilon_{\mathscr{D}'} | \langle \tau-d_{\sub{r}}\,\tau_{\sub{r}}\rangle^{-4}\langle \tau-d_{\sub{r}}\,\tau_{\sub{r}}'\rangle^{-4} \, \prod_{\sub{n}\in\Nc} \langle \tau_{\sub{n}}\rangle^{-1} \, \langle \tau_{\sub{n}}'\rangle^{-1} \\
&\prod_{\{\sub{l}_1,\sub{l}_2\}\in\Pc} \left( \vartheta\,  b_{k_{\sub{l}_1}}^2\,\langle \tau_{\sub{l}_1} -\tau_{\sub{l}_2}\rangle^{-N} \, \left\langle \tau_{\sub{l}_1}\right\rangle^{-1} \, \left\langle \tau_{\sub{l}_2}\right\rangle^{-1} + c_{k_{\sub{l}_1}}^2\, \langle \tau_{\sub{l}_1}\rangle^{-N}\, \langle \tau_{\sub{l}_2}\rangle^{-N} \right)\, d\tau_{\sub{l}_1}d\tau_{\sub{l}_2}.
\end{split}
\]
We multiply this quantity by $\langle \tau\rangle^{2b}$ and integrate in $\tau$:
\[
\begin{split}
\E \norm{(\Jc_{\Tc})_k}^2_{h^b}\lesssim \left(\frac{\lambda T}{L^{d}}\right)^{2n} & \int_{(\R^{2n+1})^2} \sum_{\mathscr{D}} \,\sum_{\mathscr{D}'}    |\epsilon_{\mathscr{D}}\, \epsilon_{\mathscr{D}'}|  \min\{ \langle \tau_{\sub{r}}\rangle^{2b-2},\langle  \tau_{\sub{r}}'\rangle^{2b-2}\}\, \langle d_{\sub{r}}\, (\tau_{\sub{r}} - \tau_{\sub{r}}')\rangle^{-2}\prod_{\sub{n}\in\Nc-\{\sub{r}\}} \langle \tau_{\sub{n}}\rangle^{-1} \, \langle \tau_{\sub{n}}'\rangle^{-1} \\
& \prod_{\{\sub{l}_1,\sub{l}_2\}\in\Pc} \left( \vartheta\,  b_{k_{\sub{l}_1}}^2\,\langle \tau_{\sub{l}_1} -\tau_{\sub{l}_2}\rangle^{-N} \, \left\langle \tau_{\sub{l}_1}\right\rangle^{-1} \, \left\langle \tau_{\sub{l}_2}\right\rangle^{-1} + c_{k_{\sub{l}_1}}^2\, \langle \tau_{\sub{l}_1}\rangle^{-N}\, \langle \tau_{\sub{l}_2}\rangle^{-N} \right)\, d\tau_{\sub{l}_1}d\tau_{\sub{l}_2}.
\end{split}
\]
Next we expand the product over $\Pc$. All such terms admit analogous estimates, and so we only detail the computation in the case of the term coming from the product of all the $ c_{k_{\sub{l}_1}}^2$. In order to bound this term, we integrate over all $\tau_{\sub{l}_1}$ and $\tau_{\sub{l}_2}$, which yields
\[
\left(\frac{\lambda T}{L^{d}}\right)^{2n}\sum_{\mathscr{D}} \,\sum_{\mathscr{D}'} |\epsilon_{\mathscr{D}}\, \epsilon_{\mathscr{D}'} | \,\prod_{\sub{l}\in\Lc}  c_{k_{\sub{l}}}^2 \min\{ \langle Tq_{\sub{r}}\rangle^{2b-2},\langle T q_{\sub{r}}'\rangle^{2b-2}\}\, \langle d_{\sub{r}}\, T (q_{\sub{r}} - q_{\sub{r}}')\rangle^{-2}
\prod_{\sub{n}\in\Nc-\{\sub{r}\}} \langle Tq_{\sub{n}}\rangle^{-1} \, \langle Tq_{\sub{n}}'\rangle^{-1}.
\]
Recall that we sum over all decorations $\mathscr{D},\mathscr{D}'$ that respect the pairings given by $\mathcal{P}$. 
Thanks to the decay of the $ c_{k_{\sub{l}}}$, we may assume that $|k_{\sub{l}}|\leq L^{\theta}$ for all $\sub{l}\in\Lc$. Given that all the $q$'s are bounded by $L^{2\theta}$ we may fix the integer parts of each $T q_{\sub{n}}$ and $T q_{\sub{n}}'$, $\sub{n}\in\mathcal{N}$, and reduce this sum to a counting bound at the cost of a power of $L^{(b-1/2)+\theta}$.

Once the integer parts of the $T q_{\sub{n}}$ are fixed, we have also fixed $\sigma_{\sub{n}}\in \R$ and $\sigma_{\sub{n}}'\in \R$ such that 
\begin{equation}\label{eq:assignment_condition}
|\Omega_{\sub{n}} - \sigma_{\sub{n}}|\leq T^{-1}, \quad |\Omega_{\sub{n}}' - \sigma_{\sub{n}}'|\leq T^{-1}.
\end{equation}

All we need to do now is count the number of decorations $(k_{\sub{l}} \mid \sub{l}\in\widetilde{\Lc})$ that give rise to the pairing $\Pc$ satisfying $|k_{\sub{l}}|\leq L^{\theta}$ and \eqref{eq:assignment_condition}. We start with the first copy of the tree $\mathcal{T}$. We let $p$ be the number of pairings within the leaves in the first tree $\mathcal{L}$. Then \Cref{thm:special_counting} allows us to bound the number of such decorations by
\[ M:= L^{\theta} \left(T^{-1}\, L^d \rho \right)^{2n-p-2} L^{2d} T^{-1},\]
where $\rho$ was defined in \eqref{eq:def_rho}.

Next we consider the second copy of the tree. We let $\mathcal{S}$ be the set of leaves which were paired to the first copy of the tree (which have already been counted), and we consider the number of decorations $(k_{\sub{l}}' \mid \sub{l}\in\mathcal{L}-\mathcal{S})$ satisfying \eqref{eq:assignment_condition}. \Cref{thm:main_counting} with $\mathcal{R}=\mathcal{S}\cup \{\sub{r}\}$ yields the bound
\[ M':=L^{\theta} \left(T^{-1} \, L^d \rho\right)^{p}.\]
Combining these counting bounds, we obtain
\[
\E \norm{(\Jc_{\Tc})_k}^2_{h^b}  \lesssim L^{\theta+c(b-1/2)} \, \left(\frac{\lambda T}{L^{d}}\right)^{2n}\, \left(T^{-1} \, L^d \rho\right)^{2n-2} L^{2d} T^{-1} \lesssim L^{\theta+c(b-1/2)} \,\left(\lambda \rho\right)^{2n-2}\, \lambda^2 T.
\]
The same bound holds for $\E \norm{\Jc_{\Tc}}^2_{H^b(\R,H^s(\T_L^d))}$, which admits a similar proof involving an additional sum in $k$ (which yields an additional factor $L^{(2s+d)\theta}$).

To finish the proof, set 
\[
A:= e^{4n+2} \, L^{\theta'}\, L^{\theta+c(b-1/2)} \,\left(\lambda \rho\right)^{2n-2}\, \lambda^2 T\]
for $\theta'>0$ as small as desired. The Chebyshev inequality and Gaussian hypercontractivity estimates (\Cref{thm:hypercontractivity}) yield:
\[
\begin{split}
\P \left( \norm{(\Jc_{\Tc})_k}^2_{h^b}>A\right) & \leq \frac{\E  \norm{(\Jc_{\Tc})_k}^{2q}_{h^b}}{A^{q}} \leq \left(\frac{q^{2n+1}}{L^{\theta'}}\right)^{q}
\end{split}
\]
Finally, set $q=\exp( \frac{\theta'}{2n+1} \, \log L -  \frac{\varepsilon}{2n+1})$ and $c=\varepsilon e^{- \frac{\varepsilon}{2n+1}}$. With this choice 
\[
\P \left( \norm{(\Jc_{\Tc})_k}^2_{h^b}>A\right) \leq e^{-cL^{\theta'/(2n+1)}} 
\]
as desired. The intersection of $L^{\theta}$ of such $L$-certain events (which is $L$-certain) yields \eqref{eq:n_iterate} for all $|k|\leq L^{\theta}$. When $|k|>L^{\theta}$ one may exploit the fast decay of $c_k$ and $b_k$ to gain an arbitrarily large negative power of $L$. A Gaussian hypercontractivity argument using $\sum_{|k|>L^{\theta}} \langle k\rangle^{2s} \E \norm{(\Jc_{\Tc})_k}_{h^b}^2$ then concludes the proof.

\medskip

{\bf Case 2.} Now suppose that $\vartheta>2$. Most of the steps are similar to those in Case 1, so we only give a sketch of the proof. We apply \Cref{thm:covariance_w0} to \eqref{eq:variance_n} in this new regime.
\[
\begin{split}
\E |\widetilde{\Jc_{\Tc}}(\tau)|^2 \lesssim \left(\frac{\lambda T}{L^{d}}\right)^{2n} &\ \int_{(\R^{2n+1})^2}  \sum_{\mathscr{D}} \,\sum_{\mathscr{D}'}    |\epsilon_{\mathscr{D}}\, \epsilon_{\mathscr{D}'}|  \,\Kc (\tau, \mathscr{D}, \tau[\Lc])\, \overline{\Kc (\tau, \mathscr{D}', \tau'[\Lc])}\\
& \prod_{\{\sub{l}_1,\sub{l}_2\}\in\Pc} \left( \frac{b_{k_{\sub{l}_1}}^2}{\vartheta} \langle \tau_{\sub{l}_1} -\tau_{\sub{l}_2}\rangle^{-N}  \langle \tau_{\sub{l}_1} / \vartheta\rangle^{-1} \langle \tau_{\sub{l}_2} /\vartheta\rangle^{-1}  +  c_{k_{\sub{l}_1}}^2 \langle \tau_{\sub{l}_1}\rangle^{-N} \langle \tau_{\sub{l}_2}\rangle^{-N} \right)\, d\tau_{\sub{l}_1}d\tau_{\sub{l}_2}
\end{split}
\]
As before, we expand the product over $\Pc$. The term involving the product over all $c_{k_{\sub{l}_1}}^2$ is treated as in Case 1, so let us focus on the term which involves a product over the $b_{k_{\sub{l}_1}}^2$ as an example. We bound the kernels $\Kc$ using \eqref{eq:FTbound_higher_iterates}, we multiply by $\langle\tau\rangle^{2b}$ and integrate in $\tau$. This leads to
\[
\begin{split}
\E \norm{(\Jc_{\Tc})_k}_{h^b}^2\lesssim \left(\frac{\lambda T}{L^{d}}\right)^{2n} & \int_{(\R^{2n+1})^2} \sum_{\mathscr{D}} \,\sum_{\mathscr{D}'}    |\epsilon_{\mathscr{D}}\, \epsilon_{\mathscr{D}'} | \min\{ \langle \tau_{\sub{r}}\rangle^{2b-2},\langle  \tau_{\sub{r}}'\rangle^{2b-2}\}\, \langle d_{\sub{r}}\, (\tau_{\sub{r}} - \tau_{\sub{r}}')\rangle^{-2} \\
& \prod_{\sub{n}\in\Nc-\{\sub{r}\}} \langle \tau_{\sub{n}}\rangle^{-1} \, \langle \tau_{\sub{n}}'\rangle^{-1}\,
\prod_{\{\sub{l}_1,\sub{l}_2\}\in\Pc}  \frac{b_{k_{\sub{l}_1}}^2}{\vartheta}\, \langle \tau_{\sub{l}_1} -\tau_{\sub{l}_2}\rangle^{-N} \, \langle \tau_{\sub{l}_1} / \vartheta\rangle^{-1} \,\langle \tau_{\sub{l}_2} /\vartheta\rangle^{-1}\, d\tau_{\sub{l}_1}d\tau_{\sub{l}_2}.
\end{split}
\]
Let us isolate the sum and rewrite this integral as 
\begin{equation}\label{eq:last_integral}
\int_{(\R^{2n+1})^2} S(k,\tau[\Lc], \tau' [\Lc])\, \prod_{\{\sub{l}_1,\sub{l}_2\}\in\Pc} \vartheta^{-1}\, \langle \tau_{\sub{l}_1} -\tau_{\sub{l}_2}\rangle^{-N} \, \langle \tau_{\sub{l}_1} / \vartheta\rangle^{-1} \,\langle \tau_{\sub{l}_2} /\vartheta\rangle^{-1}\, d\tau_{\sub{l}_1}d\tau_{\sub{l}_2},
\end{equation}
where
\[
S(k,\tau[\Lc], \tau' [\Lc])= \left(\frac{\lambda T}{L^{d}}\right)^{2n} \sum_{\mathscr{D}} \,\sum_{\mathscr{D}'}    |\epsilon_{\mathscr{D}}\, \epsilon_{\mathscr{D}'}|  \min\{ \langle \tau_{\sub{r}}\rangle^{2b-2},\langle  \tau_{\sub{r}}'\rangle^{2b-2}\}\, \langle d_{\sub{r}}\, (\tau_{\sub{r}} - \tau_{\sub{r}}')\rangle^{-2}  \prod_{\sub{n}\in\Nc-\{\sub{r}\}} \langle \tau_{\sub{n}}\rangle^{-1} \, \langle \tau_{\sub{n}}'\rangle^{-1} \, \prod_{\sub{l}\in\Lc} b_{k_{\sub{l}}}^2.
\]
We fix all the $\tau_{\sub{n}}$ and $\tau_{\sub{n}}'$ and derive a uniform bound for $S$. As in Case 1, we may assume that $|k_{\sub{l}}|\leq L^{\theta}$ for all $\sub{l}\in\Lc$ thanks to the decay of the $b_{k_{\sub{l}}}$. Once again, we can fix the integer parts of all $Tq_{\sub{n}}$ and $T q_{\sub{n}}'$ and reduce the bound to a counting problem at the cost of a power of $L^{\theta + c(b-1/2)}$. As in Case 1, we have \eqref{eq:assignment_condition}, and therefore we can use \Cref{thm:main_counting} and \Cref{thm:special_counting} again. They yield the bound
\[
|S(k,\tau[\Lc], \tau' [\Lc])|\lesssim L^{\theta+c(b-1/2)} \,\left(\lambda \rho\right)^{2n-2}\, \lambda^2 T.
\]
All that remains is estimating the integral in \eqref{eq:last_integral}, i.e.
\[
\int_{(\R^{2n+1})^2} \prod_{\{\sub{l}_1,\sub{l}_2\}\in\Pc} \vartheta^{-1}\, \langle \tau_{\sub{l}_1} -\tau_{\sub{l}_2}\rangle^{-N} \, \langle \tau_{\sub{l}_1} / \vartheta\rangle^{-1} \,\langle \tau_{\sub{l}_2} /\vartheta\rangle^{-1}\, d\tau_{\sub{l}_1}d\tau_{\sub{l}_2}.
\]
This integral can be factored into $2n+1$ integrals, and so it suffices to use the elementary bound:
\[
\int_{\R^2} \vartheta^{-1}\, \langle \tau_1 -\tau_2\rangle^{-N} \, \langle \tau_1 / \vartheta\rangle^{-1} \,\langle \tau_2/\vartheta\rangle^{-1}\, d\tau_1 d\tau_2\lesssim 1
\]
uniformly in $\vartheta>2$. 
As in Case 1, the Chebyshev inequality and \Cref{thm:hypercontractivity} finish the proof of the proposition.
\end{proof}

\section{Convergence to the WKE}\label{sec:WKE}

\subsection{First reductions}

In this section we want to find the equation satisfied by $\E |u_k(t)|^2$ in the limit.  We have the following technical lemma that will allow us to ignore small-probability events when we compute the expectation:
 
\begin{lem}\label{thm:restricted_expectation} Let the (finite) intersection of all $L$-certain events in \Cref{thm:n_iterate} and \Cref{thm:error} be $A$, then there exist $c>0$ and $0<\theta\ll 1$ such that for all $t\in [0,L^d]$ we have 
\begin{equation}\label{eq:restricted_expectation}
\E |u_k (t)|^2=  \E ( \mathbbm{1}_A |u_k (t)|^2 ) + \E ( \mathbbm{1}_{A^c} |u_k (t)|^2 ) = \E_A |u_k (t)|^2 + \O (e^{-c L^{\theta}}).
\end{equation}
\end{lem}
\begin{proof}
By the Cauchy-Schwarz inequality,
\[
\E ( \mathbbm{1}_{A^c} |u_k (t)|^2 ) \leq (\E |u_k (t)|^4)^{1/2} \, \P (A^c)^{1/2} \lesssim (\E |u_k (t)|^4)^{1/2} \, e^{-c L^{\theta}}.
\]
Our goal will be to show that $\E |u_k (t)|^4$ is at most polynomial in $L$ uniformly in $k\in\Z_L^d$ and $t\in [0,L^d]$. 

We repeat the argument that led to \eqref{eq:BalanceEnergy}, which is based on the It\^o isometry:
\[
\begin{split}
\norm{u(t)}^2_{L^2(\T_L^d)} -\norm{u(0)}^2_{L^2(\T_L^d)} = & \ -2\nu \int_0^t \norm{(1-\Delta_{\zeta})^{r/2} u(t')}^2_{L^2(\T_L^d)} \! dt' + 2\nu B t + 2\nu \mbox{Re} \int_0^t \int_{\T_L^d} \overline{u(t',x)} \, dx\,d\beta (t')\\
 \leq & \ 2\nu B t  + 2\nu \mbox{Re} \int_0^t \int_{\T_L^d} \overline{u(t',x)} \, dx\, d\beta (t') .
\end{split}
\]
Using the Plancherel theorem, we find that
\begin{equation}
|u_k(t)|^2 \leq  \sum_{k\in\Z_L^d} |u_k(0)|^2 + 2\nu B t  + 2\nu \mbox{Re} \int_0^t \int_{\T_L^d} \overline{u(t',x)} \, dx\, d\beta (t').
\end{equation}
Using the Cauchy-Schwarz inequality, we estimate the expectation of the fourth moment:
\begin{equation}\label{eq:4th_moment}
\E |u_k(t)|^4 \lesssim \E\left(\sum_{k\in\Z_L^d} c_k^2 |\eta_k|^2\right)^2 + (2\nu B t)^2  + 4\nu^2 \E \left | \int_0^t \int_{\T_L^d} \overline{u(t',x)} \, dx\, d\beta (t') \right |^2.
\end{equation}
Finally, we estimate the rightmost term in \eqref{eq:4th_moment}. We use the Cauchy-Schwarz inequality in the $x$ variable, and the It\^o formula to bound 
\[ 
\begin{split}
\E \left | \int_0^t \int_{\T_L^d} \overline{u(t',x)} \, dx\, d\beta (t') \right |^2 & 
\lesssim L^d \, \int_{\T_L^d} \int_0^t \E |u(t',x)|^2\, dt' dx \lesssim L^d \,\int_0^t \E \norm{u(t')}^2_{L^2(\T_L^d)} \, dt' \lesssim L^{2d} \,t,
\end{split}
\]
where we use the fact that the mass has size $L^d$, cf.\  \eqref{eq:BalanceEnergy2}. Back to \eqref{eq:4th_moment}, we use the fact that $\eta_k$ are independent Gaussian random variables to obtain
\[
\left(\E |u_k(t)|^4\right)^{1/2} \lesssim \sum_{k\in\Z_L^d} c_k^2+2\nu B t  + \sqrt{4\nu^2 L^{2d} \, t}.
\]
Given that $t\leq L^d$ and $B\sim L^d$ (cf.\ \eqref{eq:aboutB}), we find that $\E |u_k(t)|^4=\O (L^{4d})$ uniformly in $t\in [0,L^d]$ and $k\in\Z_L^d$, as we wanted.
\end{proof}

Let us set $t=T s$ where $s\in [1/2,1]$ and $T$ satisfies \eqref{eq:main2_T}, and recall that $\E |u_k (t)|^2 = \E |w_k (s)|^2$. Thanks to \eqref{eq:restricted_expectation}, it is enough to estimate the $\E_A |u_k (t)|^2 = \E_A |w_k (s)|^2$. To do so, we write $w_k (s)$ as in \eqref{eq:Picarditeration}, 
\begin{equation}\label{eq:expectation_Picard}
\begin{split}
\E_A |w_k (s)|^2  =&\  \E_A \left[ |w_k^{(0)} (s)|^2 +  |w_k^{(1)} (s)|^2 + 2 \mbox{Re} \overline{w_k^{(0)} (s)} w_k^{(1)} (s) +2 \mbox{Re} \overline{w_k^{(0)} (s)} w_k^{(2)} (s) \right] \\
& + \sum_{3\leq n \leq N} 2 \E_A \mbox{Re} \overline{w_k^{(0)} (s)} w_k^{(n)} (s) + \sum_{1\leq n_1,n_2 \leq N;\ n_1+n_2\geq 3} 2 \E_A  \overline{w_k^{(n_1)} (s)} w_k^{(n_2)} (s)\\
& + \sum_{n\leq N} 2 \E_A \mbox{Re} \overline{w_k^{(n)} (s)} (\mathcal{R}_{N+1})_k (s) + \E_A | (\mathcal{R}_{N+1})_k (s) |^2 \, .
\end{split}
\end{equation}

The terms in the second and third rows are lower order. For instance, \Cref{thm:n_iterate} yields
\[
\sum_{\substack{1\leq n_1,n_2 \leq N\\ n_1+n_2\geq 3}} 2 |\E_A  \overline{w_k^{(n_1)} (s)} w_k^{(n_2)} (s) | \lesssim L^{\theta+c(b-1/2)} \, \frac{T}{T_{\mathrm{kin}}}\ \sum_{\substack{1\leq n_1,n_2 \leq N\\ n_1+n_2\geq 3}} L^{-\delta (n_1+n_2-2)} \lesssim L^{-\delta/2}\, T/T_{\mathrm{kin}}
\]
by choosing $\theta$ and $b-1/2$ small enough in terms of $\delta$ (which is defined in \Cref{thm:main2}).
\Cref{thm:n_iterate} and \Cref{thm:error} take care of the terms in the third row of \eqref{eq:expectation_Picard} using similar arguments. The first term in the second row of \eqref{eq:expectation_Picard} requires slightly more work: one must repeat the arguments in the proof of \Cref{thm:n_iterate} in order to show that 
\[
| \E \overline{w_k^{(0)} (s)} w_k^{(n)} (s)|\lesssim L^{\theta} L^{-\delta (n-2)}\, \frac{T}{T_{\mathrm{kin}}}.
\]
As a result, we can arrange
\[
 \sum_{3\leq n \leq N} 2 |\E_A \mbox{Re} \overline{w_k^{(0)} (s)} w_k^{(n)} (s)|\leq L^{-\delta/2} \, \frac{T}{T_{\mathrm{kin}}}.
\]
Finally, it is easy to show that $\E\overline{w_k^{(0)} (s)} w_k^{(1)} (s) = 0$. We have thus shown that there for $\delta>0$ as in \eqref{eq:main2_T}, we have that 
\begin{equation}\label{eq:expectation_Picard_2}
\E_A |w_k (s)|^2  = \E_A \left[ |w_k^{(0)} (s)|^2 +  |w_k^{(1)} (s)|^2  +2 \mbox{Re} \overline{w_k^{(0)} (s)} w_k^{(2)} (s) \right] + \O \left( L^{-\delta/2} \frac{T}{T_{\mathrm{kin}}} \right).
\end{equation}

From now on we focus on the main terms on the right-hand side of \eqref{eq:expectation_Picard_2}, which are the ones that will give rise to the WKE.

\subsection{Main terms: from sums to integrals}

We have thus shown that 
\[
\E |u_k (t T_{\mathrm{kin}})|^2 = \E |w_k (t)|^2 = \E_A \left[ |w_k^{(0)} (t)|^2 +  |w_k^{(1)} (t)|^2  +2 \mbox{Re} \overline{w_k^{(0)} (t)} w_k^{(2)} (t) \right] + \O \left( L^{-\delta/2} t \right)
\]
for $t\leq T/T_{\mathrm{kin}}=L^{-\varepsilon}$ (which follows from \eqref{eq:expectation_Picard_2} with $s=t T_{\mathrm{kin}}/T$). We now proceed to compute the contributions from each of the terms in the right-hand side. We note that the results that follow in the remaining of this section are valid for any $t\in [0,1]$, not only $t\in [0,L^{-\varepsilon}]$.

First, we compute
\begin{equation}\label{eq:def_f}
\E |w_k^{(0)} (t)|^2 = c_k^2 \, e^{-2\varrho \gamma_k t}+\frac{b_k^2}{\gamma_k} \left( 1- e^{-2\varrho\gamma_k t}\right)
=: f(k,t), \qquad \qquad t\in [0,1].
\end{equation} 

Consider next $\E |w_k^{(1)} (t)|^2$. Using \eqref{eq:Picard_iterates}, we have that 
\begin{equation}\label{eq:w1_abs}
\begin{split}
|w_k^{(1)} (t)|^2  =&\ \left(\frac{\lambda T_{\mathrm{kin}}}{L^{d}} \right)^2\,  \sum_{k=k_1-k_2+k_3}  \sum_{k=k_1'-k_2'+k_3'} \, \epsilon_{k_1,k_2,k_3}\,\epsilon_{k_1',k_2',k_3'} \\
& \hspace{1cm} \int_0^t\int_0^t  e^{-\varrho\gamma_k (2t-t_1-t_2)} \prod_{j=1}^{3} w_{k_j}^{(0)} (t_1)^{\imath_j}\, w_{k_j'}^{(0)} (t_2)^{-\imath_j} \, e^{iT_{\mathrm{kin}} t_1\Omega^{13}_{2k} - iT_{\mathrm{kin}} t_2\Omega^{1'3'}_{2'k}} \, dt_1 dt_2 .
\end{split}
\end{equation}

We take the expectation and consider the terms where $\epsilon=+1$ (the rest are lower order). Using \Cref{thm:Iserlis}, there are two possible pairings between $\{ k_1,k_2,k_3\}$ and $\{ k_1',k_2',k_3'\}$, which correspond to $k_1$ being paired with $k_1'$ or $k_3'$. This implies that $\Omega^{13}_{2k}=\Omega^{1'3'}_{2'k}$ for all pairings. We also need to use the fact that 
\begin{equation}
\E w_k^{(0)} (t_1) \overline{w_k^{(0)} (t_2)} = c_k^2\, e^{-\varrho \gamma_k (t_1+t_2)} +\frac{b_k^2}{\gamma_k} \left( e^{-\varrho\gamma_k |t_1-t_2|} - e^{-\varrho\gamma_k (t_1+t_2)}\right),
\end{equation}
which follows from the independence of $\beta_k(t)$ and $\eta_k$. As a result, we have that 
\begin{equation}\label{eq:triple_product}
\E \prod_{j=1}^{3} w_{k_j}^{(0)} (t_1)^{\imath_j}\, w_{k_j'}^{(0)} (t_2)^{-\imath_j}  = 2\, \prod_{j=1}^{3}  \left[ c_{k_j}^2\, e^{-\varrho \gamma_{k_j} (t_1+t_2)}+\frac{b_{k_j}^2}{\gamma_{k_j}} \left( e^{-\varrho \gamma_{k_j} |t_1-t_2|} - e^{-\varrho\gamma_{k_j} (t_1+t_2)}\right)\right]
\end{equation}
accounting for the two possible pairings between $\{ k_1,k_2,k_3\}$ and $\{ k_1',k_2',k_3'\}$.

Next we consider two different cases: $\varrho\gtrsim 1$ and $\varrho\ll 1$. The integration in \eqref{eq:w1_abs} must be performed differently in each case in order to identify the top order terms.

\begin{prop} Suppose that $\varrho=\nu T_{\mathrm{kin}} \sim L^{-\kappa}$ for some $\kappa>0$ with $T_{\mathrm{kin}}$ is as in \Cref{thm:main2}. Then there exists some $0<\theta\ll \kappa$ such that for all $0\leq t\leq 1$,
\begin{equation}\label{eq:towards_cubicNLS}
\E |w_k^{(1)} (t)|^2 =2\, \left(\frac{\lambda t\,T_{\mathrm{kin}}}{L^{d}} \right)^2\, \sum_{k=k_1-k_2+k_3}  \frac{\Big |\sin \left(\frac{tT_{\mathrm{kin}}\, \Omega^{13}_{2k}}{2}\right)\Big |^2}{|t\,T_{\mathrm{kin}}\,\Omega^{13}_{2k}/2|^2}\, \prod_{j=1}^{3} c_{k_j}^2 + \O \left( L^{-\theta}\, t \right).
\end{equation}
Moreover,
\begin{equation}\label{eq:conclusion_WKE_cubicNLS}
\E |u(k,s)|^2 = c_k^2 + \frac{s}{T_{\mathrm{kin}}}\,\Kc (c_{\cdot})(k) + \O \left( L^{-\delta}\, \frac{s}{T_{\mathrm{kin}}}\right),
\end{equation}
for $\Kc$ is as in \eqref{eq:def_K} and $L^{\delta}\leq s\leq T$ with $T$ satisfying \eqref{eq:main2_T}.
\end{prop}
\begin{rk}\label{rk:towards_cubicNLS}
Note that the top order term in \eqref{eq:towards_cubicNLS}, which has size $\O(t)$, is one of the terms that leads to the WKE for the cubic NLS equation, see Section 4 and Theorem 5.1 in \cite{DengHani}.
\end{rk}
\begin{proof}
We use \eqref{eq:triple_product} to compute the expectation of \eqref{eq:w1_abs}. Note that the only possible pairings imply that $\Omega:= \Omega^{13}_{2k} = \Omega^{1'3'}_{2'k}$. We focus on the terms where $\epsilon=+1$ since the rest can be shown to be lower order.
\[
\begin{split}
\E |w_k^{(1)} (t)|^2   = 2\, \left(\frac{\lambda T_{\mathrm{kin}}}{L^{d}} \right)^2\, & \sum_{k=k_1-k_2+k_3} \int_0^t \! \int_0^t  e^{-\varrho\gamma_k (2t-t_1-t_2)} \, e^{iT_{\mathrm{kin}}\Omega^{13}_{2k} (t_1-t_2)} \\ 
& \prod_{j=1}^{3}  \left[ c_{k_j}^2\, e^{-\varrho \gamma_{k_j} (t_1+t_2)}+\frac{b_{k_j}^2}{\gamma_{k_j}} \left( e^{-\varrho \gamma_{k_j} |t_1-t_2|} - e^{-\varrho\gamma_{k_j} (t_1+t_2)}\right)\right]\, dt_1 dt_2   +\O (L^{-\theta} t^2).
\end{split}
\]

We split the integral into two parts: $t_1\geq t_2$ and its complement. Note that the resulting integrands are similar up to conjugation, and so we may write them as twice the real part of the integration over $t_1\geq t_2$. More precisely, 
\[
\begin{split}
\E |w_k^{(1)} (t)|^2   = 4\, \left(\frac{\lambda T_{\mathrm{kin}}}{L^{d}} \right)^2\, & \sum_{k=k_1-k_2+k_3}  \int_0^t\int_{t_2}^t  e^{-\varrho\gamma_k (2t-t_1-t_2)} \, \cos \left( T_{\mathrm{kin}} \Omega^{13}_{2k} (t_1-t_2)\right) \\ & \prod_{j=1}^{3}  \left[ c_{k_j}^2\, e^{-\varrho \gamma_{k_j} (t_1+t_2)}+\frac{b_{k_j}^2}{\gamma_{k_j}} \left( e^{-\varrho \gamma_{k_j} |t_1-t_2|} - e^{-\varrho\gamma_{k_j} (t_1+t_2)}\right)\right]\, dt_1\, dt_2 + \O (L^{-\theta} t^2).
\end{split}
\]
When $t_1\geq t_2$, we may rewrite
\begin{equation}\label{eq:triple_product2}
 \prod_{j=1}^{3}  \left[ c_{k_j}^2\, e^{-\varrho \gamma_{k_j} (t_1+t_2)}+\frac{b_{k_j}^2}{\gamma_{k_j}} \left( e^{-\varrho \gamma_{k_j} |t_1-t_2|} - e^{-\varrho\gamma_{k_j} (t_1+t_2)}\right)\right] 
 = e^{-\varrho (t_1-t_2) \sum_{j=1}^{3} \gamma_{k_j}}\, \prod_{j=1}^{3}  f(k_j,t_2)
\end{equation}
for $f=\E |w^{(0)}|^2$ as in \eqref{eq:def_f}. Letting 
\begin{equation}\label{eq:Gammas}
\Gamma_{\pm} =\pm \gamma_k+\sum_{j=1}^{3} \gamma_{k_j}\ ,
\end{equation}
we integrate by parts the cosine in the $t_1$ variable:
\[
\begin{split}
\E |w_k^{(1)} (t)|^2   = &\ 4\, \left(\frac{\lambda T_{\mathrm{kin}}}{L^{d}} \right)^2\, \sum_{k=k_1-k_2+k_3} \int_0^t  e^{-(t-t_2)\,\varrho \Gamma_{+}} \, \frac{\sin \left( T_{\mathrm{kin}} \Omega^{13}_{2k} (t-t_2)\right)}{T_{\mathrm{kin}} \Omega^{13}_{2k}} \prod_{j=1}^{3}  f(k_j,t_2)\,  dt_2\\ 
& + E_{1,k} (t)+  \O (L^{-\delta} t^2),
\end{split}
\]
where we have only explicitly written the boundary terms from the integration by parts, the rest being part of $E_{1,k} (t)$. Finally, we integrate by parts the sine in the $t_2$ variable:
\[
\begin{split}
\E |w_k^{(1)} (t)|^2   = &\ 4\, \left(\frac{\lambda T_{\mathrm{kin}}}{L^{d}} \right)^2\, \sum_{k=k_1-k_2+k_3} e^{-t \varrho \Gamma_{+}} \, \frac{1-\cos \left( T_{\mathrm{kin}} \Omega^{13}_{2k} t\right)}{(T_{\mathrm{kin}} \Omega^{13}_{2k})^2}\, \prod_{j=1}^{3} f(k_j,0)\\
&+ E_{1,k} (t)+ E_{2,k} (t)+ \O (L^{-\theta} t^2).
\end{split}
\]
Using the identity $1-\cos (x)= 2\sin^2 (x/2)$, together with the fact that $e^{-t\varrho\Gamma_{+}}=1+ \O(t\,\varrho L^{\theta})$, and choosing $\theta < \kappa/2$ yields the top order term in \eqref{eq:towards_cubicNLS}.

In order to finish the proof, let us study the errors $E_{1,k}(t)$ and $E_{2,k}(t)$. We start with the latter:
\[
E_{2,k}(t) = 8\, \left(\frac{\lambda T_{\mathrm{kin}}}{L^{d}} \right)^2\,  \sum_{k=k_1-k_2+k_3}  \int_0^t \frac{\sin^2 \left( T_{\mathrm{kin}} \Omega^{13}_{2k} (t-t_2)/2\right)}{(T_{\mathrm{kin}} \Omega^{13}_{2k})^2}  \frac{\pa}{\pa t_2}\left( e^{- (t-t_2)\varrho\Gamma_{+}} \, \prod_{j=1}^{3}  f(k_j,t_2) \right) \, dt_2.
\]
Regardless of which term the derivative hits, we gain a factor of $\varrho$, and potential factors of $\gamma_{k_j}$ (which are $\O (L^{2\theta})$ thanks to the rapid decay of $c_{k_j}$ and $b_{k_j}$). Therefore 
\[
\begin{split}
|E_{2,k}(t)|&  \lesssim  \varrho\,L^{2\theta}\, \left(\frac{\lambda T_{\mathrm{kin}}}{L^{d}} \right)^2\, \sum_{k=k_1-k_2+k_3}  \int_0^t \frac{\sin^2 \left( T_{\mathrm{kin}} \Omega^{13}_{2k} (t-t_2)/2\right)}{(T_{\mathrm{kin}} \Omega^{13}_{2k})^2} \, \prod_{j=1}^{3}  \left[ c_{k_j}^2+b_{k_j}^2\right]\, dt_2\\
&  \lesssim  L^{2\theta-\kappa}\, \left(\frac{\lambda T_{\mathrm{kin}}}{L^{d}} \right)^2\,  \int_0^t \sum_{k=k_1-k_2+k_3}  \prod_{j=1}^{3}  \left(c_{k_j}^2+b_{k_j}^2\right)\, \frac{\sin^2 \left( T_{\mathrm{kin}} \Omega^{13}_{2k} t_2/2\right)}{(T_{\mathrm{kin}} \Omega^{13}_{2k})^2} \,   dt_2.
\end{split}
\]
Using the arguments in Theorem 5.1 in \cite{DengHani} (see also the analogous \Cref{NTDH} below), one can replace the inner sum by an integral and prove that it has size $L^{2d}\,T_{\mathrm{kin}}^{-1}$. All in all $E_{2,k}(s)=\O ( L^{2\theta-\kappa} t)=\O( L^{-\theta} t)$ by further reducing $\theta < \kappa/3$.

Finally, consider the error $E_{1,k}(t)$, which reads
\[
\begin{split}
E_{1,k}(t) = &\ 4\, \left(\frac{\lambda T_{\mathrm{kin}}}{L^{d}} \right)^2\! \sum_{k=k_1-k_2+k_3}  \int_0^t  \! \int_{t_2}^{t}  \frac{\sin \left( T_{\mathrm{kin}} \Omega^{13}_{2k} (t_1-t_2)\right)}{T_{\mathrm{kin}} \Omega^{13}_{2k}} \frac{\pa}{\pa t_1}\left( e^{-t_1 \varrho \Gamma_{-}} \right)\, e^{-\varrho\, 2t\gamma_k +t_2 \varrho \Gamma_{+}} \, \prod_{j=1}^{3}  f(k_j,t_2)\, dt_1\,  dt_2 \\ 
=&\ 4\, \left(\frac{\lambda T_{\mathrm{kin}}}{L^{d}} \right)^2\! \sum_{k=k_1-k_2+k_3}  \int_0^t   \! \int_{t_2}^{t}  \frac{\sin \left( T_{\mathrm{kin}} \Omega^{13}_{2k} (t_1-t_2)\right)}{T_{\mathrm{kin}} \Omega^{13}_{2k}} (-\varrho \Gamma_{-})\, e^{-t_1 \varrho \Gamma_{-}}\, e^{-\varrho\, 2t\gamma_k +t_2 \varrho \Gamma_{+}} \, \prod_{j=1}^{3}  f(k_j,t_2)\, dt_1\,  dt_2.
\end{split}
\]
We integrate by parts the sine in $t_1$ again. This leads to a term like $E_{2,k}$, which can be estimated as detailed above, plus the following contribution:
\[
8\, \left(\frac{\lambda T_{\mathrm{kin}}}{L^{d}} \right)^2\! \sum_{k=k_1-k_2+k_3}   \int_0^t \! \int_{t_2}^{t}  \frac{\sin^2 \left( T_{\mathrm{kin}}\Omega^{13}_{2k} (t_1-t_2)/2\right)}{(T_{\mathrm{kin}} \Omega^{13}_{2k})^2} (\varrho\Gamma_{-})^2\, e^{-t_1 \varrho \Gamma_{-}}\, e^{-\varrho\, 2t\gamma_k +t_2 \varrho\Gamma_{+}} \, \prod_{j=1}^{3}  f(k_j,t_2)\, dt_1\,  dt_2 .
\]
As in the case of $E_{2,k}$, we have gained $\varrho^2 L^{2\theta}$ (from bounding $\Gamma_{-}^2$) times a Schwarz function of $k_1,k_2,k_3$ (uniformly in $t_1,t_2,t$). A similar argument to that of $E_{2,k}$ finishes the proof of \eqref{eq:towards_cubicNLS}.

Finally, \eqref{eq:conclusion_WKE_cubicNLS} follows from \eqref{eq:expectation_Picard},  \eqref{eq:towards_cubicNLS} and analogous estimates for $2 \mbox{Re} \E[\overline{w_k}^{(0)} (t) w_k^{(2)} (t)]$ . The proof may be found in Theorem 5.1 in \cite{DengHani}, and so we omit it.
\end{proof}

In the rest of this section we assume that $\varrho\gtrsim 1$. Our first result is a useful formula for $\E |w_k^{(1)} (t)|^2$ for such values of $\varrho$.

\begin{lem}\label{lem:towardsWKE}
Let $f$ be as in \eqref{eq:def_f}, $\Gamma_{\pm}$ as in \eqref{eq:Gammas} and $T_{\mathrm{kin}}$ as in \Cref{thm:main2}. Then there exists some $0<\theta \ll 1$ such that
\begin{equation}
\E |w_k^{(1)} (t)|^2 = I_1 (f) (k,t) - I_2 (f) (k,t) + \O \left(L^{-\theta}t\right) \quad 0\leq t\leq 1,
\end{equation}
where
\begin{equation}\label{eq:towardsWKE1}
I_1(f)(k,t) = 4 L^{-2d} \sum_{k=k_1-k_2+k_3} \frac{\nu^{-1}\Gamma_{-}}{\Gamma_{-}^2+(\nu^{-1}\,\Omega^{13}_{2k})^2}  \int_0^t e^{-2\varrho\gamma_k (t-t')}  \prod_{j=1}^{3}  f(k_j,t')\, dt'
\end{equation}
and 
\begin{equation}\label{eq:towardsWKE2}
\begin{split}
I_2(f)(k,t) = 4 L^{-2d} \sum_{k=k_1-k_2+k_3}
 \int_0^t &\ e^{-\varrho\Gamma_{+} (t-t')}  \mbox{Re}\left( \frac{\nu^{-1}\,e^{iT_{\mathrm{kin}}\Omega^{13}_{2k}(t-t')}}{\Gamma_{-}-i\nu^{-1}\Omega^{13}_{2k}}\right) \prod_{j=1}^{3} f(k_j,t')\, dt'.
\end{split}
\end{equation}
Note that the denominators in \eqref{eq:towardsWKE1}-\eqref{eq:towardsWKE2} do not vanish in view of \Cref{VIP_claim} below. Similarly,
\begin{equation}\label{eq:w0w2}
2\mbox{Re}\, \E \left[ w_k^{(0)}(t) \,\overline{w_k^{(2)}(t)} \right] = \widetilde{I_1} (f) (k,t) - \widetilde{I_2} (f) (k,t) + \O \left( L^{-\theta}\, t\right)
\end{equation}
where 
\begin{equation}\label{eq:w0w2_kernels}
\begin{split}
\widetilde{I_1} (f) (k,t) & = 4 L^{-2d} \sum_{k=k_1-k_2+k_3}  \frac{\nu^{-1}\Gamma_{-}}{\Gamma_{-}^2+(\nu^{-1}\,\Omega^{13}_{2k})^2}  \int_0^t e^{-2\varrho\gamma_k (t-t')} \, \Kf[f(\cdot,t')]\, dt'\\
\widetilde{I_2} (f) (k,t) & =  4 L^{-2d}\sum_{k=k_1-k_2+k_3} 
 \int_0^t  e^{-2\varrho \Gamma_{+}\, (t-t')}  \mbox{Re}\left( \frac{\nu^{-1}e^{iT_{\mathrm{kin}}\Omega^{13}_{2k}(t-t')}}{\Gamma_{-}-i\nu^{-1}\Omega^{13}_{2k}}\right) \Kf[f(\cdot,t')]\, dt'
\end{split}
\end{equation}
where 
\begin{equation}\label{eq:rest_of_kernel}
\Kf[a] =  a(k_1) a (k_3) a (k) - a (k_2) a(k_3) a (k) - a (k_1) a (k_2) a (k) .
\end{equation}
\end{lem}
\begin{proof}
\textbf{The term $\E |w_k^{(1)} (t)|^2$.} Let us ignore the terms corresponding to $\epsilon=-1$ (i.e. $k_1=k_2=k_3$) since they give rise to the lower order error and may easily be computed separately. The integral in \eqref{eq:w1_abs} over the square $[0,t]^2$ may be separated into two triangles: the upper triangle $t_2\geq t_1$ and the lower triangle $t_1>t_2$, which are complex conjugates. For that reason, we may focus on the region $t_1>t_2$. Using \eqref{eq:triple_product}, we have that
\[
\E \prod_{j=1}^{3} w_{k_j}^{(0)} (t_1)^{\imath_j}\, w_{k_j'}^{(0)} (t_2)^{-\imath_j} = 2\,e^{-\varrho (t_1-t_2) \sum_{j=1}^{3} \gamma_{k_j}}\, \prod_{j=1}^{3}  f(k_j,t_2).
\]

We plug this identity into \eqref{eq:w1_abs} and integrate in the variable $t_1$: 
\begin{equation}\label{eq:upper_triangle}
\begin{split}
 \E |w_k^{(1)} (t)|^2 =&\ \left(\frac{\lambda T_{\mathrm{kin}}}{L^{d}} \right)^2 \! \sum_{k=k_1-k_2+k_3} \frac{2}{\varrho \Gamma_{-}-iT_{\mathrm{kin}} \Omega^{13}_{2k}} \int_0^t  \prod_{j=1}^{3}  f(k_j,t_2)\, \left[ e^{-2\varrho\gamma_k (t-t_2)} - e^{-\varrho \Gamma_+ (t-t_2) +i T_{\mathrm{kin}} \Omega^{13}_{2k} (t-t_2)} \right]\, dt_2\\
& + \mbox{complex conjugate} + \O (L^{-\theta} t),
\end{split}
\end{equation}
where the error terms accounts for the terms in $\epsilon=-1$. Formulas \eqref{eq:towardsWKE1}-\eqref{eq:towardsWKE2} now follow from the complex conjugate.

\medskip

\textbf{The term $\E w_k^{(0)}(t) \,\overline{w_k^{(2)}(t)}$.} Following \eqref{eq:Picard_iterates}, one may write
\begin{equation}\label{eq:bigW}
\begin{split}
w_k^{(2)}(t) & = \Ic \Wc (w^{(1)},w^{(0)},w^{(0)})_k(t) + \Ic \Wc (w^{(0)},w^{(1)},w^{(0)})_k(t) + \Ic \Wc (w^{(0)},w^{(0)},w^{(1)})_k (t)\\
&  =: W_k^1 (t) + W_k^2 (t)+ W_k^3 (t).
\end{split}
\end{equation}
Consider, for example, 
\begin{multline}\label{eq:w0w2_example}
\E [ W_k^1 (t) \, \overline{w_k^{(0)}(t)} ]= -\left(\frac{\lambda T_{\mathrm{kin}}}{L^{d}} \right)^2\!\!
\sum_{\substack{k=k_1-k_2+k_3\\k_1=k_{11}-k_{12}+k_{13}}} \!\! \epsilon_{k_1,k_2,k_3} \epsilon_{k_{11},k_{12},k_{13}}
\int_0^t \! \int_0^{t_1} e^{-\varrho \gamma_k (t-t_1) - \vartheta \gamma_{k_1} (t_1-t_2)} e^{iT_{\mathrm{kin}} t_1 \Omega_k +i T_{\mathrm{kin}} t_2 \Omega_{k_1}} \\
 \E[ \overline{w_{k_2}^{(0)}(t_1)}  w_{k_3}^{(0)}(t_1) w_{k_{11}}^{(0)}(t_2) \overline{w_{k_{12}}^{(0)}(t_2)} w_{k_{13}}^{(0)}(t_2) w_{k}^{(0)}(t)]\, dt_2 \, dt_1.
\end{multline}
Note that there are only two possible pairings between $\{k_{11},k_{12},k_{13}\}$ and $\{k,k_2,k_3\}$ depending on whether we pair $k_{11}$ with $k$ or with $k_2$. Note that for both parings we have that $\Omega_{k}=-\Omega_{k_1}$. Therefore
\begin{multline}\label{eq:w0w2_correlations}
 \E[ \overline{w_{k_2}^{(0)}(t_1)}  w_{k_3}^{(0)}(t_1) w_{k_{11}}^{(0)}(t_2) \overline{w_{k_{12}}^{(0)}(t_2)} w_{k_{13}}^{(0)}(t_2) w_{k}^{(0)}(t)]  = 2\,  \left(c_{k}^2 e^{-\varrho \gamma_{k}(t+t_2)}+ \frac{b_{k}^2}{\gamma_{k}} [ e^{-\varrho \gamma_{k}(t-t_2)} - e^{-\varrho \gamma_{k}(t+t_2)}]\right)\\
 \prod_{j\in \{k_2,k_3\}}  \left(c_{j}^2 e^{-\varrho \gamma_{j}(t_1+t_2)}+ \frac{b_{j}^2}{\gamma_{j}} [ e^{-\varrho \gamma_{j}(t_1-t_2)} - e^{-\varrho \gamma_{j}(t_1+t_2)}]\right)\\
 = 2\, e^{-\varrho \gamma_{k}(t+t_2) - \varrho (\gamma_{k_2}+\gamma_{k_3})(t_1+t_2)} \,   f(k,t_2)\, f(k_2,t_2)\, f(k_3,t_2).
\end{multline}

Plugging \eqref{eq:w0w2_correlations} into \eqref{eq:w0w2_example}, we may integrate in $t_1$ first. As before, each boundary term yields one of the terms in \eqref{eq:w0w2}. The other terms which give rise to $\Kf[f]$ in \eqref{eq:rest_of_kernel} come from $W_k^2$ and $W_k^3$ in \eqref{eq:bigW}.
\end{proof}

Let us now focus on the term $\E |w_k^{(1)} (t)|^2$ as an example.
Our next goal is to replace the sums in \eqref{eq:towardsWKE1}-\eqref{eq:towardsWKE2} by integrals. In order to do so, we first isolate the sums by rewriting $I_1$ and $I_2$ as:
\begin{equation}\label{eq:fromItoSigma}
I_1(k,t)   = \int_0^t e^{-2\varrho\gamma_k (t-t')} \Sigma_1(k,t')\, dt', \qquad I_2 (k,t) = \int_0^t \Sigma_2(k,t',t-t')\, dt',
\end{equation}
for $k\in \Z_L^d$ and $t\in [0,1]$, where
\begin{equation}\label{eq:Sigmas}
\begin{split}
 \Sigma_1(k, t) & =4L^{-2d} \sum_{k=k_1-k_2+k_3} \frac{\nu^{-1}\Gamma_{-}}{\Gamma_{-}^2+(\nu^{-1} \Omega^{13}_{2k})^2}\prod_{j=1}^{3}  f_j(k_j,t) \ ,\\
 \Sigma_2(k,t, \tau) & =4L^{-2d}\sum_{k=k_1-k_2+k_3} \mbox{Re}\left( \frac{\nu^{-1} e^{iT_{\mathrm{kin}} \Omega^{13}_{2k}\tau}}{\Gamma_{-}-i\nu^{-1} \Omega^{13}_{2k}}\right)\,  e^{-\varrho\Gamma_{+} \tau}\, \prod_{j=1}^{3}  f_j(k_j,t) \ ,
\end{split}
\end{equation}
and where $f=f_1=f_2=f_3$ in our setting. 

\begin{prop}\label{thm:sum_to_integral} Suppose that $f_j (k,t)$ are Schwartz functions of $k$ uniformly in $t$,  $j=1,2,3$, and suppose that $\varrho\gtrsim 1$. Let $\Gamma_{\pm}$ be as in \eqref{eq:Gammas} and let $T_{\mathrm{kin}}$ satisfy \eqref{eq:main2_T}. Then there exists $c>0$ such that
\begin{equation}\label{eq:sum_to_integral}
\begin{split}
\Sigma_1(k,t) & = \Ss_1 (k,t) + \O ( L^{-c\delta}) \quad 0\leq t\leq 1,\\
\Sigma_2(k,t,\tau) & = \Ss_2 (k,t,\tau) +  \O (L^{-c\delta})\quad 0\leq t,\tau\leq 1,
\end{split}
\end{equation}
uniformly in $k$, $t$ and $\tau$, where $\delta$ was defined in \eqref{eq:main2_T}, and
\begin{equation}\label{eq:Ss}
\begin{split}
\Ss_1 (k,t) & = 4\int_{k=k_1-k_2+k_3} \frac{\nu^{-1}\,\Gamma_{-}}{\Gamma_{-}^2+(\nu^{-1} \Omega^{13}_{2k})^2} \, \prod_{j=1}^{3} f_j(k_j,t)\, dk_1\, dk_3,\\
\Ss_2 (k, t,\tau) & = 4\int_{k=k_1-k_2+k_3}  \mbox{Re}\left( \frac{\nu^{-1}\,e^{iT_{\mathrm{kin}} \Omega^{13}_{2k}\tau}}{\Gamma_{-}-i\nu^{-1} \Omega^{13}_{2k}}\right)\,  e^{-\varrho\Gamma_{+} \tau}\, \prod_{j=1}^{3} f_j(k_j,t)\, dk_1\, dk_3.
\end{split}
\end{equation}
\end{prop}

Before we prove this result, we cite the following lemma, whose proof is contained in Theorem 5.1 in \cite{DengHani}.

\begin{lem}\label{NTDH}
Let $\chi\in \mathcal S(\R)$, and $W\in \mathcal S(\R^{2d})$. Consider the sum
\begin{equation}\label{SDH}
\mathscr S=L^{-2d}\sum_{(k_1, k_2)\in \Z_L^{2d}} W\left(k_1,k_2\right)\chi(\widetilde{T}\, \Omega(k_1,k_2)),
\end{equation}
where $\Omega(k_1,k_2)=k_1 \cdot k_2$ and $\widetilde{T}$ as in \eqref{eq:main2_T}. Then, there exists $c>0$ such that
\begin{equation*}
\mathscr S=\int_{\R^{2d}}W(z_1, z_2) \chi(\widetilde{T} \, \Omega(z_1,z_2)) \, dz_1 dz_2+ \O ( L^{-c\delta} \widetilde{T}^{-1}\|W\|_{X^{100d, 1}})
\end{equation*}
where the implicit constants depend on finitely many Schwartz seminorms of $\chi$ and 
\[
\|W\|_{X^{k,1}}=\sum_{|\alpha|+|\beta|\leq k} \|x^\alpha \partial^\beta W\|_{L^1(\R^{2d})}.
\]
\end{lem}

\begin{proof}[Proof of Proposition \ref{thm:sum_to_integral}]
Since, we are only interested in the case when $t, t' \in [0,1]$ are fixed, it will be enough to prove a convergence theorem for a sum of the form:
\begin{equation}\label{Sigma}
\Sigma(k,t)=L^{-2d}\sum_{k=k_1-k_2+k_3} \frac{\nu^{-1}\,\Gamma_{-}}{\Gamma_{-}^2+(\nu^{-1}\Omega)^2}\, W(k_1, k_2, k_3; t),
\end{equation}
where $W\in \mathcal S(\R^{3d})$ uniformly in $t\in [0,1]$ and we set $\Omega=\Omega^{13}_{2k}$ for the rest of the proof. We will detail how to adapt the proof for the sum $\Sigma_2$ at the end.
We shall show that 
\begin{equation}\label{SigmaAim}
\Sigma(k,t)=\int_{k=k_1-k_2+k_3} \frac{\nu^{-1} \Gamma_{-}}{\Gamma_{-}^2+(\nu^{-1}\Omega)^2}\, \, W(k_1, k_2, k_3)\, dk_1\, dk_3 +O(L^{-c\delta}),
\end{equation}
uniformly in $k \in \Z^d_L$ and $t\in [0,1]$ for some $c>0$. Before we start, let us mention an elementary observation that will be useful for us:

\begin{claim}\label{VIP_claim} With $\Gamma_{-}$ defined as in \eqref{eq:Gammas} and with $0\leq r\leq 1$ and $\nu\leq 1$, there holds that $\Gamma_{-}^2+(\nu^{-1} \Omega)^2 \geq \Gamma_{-}^2+\Omega^2 \geq 1$. Moreover, if $\Omega=0$ then $\Gamma_-\geq 1$.
\end{claim}
\begin{proof}[Proof of Claim]
To see this, suppose that $|\Omega| \leq 1$, and we will notice that his implies that $\Gamma_-\geq 1$. In fact, this follows by noticing that since $|k|_{\zeta}^2\leq |k_1|_{\zeta}^2-|k_2|_{\zeta}^2+|k_3|_{\zeta}^2+1$, we have
\[
(1+|k|_{\zeta}^2)^r\leq (1+|k_1|_{\zeta}^2+|k_3|_{\zeta}^2+1)^r\leq (1+|k_1|_{\zeta}^2)^r+(1+|k_3|_{\zeta}^2)^r\leq \sum_{j=1}^3 (1+|k_j|_{\zeta}^2)^r-1.
\]
\end{proof}

Going back to proving \eqref{SigmaAim}, we start by performing a smooth partition of unity to localize $\nu^{-1}\Omega$ in intervals of size $\varepsilon:=L^{-c_0\delta}$ for an adequately chosen constant $0<c_0\ll 1$ to be specified later. In particular we let $\chi\in C_c^{\infty}(\R)$ with $\mbox{supp}(\chi)\subseteq [-1,1]$ and such that $\sum_{n\in\Z} \chi(x-n)=1$. We then have that:
\begin{align}
\Sigma(k,t)=&\ \sum_{n\in \Z}L^{-2d}\sum_{k=k_1-k_2+k_3} \chi\left(\frac{\nu^{-1}\Omega}{\varepsilon}-n\right)\frac{\nu^{-1}\Gamma_{-}}{\Gamma_{-}^2+(\nu^{-1}\Omega)^2}\, \, W(k_1, k_2, k_3) \nonumber \\
=&\ \sum_{n\in \Z}L^{-2d}\sum_{k=k_1-k_2+k_3} \chi\left(\frac{\nu^{-1}\Omega}{\varepsilon}-n\right)\frac{\nu^{-1}\Gamma_{-}}{\Gamma_{-}^2+(\varepsilon n)^2}\, \, W(k_1, k_2, k_3) \label{eq:SigmaAim_main} \\
&-\sum_{n\in \Z}L^{-2d}\sum_{k=k_1-k_2+k_3} \chi\left(\frac{\nu^{-1}\Omega}{\varepsilon}-n\right)\,\left[\frac{\nu^{-1}\Gamma_{-}}{\Gamma_{-}^2+(\varepsilon n)^2}\,-\frac{\nu^{-1}\Gamma_{-}}{\Gamma_{-}^2+(\nu^{-1}\Omega)^2}\,\right] \, W(k_1, k_2, k_3). \label{eq:SigmaAim_error}
\end{align}

Note that 
\begin{equation}\label{MVT0}
\Big |\frac{\nu^{-1} \Gamma_{-}}{\Gamma_{-}^2+(\varepsilon n)^2}\,-\frac{\nu^{-1}\Gamma_{-}}{\Gamma_{-}^2+(\nu^{-1}\Omega)^2}\Big | 
\lesssim  \varepsilon \, \nu^{-1}\, \frac{\varepsilon n \Gamma_{-}}{[\Gamma_{-}^2+(\varepsilon n)^2]^2}   \lesssim  \frac{\varepsilon\nu^{-1}}{\Gamma_{-}^2+(\varepsilon n)^2} \lesssim \varepsilon\, \nu^{-1}\, \langle\varepsilon n\rangle^{-2}
\end{equation}
by \Cref{VIP_claim}. As a result, we may bound 
\[
| \eqref{eq:SigmaAim_error} | \leq
\sum_{n\in \Z}  \frac{\, \varepsilon\, \nu^{-1}}{\langle\varepsilon n\rangle^{2}}\, L^{-2d}\, \sum_{k=k_1-k_2+k_3} \Big | \chi\left(\frac{\nu^{-1}\Omega}{\varepsilon}-n\right)\,\, W(k_1, k_2, k_3) \Big | \lesssim
\sum_{n\in \Z}  \frac{ \varepsilon\, \nu^{-1}}{\langle \varepsilon n\rangle^{2}}\, \, \varepsilon\,\nu\lesssim \varepsilon
\]
the second inequality following from \Cref{NTDH}. Let us now consider the term \eqref{eq:SigmaAim_main}. We write 
\[
\eqref{eq:SigmaAim_main}=\sum_{n\in \Z}L^{-2d}\, \nu^{-1}\, \sum_{k=k_1-k_2+k_3} \chi\left(\frac{\nu^{-1}\Omega}{\varepsilon}-n\right) \, \widetilde{W}_n(k_1, k_2, k_3)
\]
with 
\[
\widetilde{W}_n(k_1, k_2, k_3)=\Gamma_{-} \, (\Gamma_{-}^2+ \varepsilon^2 n^2)^{-1} \, W(k_1, k_2, k_3).
\]
By Lemma \ref{NTDH}, and using that $\Gamma_{-}^2+ \varepsilon^2 n^2 \gtrsim \langle \varepsilon n\rangle^2$, we have
\[
\begin{split}
\eqref{eq:SigmaAim_main}=&\ \sum_{n} \, \nu^{-1}\, \int_{k_1-k_2+k_3=k}  \chi\left(\frac{\nu^{-1}\Omega}{\varepsilon}-n\right) \, \widetilde{W}_n (k_1, k_2, k_3) \, dk_1 dk_3  + \O(  \, L^{-(c-c_0)\delta})
\end{split}
\]
which is sufficient with a well-chosen $c_0<c$. 

In order to obtain \eqref{SigmaAim}, we argue as in \eqref{MVT0} again 
\begin{align*}
& \sum_{n\in\N} \, \nu^{-1}\, \int_{k_1-k_2+k_3=k}  \chi\left(\frac{\nu^{-1}\Omega}{\varepsilon}-n\right) \, \widetilde{W}_n (k_1, k_2, k_3) \, dk_1 dk_3 \\
 &= \sum_{n\in\Z}  \int_{k_1-k_2+k_3=k}  \chi\left(\frac{\nu^{-1}\Omega}{\varepsilon}-n\right) \, \frac{\nu^{-1}\Gamma_{-}}{\Gamma_{-}^2+(\nu^{-1}\Omega)^2} \, W(k_1, k_2, k_3) \, dk_1 dk_3 \\
& \qquad + \sum_{n\in\Z} \int_{k_1-k_2+k_3=k}  \chi\left(\frac{\nu^{-1}\Omega}{\varepsilon}-n\right) \, \left[\frac{\nu^{-1}\Gamma_{-}}{\Gamma_{-}^2+\varepsilon^2 n^2}-\frac{\nu^{-1}\Gamma_{-}}{\Gamma_{-}^2+(\nu^{-1}\Omega)^2}\right] \, W(k_1, k_2, k_3) \, dk_1 dk_3 \\
& = \int_{k_1-k_2+k_3=k}  \left[\sum_{n\in\Z}  \chi\left(\frac{\nu^{-1}\Omega}{\varepsilon}-n\right)\right] \, \frac{\nu^{-1}\Gamma_{-}}{\Gamma_{-}^2+(\nu^{-1}\Omega)^2} \, W(k_1, k_2, k_3) \, dk_1 dk_3 +\O(\varepsilon) \\
  &= \int_{k_1-k_2+k_3=k}  \frac{\nu^{-1}\Gamma_{-}}{\Gamma_{-}^2+(\nu^{-1}\Omega)^2} \, W(k_1, k_2, k_3) \, dk_1 dk_3 +  \O (L^{-c_0\delta}) ,
\end{align*}
where we bound the error terms with same arguments as previously used after \eqref{MVT0}. This completes the proof of \eqref{SigmaAim} in the case of $\Sigma_1$.

A similar strategy allows us to treat $\Sigma_2$, but one must perform the cut-off at the level of the variable $T_{\mathrm{kin}}\Omega$ instead of $\nu^{-1}\Omega$, i.e.
\[
\begin{split}
 \Sigma_2(k,t, \tau) & =  \sum_{n\in \Z} L^{-2d}\sum_{k=k_1-k_2+k_3}  \chi\left(\frac{T_{\mathrm{kin}}\,\Omega}{\varepsilon}-n\right) \mbox{Re}\left( \frac{\nu^{-1} e^{iT_{\mathrm{kin}} \Omega^{13}_{2k}\tau}}{\Gamma_{-}-i\nu^{-1} \Omega^{13}_{2k}}\right)\, W(k_1,k_2,k_3; t,\tau)\\
W(k_1,k_2,k_3; t,\tau) & = e^{-\varrho \Gamma_{+} \tau}\,\prod_{j=1}^{3}  f_j(k_j,t).
\end{split}
\]

One may then substitute \eqref{MVT0} by
\begin{equation}\label{MVT1}
\Big |\frac{\nu^{-1}\, e^{i\varepsilon n \tau }}{\Gamma_{-}-i\varepsilon \varrho^{-1} n}-\frac{\nu^{-1}\, e^{iT_{\mathrm{kin}}\Omega \tau}}{\Gamma_{-}-i\nu^{-1}\Omega}\Big | 
\lesssim \varepsilon\, T_{\mathrm{kin}}\, \frac{\tau}{\sqrt{\varrho^2 \Gamma_{-}^2+(\varepsilon n)^2}}\lesssim \varepsilon\, T_{\mathrm{kin}}\, \langle\varepsilon n  \rangle^{-1}
\end{equation}
in view of \Cref{VIP_claim}.

In comparison to the argument for $\Sigma$ above, this only leads to logarithmic losses in $L$, which can be absorbed by the the $L^{-c\delta}$ gains. To see this, we just need to note that the sum $n$ can be restricted to the range $|n|\leq L^{100d}$, since the sum in $k$ can be restricted from the start to the range where $|k_j|\leq L$ thanks to the Schwartz decay of $W$.
\end{proof}

\subsection{Main terms: convergence to kinetic kernel}

We are ready to study the asymptotic behavior of the integrals in \eqref{eq:Ss}.

\begin{thm}\label{thm:integral_to_delta} Consider the integrals $\Ss_1$ and $\Ss_2$ in \eqref{eq:Ss} for $\varrho\gtrsim 1$. Then there exists a small $\theta\in (0,1)$  (which is independent of $t$ and $\tau$) such that the following holds as $\nu\rightarrow 0$
\begin{equation}\label{eq:approx_kernel}
\begin{split}
\Ss_1(k,t) & = \Kc_1 (f_1,f_2,f_3) + \O \left(\nu^{\theta}\right) \qquad t\in [0,1]\\
\Ss_2 (k,t,\tau)   & = \begin{cases}
\O \left( \nu^{\theta} \right) & \tau \in [\varrho^{-1}\nu^{\theta} ,1], \ t\in [0,1],\\
\O(|\log \nu|) & \tau \in (0,\varrho^{-1}\nu^{\theta}),\  t\in [0,1].
\end{cases}
\end{split}
\end{equation}
with bounds independent of $k$, $\tau$ and $t$, where
\begin{equation}\label{eq:kinetickernel1}
\mathcal{K}_1 \left(\phi_1,\phi_2,\phi_3\right)(k) = \int_{k=k_1-k_2+k_3} 4\pi\,\delta_{\R} (|k_1|_{\zeta}^2 - |k_2|_{\zeta}^2 +|k_3|_{\zeta}^2 -|k|_{\zeta}^2) \, \prod_{j=1}^{3} \phi_j(k_j) \,dk_1\, dk_3 .
\end{equation}
\end{thm}

\begin{rk} This result implies that $I_1$ and $I_2$ in \eqref{eq:fromItoSigma} are well-approximated by
\[
\begin{split}
I_1(k,t)  & = \int_0^t e^{-2\varrho\gamma_k (t-t')}  \Kc_1 (f,f,f)(t')\, dt' + \O(L^{-\theta}), \qquad t\in [0,1],\\
I_2 (k,t) & = \int_0^t \Ss_2 (k,t-\tau,\tau)\, d\tau+ \O(L^{-\theta}t) = \int_0^{\varrho^{-1}\,\nu^{\theta}} \O( |\log \nu|) \, d\tau + \O ( \nu^{\theta}) + O(L^{-\theta}t ) \\
& = \O ( \nu^{\theta/2}) + O(L^{-\theta}t ).
\end{split}
\]
\end{rk}
\begin{proof} 
We fix $t\in [0,1]$ and rewrite $\Ss_1$ in \eqref{eq:Ss} as follows:
\[
\Ss_1(k,t) = \int_{\R^{2d}} \nu^{-1} h_0(\nu^{-1} \Omega,\Gamma_{-}) \, f (k_1, t) \, f(k_1+k_3-k, t) \, f(k_3,t)\, dk_1 dk_3
\]
where $\Gamma_{-}$ was defined in \eqref{eq:Gammas} and
\begin{equation}\label{eq:first_h}
h_0(x,y)  = \frac{4y}{y^2+x^2}.
\end{equation}

Note that we may rewrite $\Omega$ as $\Omega = 2\,(k_1-k)\cdot(k-k_3)$. By the change variables $k_1 \mapsto \sqrt{2} (k_1-k)$ and $k_3\mapsto \sqrt{2} (k-k_3)$, we have that 
\[
\Ss_1(k,t) = \int_{\R^{2d}} \nu^{-1} h_0(\nu^{-1} k_1\cdot k_3,y) \, \varphi(k_1,k_3;t)\, dk_1 dk_3
\]
where 
\[
\begin{split}
y& = -\gamma_k+\gamma_{\frac{k_1}{\sqrt{2}} +k} + \gamma_{\frac{k_1-k_3}{\sqrt{2}} +k} + \gamma_{k-\frac{k_3}{\sqrt{2}}}\\
\varphi(k_1,k_3)& =\varphi(k_1,k_3;t)= f \left(\frac{k_1}{\sqrt{2}}+k, t\right) \, f\left(\frac{k_1-k_3}{\sqrt{2}}+k, t\right)\, f\left(k-\frac{k_3}{\sqrt{2}},t\right).
\end{split}
\]
Note that $\varphi$ is a Schwartz function of $k_1,k_3$ uniformly in $t$, and that $\Omega=k_1\cdot k_3$ in these new variables. Similarly, note that
\begin{equation}\label{eq:approx_id_Ss2}
4\nu^{-1}\,\mbox{Re}\left( \frac{e^{i\tau T_{\mathrm{kin}} \, \Omega}}{y-i\nu^{-1}\Omega}\right)  = \nu^{-1} h_1(\nu^{-1}\Omega,y) - \nu^{-1} h_2(\nu^{-1}\Omega,y),
\end{equation}
where
\begin{equation}\label{eq:h12}
h_1(x,y) = \frac{4y\,\cos(\varrho \tau x)}{y^2+x^2}\qquad \qquad h_2(x,y) = \frac{4x\,\sin(\varrho \tau x)}{y^2+x^2}.
\end{equation}
Therefore
\[
\Ss_2(k,t,\tau) = \int_{\R^{2d}} [\nu^{-1} h_1(\nu^{-1} \Omega,y)-\nu^{-1} h_2(\nu^{-1} \Omega,y)]  \, \varphi(k_1,k_3;t,\tau)\, dk_1 dk_3,
\]
where\footnote{
Note that $\varphi(k_1,k_3;t,\tau)\in \Sc (\R^d_{k_1}\times \R^d_{k_3})$ \emph{uniformly} in $t,\tau$ and $\varrho$. Indeed, we claim that $\pa_{k_{ij}}^n e^{-\varrho (y+\gamma_k) \tau} = \O_n( 1)$ for any $n\in\N_0$, $i=1,3$, $j=1,\ldots ,d$. This follows from the following estimates:
\[
\pa_{k_{ij}}^n e^{-\varrho \gamma_{k_{\ell}} \, \tau} = \O_n( \sum_{\alpha=0}^n (\varrho\tau)^{\alpha} \, \gamma_{k_{\ell}} \, e^{-\varrho \gamma_{k_{\ell}} \, \tau}) = \O_n( 1) ,
\]
which in turn follow from the fact that $z^n \, e^{-z}=\O_n(1)$ for $n\geq 0$, from $\pa_{k_{ij}}^n \gamma_{k_\ell} \lesssim_n \gamma_{k_\ell}$ for any $n\in\N_0$, $i=1,3$, $j=1,\ldots ,d$, and from $1\leq \gamma_{k_\ell}$. As a result we may treat $\varphi(k_1,k_3;t,\tau)$ as $\varphi(k_1,k_3)$, which we do from now on.
}
\[
\varphi(k_1,k_3;t,\tau)=e^{-\varrho\, (y+2\gamma_k) \, \tau}\,f \left(\frac{k_1}{\sqrt{2}}+k, t\right) \, f\left(\frac{k_1-k_3}{\sqrt{2}}+k, t\right)\, f\left(k-\frac{k_3}{\sqrt{2}},t\right).
\]

Using the inverse Fourier transform\footnote{Since the focus from here on is to estimate error terms which vanish in the limit, we shall remove the factor $2\pi$ from the Fourier transform in order to shorten formulas.}, we may rewrite these expressions as:
\begin{equation}\label{eq:new_S}
\begin{split}
\Ss_1(k,t) & =\int_{\R^{2d}} \int_{\R} \widehat{h_0}(\nu \xi,y)\, e^{i\xi\Omega}\, \varphi(k_1,k_3)\,  d\xi\, dk_1 dk_3\\
\Ss_2(k,t,\tau) & = \int_{\R^{2d}} \int_{\R} [\widehat{h_1}(\nu \xi,y)-\widehat{h_2}(\nu \xi,y)] \, e^{i\xi\Omega}\, \varphi(k_1,k_3;t,\tau)\,  d\xi\, dk_1 dk_3.
\end{split}
\end{equation}
We note that 
\begin{equation}\label{eq:h}
\begin{split}
\widehat{h_0}(\xi,y) & =  4\pi \, e^{-|\xi y|}\\
\widehat{h_1}(\xi,y) & = 2\pi\, e^{-|y|\,|\xi -\varrho \tau|} +2\pi\, e^{-|y|\, |\xi +\varrho \tau|}\\
\widehat{h_2}(\xi,y) & = -2\pi\, e^{-|y| \, |\xi -\varrho \tau|}  \, \mbox{sign}(y \xi -y\varrho \tau ) + 2\pi\, e^{-|y|\, |\xi +\varrho \tau|}\, \mbox{sign}(y\xi+y\varrho \tau ) \qquad y\neq 0,\\
\widehat{h_2}(\xi,0) & = 4\pi\,\mathbbm{1}_{[-1,1]}\left(\frac{\xi}{\varrho \tau}\right).
\end{split}
\end{equation}

We introduce a cut-off $\chi_0\in C_c^{\infty}(\R)$ such that $\mbox{supp}(\chi_0)\subset [-2,2]$, $\chi_0\geq 0$, $\chi_0=1$ in $[-1,1]$, and we let $\chi_1= 1-\chi_0$. We also add and subtract the top order behavior (a delta function).
\begin{align}
\Ss_1(k,t) =&\ \int_{\R^{2d}} \delta (\Omega)\, \widehat{h_0}(0,y)\, \varphi(k_1,k_3)\, dk_1\, dk_3 \label{eq:S1_delta}\\
& +\int_{\R^{2d}} \int_{\R} [\widehat{h_0}(\nu \xi,y)-\widehat{h_0}(0,y)]\, e^{i\xi\Omega}\, \varphi(k_1,k_3)\,d\xi\, dk_1 dk_3. \label{eq:S1_error}
\end{align}

A similar decomposition can be done for $\Ss_2$. Note, however, that for any $y\geq 0$
\begin{equation}
\widehat{h_1}(0,y) = \widehat{h_2}(0,y)= 4\pi \, e^{-\varrho \tau\, y},
\end{equation}
and therefore \emph{the delta function vanishes for positive values of $y$} in the case of $\Ss_2$. Therefore:
\begin{align}
\Ss_2(k,t,\tau) =&\  \int_{\R^{2d}} \delta (\Omega)\, [\widehat{h_1}(0,y)-\widehat{h_2}(0,y)] \, \mathbbm{1}_{y<0} \,\varphi(k_1,k_3)\, dk_1\, dk_3 \label{eq:S2_delta}\\
& +  \int_{\R^{2d}} \int_{\R} \left[ \widehat{h_1}(\nu \xi,y)-\widehat{h_1}(0,y) \right] \, e^{i\xi\Omega}\, \varphi(k_1,k_3)\,d\xi\, dk_1 dk_3 \label{eq:S2_error1}\\
& -  \int_{\R^{2d}} \int_{\R} \left[ \widehat{h_2}(\nu \xi,y)-\widehat{h_2}(0,y) \right] \, e^{i\xi\Omega}\, \varphi(k_1,k_3)\,d\xi\, dk_1 dk_3 \label{eq:S2_error2}.
\end{align}

Moreover, we proved in \Cref{VIP_claim} that $y\geq 1$ whenever $|\Omega|\leq 1$. As a result, $\{\Omega=0\}\cap \{y<0\} =\emptyset$ and thus \emph{the right-hand side of \eqref{eq:S2_delta} vanishes}.

Over the next steps, we will estimate the terms  \eqref{eq:S1_error}, \eqref{eq:S2_error1} and \eqref{eq:S2_error2}, which will finish the proof of \Cref{thm:integral_to_delta}. In order to do so, we will work with a general kernel $h$ which satisfies some common properties to all $h_0,h_1,h_2$. We introduce the following:

\begin{defn}\label{def:accept_kernel} A kernel $h(x,y)$ is an \emph{acceptable kernel} if its Fourier transform in $x$ satisfies the following properties for some $\theta\in [0,1)$ as small as desired.
\begin{enumerate}[label=\textbf{(K.\arabic*)}]
\item \label{k:diff} $\widehat{h}(\xi,y)$ is twice differentiable in both variables except on the lines $(\xi,0)$ and $\{(\xi_j, y)\}_{j=1,\ldots, n}$ for some finite number $n\in\N_0$. Let $\Ac(h) = \cup_{j=1}^{n} [\xi_j-\nu^{\theta} ,\xi_j+\nu^{\theta}]$.
\item \label{k:bounded} $\sup_{\xi} | \widehat{h}(\xi, y)|\lesssim 1$ for all $(\xi, y)$.
\item \label{k:lips} $\Big | \widehat{h}(\nu \xi, y) - \widehat{h}(0, y) \Big | \lesssim
 \begin{cases}
 \nu\,|y| \, |\xi| & \mbox{if}\ |\xi|\leq \nu^{-\theta},\\
\nu^{1-\theta} \langle y\rangle \, |\xi|& \mbox{if}\  |\xi|>\nu^{-\theta}.
\end{cases}$
\item \label{k:lips_xi} $\Big | \frac{\pa^j}{\pa \xi^j} \widehat{h}(\nu \xi, y ) \Big | \lesssim (\nu\, |y|)^j $ for $j=1,2$, $\xi\in \R-\nu^{-1} \Ac(h)$.
\item \label{k:int_xi} $\int_{\R-\nu^{-1} \Ac (h)} \Big | (\pa_{\xi}^j \widehat{h}) (\nu \xi, y) \Big | \, d\xi \lesssim (\nu |y|)^{j-1}$ for $j=1,2$.
\item \label{k:lips_y}  $\Big | \frac{\pa}{\pa y} [\widehat{h}(\nu \xi, y) - \widehat{h}(0, y) ]\Big | \lesssim \nu\, |\xi|$, whenever $y\neq 0$ and $\xi\in \R- \nu^{-1} \Ac(h)$.
\item \label{k:y0}For all $j=1,2$,
 \[
|\pa_{y}^{j} \widehat{h}(0, y) | \lesssim |y|^{-j}
 \]
 \item \label{k:R0} There exists some $R_0\geq 0$ (which may depend on $\nu$, $T_{\mathrm{kin}}$ and $\tau$) such that $\Ac(h)\subset [-\nu R_0,\nu R_0]$ and for any $R\geq R_0$ any $j=0,1,2$ and any $y\neq 0$, we have that  
\begin{equation}\label{k:R_large}
\int_{|\xi|>R}  |(\pa_{y}^{j} \widehat{h})(\nu \xi, y) | \, \frac{d\xi}{|\xi|^j} \lesssim  |y|^{-1} \,  \nu^{j-1} \, e^{- \nu |y| (R-R_0)}\ .
\end{equation}
Moreover, for $j=1,2$ and $y\neq 0$,
\begin{equation}\label{k:R_small}
\int_{\{\nu^{\theta} <|\xi|<R_0\}-\nu^{-1}\Ac(h)}  |\pa_{y}^{j}[\widehat{h}(\nu \xi, y)- \widehat{h}(0, y) ]|\, \frac{d\xi}{|\xi|^2} \lesssim \max\{ 1, |y|^{1-j}\} \, \nu \,\log (\nu^{-\theta} \langle R_0\rangle) \ .
\end{equation}
\end{enumerate}
\end{defn}

\begin{claim}\label{thm:accept_kernel} The kernels $h_0,h_1$, defined in \eqref{eq:first_h} and \eqref{eq:h12}, satisfy \ref{k:diff}-\ref{k:R0}. The kernel $h_2$ in \eqref{eq:h12} also satisfies \ref{k:diff}-\ref{k:R0} provided $\varrho\tau \geq \nu^{\theta}$. When $\varrho\tau < \nu^{\theta}$, $h_2$ satisfies \ref{k:diff}-\ref{k:R0} except for \ref{k:lips}.
In the case of $\widehat{h_0}$, the only singular lines from \ref{k:diff} are $\xi=0$ and $y=0$, and we may choose $R_0=\nu^{-1+\theta}$ (note that \eqref{k:R_large} is holds with $R_0=0$ but we must make $R_0$ large enough so that $\Ac(h_0)=[-\nu^{\theta},\nu^{\theta}]\subset [-\nu R_0,\nu R_0]$). In the case of $\widehat{h_1}$ and $\widehat{h_2}$, we remove the singular lines $y=0$, and $\xi=0,\pm \varrho \tau$, with $R_0 = \max\{ T_{\mathrm{kin}}\,\tau,\nu^{-1}\}$.
\end{claim}

We shall first show that for any acceptable kernel $h$ with $R_0\leq \exp(\nu^{-1+})$, there holds that 
\begin{equation}\label{sec4aim}
\left|\int_{\R^{2d}} \int_{\R} [\widehat{h}(\nu \xi,y)-\widehat{h}(0,y)]\, e^{i\xi\Omega}\, \varphi(k_1,k_3)\,d\xi\, dk_1 dk_3\right|\lesssim \nu^\theta.
\end{equation}
This will cover the kernels $h_0$ and $h_1$, as well as $h_2$ in the case when $\varrho \tau\geq \nu^\theta$. The remaining case of $h_2$ in the regime $\varrho \tau\leq \nu^\theta$ will be easily treated afterwards. Finally, we shall give the proof of \Cref{thm:accept_kernel}.

\medskip

To prove \eqref{sec4aim} for any admissible kernel $h$, it suffices to consider the case when $n=1$ in \ref{k:diff} and $\xi_1=0$. The cases when $n >1$ are treated in exactly the same way. 
As such, we decompose:
\begin{equation}\label{eq:isolated_points}
\widehat{h}(\nu\xi,y) : = \widehat{h}(\nu\xi,y)  \chi_0 (\nu^{-\theta} \xi)  + \widehat{h}(\nu\xi,y)  \chi_1 (\nu^{-\theta} \xi) \chi_0 (\nu^{-\theta} y) + \widehat{h}(\nu\xi,y)  \chi_1 (\nu^{-\theta} \xi)  \chi_1 (\nu^{-\theta} y)
 \end{equation}
for a fixed $0<\theta\ll 1$ as small as desired. These decomposition may be adapted to deal with any finite number of singular points $\xi_j$ of $\widehat{h}$ by adding adequate cutoff functions in addition to $ \chi_0 (\nu^{-\theta}  \xi)$ such as $ \chi_0 (\nu^{-\theta}  (\xi-\nu^{-1}\xi_j))$.

We will show that the contribution of the first two summands is negligible, and thus we may focus on the kernel $\widehat{h}(\xi,y) \chi_1 (\nu^{-\theta}  \xi) \chi_1 (\nu^{-\theta}  y) $.

\medskip 

\noindent  {\bf Step 0. Integration around singularities.}

Consider the contribution of the first summand in \eqref{eq:isolated_points}. Using \ref{k:bounded} and the fast decay of $\varphi$, we have that:
\[
\Big | \int_{\R^{2d}} \int_{\R} [\widehat{h}(\nu\xi,y)-\widehat{h}(0,y)]\, \chi_0 (\nu^{-\theta} \xi) \, e^{i\xi\Omega}\, \varphi(k_1,k_3)\, d\xi\, dk_1 dk_3 \Big | \lesssim \nu^{\theta}.
\]

Consider the second summand in \eqref{eq:isolated_points}. We integrate by parts $e^{i\Omega\xi}$ twice in $\xi$, which yields:
\begin{align}
\int_{\R^{2d}} \int_{\R} & [\widehat{h}(\nu\xi,y)-\widehat{h}(0,y)]\, \chi_1 (\nu^{-\theta} \xi)\, \chi_0 (\nu^{-\theta} y)\, e^{i\xi\Omega}\, \varphi(k_1,k_3)\, d\xi\, dk_1 dk_3 \nonumber \\
= &\ \int_{\R^{2d}} \int_{\R} \frac{1}{(i\Omega)^2} \, e^{i\xi\Omega}   \frac{\pa^2}{\pa \xi^2}\left[  \left( \widehat{h}(\nu\xi,y)-\widehat{h}(0,y) \right)\, \chi_1 (\nu^{-\theta} \xi)  \right] \, \varphi(k_1,k_3)\, \chi_0 (\nu^{-\theta}\, y)\, d\xi\, dk_1 dk_3. \label{eq:y_small_1}
\end{align}
\begin{itemize}
\item Suppose first that no derivative hits $\chi_1 (\nu^{-\theta} \xi)$. Using \ref{k:int_xi}, we may bound this contribution by
\begin{equation}\label{eq:y_small_11}
 \int_{\R^{2d}}\frac{\nu |y|}{\Omega^2} \, |\varphi(k_1,k_3)|\, |\chi_0 (\nu^{-\theta}\, y)|\, dk_1 dk_3 \lesssim \nu^{\theta} \, \int_{\R^{2d}}\frac{\nu}{\Omega^2} \, |\varphi(k_1,k_3)|\, |\chi_0 (\nu^{-\theta}\, y)|\, dk_1 dk_3 .
\end{equation}
By \Cref{VIP_claim}, i.e. the fact\footnote{As a matter of fact, \Cref{VIP_claim} yields the stronger bound $\Omega\gtrsim 1$, but we do not need such a strong result here.} that $y^2 + (\nu^{-1}\Omega)^2\geq 1$, we must have that $\Omega\gtrsim \nu$ in this region, and thus we change variables in order to integrate in $\Omega$. Given that $\Omega=k_1\cdot k_3$, we divide $\R^{2d}$ into regions $\Rc_{ij}$ where each variable $|k_{ij}|$, $i=1,3$, $j=1,\ldots,d$, is maximal (up to a multiple) than the remaining $2d-1$ variables. We do so via a smooth partition of unity, i.e. 
\[
1 =\sum_{i=1,3} \sum_{j=1,\ldots,d} \chi_{\Rc_{ij}}.
\]

Without loss of generality let us assume that $|k_{31}|$ is maximal. We single out $k_{11}$, define $k_1'=(k_{12},\ldots,k_{1d})$, and make the change of variables $k_{11} \mapsto \Omega$. The Jacobian is proportional to $k_{31}^{-1}$, which is locally integrable in $\R^{2d-1}$. As a result, we have that 
\[
|\eqref{eq:y_small_11}|\lesssim \nu^{1+\theta} \, \int_{\R^{2d-1}} |k_{31}|^{-1}\, \int_{|\Omega|>\nu}  |\Omega|^{-2} \, |\varphi ( k_1 (\Omega),k_1',k_3)|\, d\Omega dk_1' dk_3  = \O ( \nu^{\theta}).
\]

\item Back to \eqref{eq:y_small_1}, if a single derivative hits $\chi_1$, we note that $\pa_{\xi}  \chi_1 (\nu^{-\theta} \xi) = - \nu^{-\theta} \chi_0' (\nu^{-\theta} \xi)$. Then \ref{k:lips_xi} yields:
\[
\Big | \int_{\R^{2d}} \int_{\R}  \frac{1}{|\Omega|^2} \, \nu^{1-\theta} |y|\, |\chi_0' (\nu^{-\theta} \xi)| \, |\varphi(k_1,k_3)|\, \chi_0 (\nu^{-\theta}\, y)\, d\xi\, dk_1 dk_3\Big |  \leq \int_{\R^{2d}} \frac{1}{|\Omega|^2} \, \nu^{1+\theta} \, |\varphi(k_1,k_3)|\, \chi_0 (\nu^{-\theta}\, y)\, dk_1 dk_3 
\]
and the same integration in $\Omega$ as for \eqref{eq:y_small_11} yields a contribution of size $\O( \nu^{\theta})$.

\item A similar strategy based on \ref{k:lips} works if two derivatives hit $\chi_1$ in \eqref{eq:y_small_1}. Indeed,
\begin{align*}
\Big | \int_{\R^{2d}} \int_{\R} & \frac{1}{(i\Omega)^2} \, e^{i\xi\Omega}    \left( \widehat{h}(\nu\xi,y)-\widehat{h}(0,y) \right)\, \nu^{-2\theta} \, \chi_0'' (\nu^{-\theta} \xi)  \, \varphi(k_1,k_3)\, \chi_0 (\nu^{-\theta}\, y)\, d\xi\, dk_1 dk_3\Big | \\
& \lesssim \int_{\R^{2d}} \frac{1}{\Omega^2} \, \int_{|\xi|\leq \nu^{\theta}} \nu^{1-2\theta}\, |y|\,  |\xi| \, |\varphi(k_1,k_3)|\, \chi_0 (\nu^{-\theta}\, y)\, d\xi\, dk_1 dk_3\\
& \lesssim \int_{\R^{2d}} \frac{1}{\Omega^2} \, \nu^{1+\theta}\, |\varphi(k_1,k_3)|\,\chi_0 (\nu^{-\theta}\, y)\, d\xi\, dk_1 dk_3\lesssim \nu^{\theta},
\end{align*}
the last step coming from integration over $|\Omega|>\nu$.
\end{itemize}

By \eqref{eq:isolated_points} and the above simplifications,
\[
\eqref{eq:S1_error}  = \int_{\R^{2d}} \int_{\R} [\widehat{h}(\nu \xi,y)-\widehat{h}(0,y)]\, e^{i\xi\Omega}\, \varphi(k_1,k_3)\,  \chi_1 (\nu^{-\theta} \xi) \, \chi_1 (\nu^{-\theta} y)\, d\xi\, dk_1 dk_3 + \O ( \nu^{\theta}).
\]

\medskip

\noindent  {\bf Step 1. Localization around $\Omega=0$.} We write,
\begin{align}
\eqref{eq:S1_error}  = &\  \int_{\R^{2d}} \int_{\R} [\widehat{h}(\nu \xi,y)-\widehat{h}(0,y)]\, e^{i\xi\Omega}\, \varphi(k_1,k_3)\,  \chi_1 (\nu^{-\theta} \xi) \, \chi_1 (\nu^{-\theta} y)\, \chi_0 (\nu^{-\beta} \Omega) \, d\xi\, dk_1 dk_3 \nonumber \\
& + \int_{\R^{2d}} \int_{\R} [\widehat{h}(\nu \xi,y)-\widehat{h}(0,y)]\, e^{i\xi\Omega}\, \varphi(k_1,k_3)\,  \chi_1 (\nu^{-\theta} \xi) \, \chi_1 (\nu^{-\theta} y)\, \chi_1 (\nu^{-\beta} \Omega) \, d\xi\, dk_1 dk_3 + \O ( \nu^{\theta} ) \label{eq:large_omega} 
\end{align}
for $\beta\in (0,1)$ to be fixed later. As in Step 0, we integrate $e^{i\xi\Omega}$ by parts twice in the $\xi$ variable
\[
\begin{split}
\eqref{eq:large_omega} 
=&\ \int_{\R^{2d}} \int_{\R} \frac{e^{i\xi\Omega}}{(i\Omega)^2}  \frac{\pa^2}{\pa \xi^2} \left[ ( \widehat{h}(\nu\xi,y)-\widehat{h}(0,y)) \, [1- \chi_0 (\nu^{-\theta} \xi) ]\right] \, \varphi(k_1,k_3)\, \chi_1 (\nu^{-\theta}\, y)\, \chi_1 (\nu^{-\beta} \Omega) \, d\xi\, dk_1 dk_3.
\end{split}
\]

If both derivatives hit  $\widehat{h}(\nu\xi,y)-\widehat{h}(0,y)$, then we use \ref{k:int_xi} and proceed as in Step 0 to obtain a contribution of size $\O (\nu^{1-\beta})$. If both derivatives hit $\chi_0 (\nu^{-\theta} \xi)$, we use the rapid decay of $\varphi$ to reduce the integral to $|(k_1,k_3)|\leq \nu^{-\theta}$, which implies that $|y|\leq \nu^{-2\theta}$. Then \ref{k:lips} yields
\begin{align*}
& \nu^{-2\theta}\, \Big | \int_{\R^{2d}} \int_{\R} \frac{1}{(i\Omega)^2} \, e^{i\xi\Omega} \, [ \widehat{h}(\nu\xi,y)-\widehat{h}(0,y)] \, \chi_0'' (\nu^{-\theta} \xi) \, \varphi(k_1,k_3)\, \chi_1 (\nu^{-\theta}\, y)\, \chi_1(\nu^{-\beta} \Omega) \, d\xi\, dk_1 dk_3 \Big | \\
&\ \lesssim  \int_{\R^{2d}} \frac{1}{\Omega^2}\, \nu^{1-3\theta}\, |\chi_1 (\nu^{-\theta}\, y)|\, |\chi_1 (\nu^{-\beta} \Omega)| \, dk_1\, dk_3 = \O(\nu^{1-\beta-3\theta} ),
\end{align*}
which is small by imposing $4\theta <1-\beta$.

The case where one derivative hits each term follows from an analogous argument and \ref{k:lips_xi}. All in all, we have shown that $\eqref{eq:large_omega} = \O ( \nu^{\theta})$.

\medskip

\noindent  {\bf Step 2. Localization in $\xi$.}  Our goal is to show that 
\begin{equation}\label{eq:S1_error_new}
\int_{\R^{2d}} \int_{\R} [\widehat{h}(\nu \xi,y)-\widehat{h}(0,y)]\, e^{i\xi\Omega}\, \varphi(k_1,k_3)\,  \chi_1 (\nu^{-\theta} \xi) \, \chi_1 (\nu^{-\theta} y)\, \chi_0 (\nu^{-\beta} \Omega) \, d\xi\, dk_1 dk_3  
= \O (\nu^{\theta}).
\end{equation}
In order to prove \eqref{eq:S1_error_new}, we add an additional cut-off in the $\xi$-variable. More precisely, we write:
\begin{align}
\eqref{eq:S1_error_new}=&\ \int_{\R^{2d}} \int_{\R} (\widehat{h}(\nu \xi, y) - \widehat{h}(0, y) )\, e^{i\xi\Omega}\, \varphi(k_1,k_3)\, \chi_0 (\nu^{-\beta}\, \Omega)\,\chi_1 (\nu^{-\theta}\xi)\, \chi_0 (\xi/R)\, \chi_1 (\nu^{-\theta} y)\,d\xi\, dk_1 dk_3\nonumber \\
& + \int_{\R^{2d}} \int_{\R}  \widehat{h}(\nu \xi, y) \, e^{i\xi\Omega}\, \varphi(k_1,k_3)\, \chi_0 (\nu^{-\beta}\, \Omega)\, \chi_1 (\xi/R) \, \chi_1 (\nu^{-\theta} y)\, d\xi\, dk_1 dk_3\label{eq:xi_large1}\\
& - \int_{\R^{2d}} \int_{\R}  \widehat{h}(0, y) \, e^{i\xi\Omega}\, \varphi(k_1,k_3)\, \chi_0 (\nu^{-\beta}\, \Omega)\, \chi_1 (\xi/R) \, \chi_1 (\nu^{-\theta}y)\, d\xi\, dk_1 dk_3\label{eq:xi_large2}
\end{align}
for some $R>\nu^{\theta}$ to be fixed later.

Firstly, we show that the contribution of \eqref{eq:xi_large2} is negligible. Once again, let us assume that $|k_{31}|$  is maximal among the components of $(k_1,k_3)$ as an example. Let us integrate $e^{i\xi\Omega}$ by parts in the variable $k_{11}$:
\[
\begin{split}
\eqref{eq:xi_large2} =&\ \int_{\R^{2d}} \int_{\R}  \frac{1}{(ik_{31}\xi)^2} e^{i\xi\Omega}  \frac{\pa^2}{\pa k_{11}^2} \left[ \widehat{h}(0, y) \, \varphi(k_1,k_3) \chi_0 (\nu^{-\beta}\, \Omega) \chi_1 (\nu^{-\theta}y) \right]  \chi_1 (\frac{\xi}{R}) \, \chi_{\Rc_{31}} \, d\xi \, dk_1 \, dk_3 .
\end{split}
\]
Note that $\frac{\pa y}{\pa k_{11}}$ and $\frac{\pa^2 y}{\pa k_{11}^2}$ are non-singular and easily bounded by
\begin{equation}\label{eq:deriv_y}
\Big | \frac{\pa^j y}{\pa k_{11}^j} \Big | \lesssim \langle k\rangle^{j} + \sum_{i=1}^3 \langle k_i\rangle^{j}\qquad j=1,2.
\end{equation}
Therefore the worst terms come from $\chi_0 (\nu^{-\beta}\,\Omega)$ and $\widehat{h}(0, y)$
which we bound as examples. 
\begin{itemize}
\item If both derivatives hit $\chi_0 (\nu^{-\beta}\,\Omega)$, \ref{k:bounded} yields:
\begin{align*}
& \nu^{-2\beta}\, \Big | \int_{\R^{2d}} \int_{\R}  \frac{1}{(ik_{31}\xi)^2} e^{i\xi\Omega}\, \widehat{h}(0, y) \, \varphi(k_1,k_3) \chi_0'' (\nu^{-\beta}\, \Omega) \chi_1 (\nu^{-\theta}y)\,  \chi_1 (\frac{\xi}{R}) \, \chi_{\Rc_{31}} \, d\xi \, dk_1 \, dk_3\Big |  \\
&\ \lesssim \nu^{-2\beta-\theta}\, \int_{\R} \xi^{-2}\, \chi_1 (\xi/R) \,d\xi \int_{\R^{2d}} |\varphi(k_1,k_3)|\, \chi_0'' (\nu^{-\beta} \Omega) \,dk_1 dk_3 \lesssim  R^{-1} \,\nu^{-\beta-\theta} \lesssim \nu^{\theta}
\end{align*}
after imposing
\begin{equation}\label{eq:choose_R} 
R \geq  \nu^{-\beta-2\theta} \,\gg 1.
\end{equation}
\item If both derivatives hit $\widehat{h}(0, y)$, we use the fast decay of $\varphi$ to reduce integration to $|(k_1,k_3)|\leq \nu^{-\theta}$. Then \ref{k:y0} and \eqref{eq:deriv_y} yield
\begin{multline*}
\Big | \int_{\R^{2d}} \int_{\R}  \frac{1}{(ik_{31}\xi)^2} e^{i\xi\Omega}\, \pa_{k_{11}}^2 \widehat{h}(0, y) \, \varphi(k_1,k_3) \chi_0 (\nu^{-\beta}\, \Omega) \chi_1 (\nu^{-\theta}y)\,  \chi_1 (\frac{\xi}{R}) \, \chi_{\Rc_{31}} \, d\xi \, dk_1 \, dk_3\Big |  \\
\lesssim \nu^{-4\theta}\,  \int_{\R} \xi^{-2}\, \chi_1 (\xi/R) \,d\xi \int_{\R^{2d}} |k_{31}|^{-2}\, |\varphi(k_1,k_3)|\, \chi_0(\nu^{-\beta} \Omega) \,dk_1 dk_3  \lesssim R^{-1} \,\nu^{-4\theta+\beta} \lesssim \nu^{\theta}
\end{multline*}
after imposing
\begin{equation}\label{eq:choose_R_1} 
R \geq  \nu^{\beta-5\theta} .
\end{equation}
\end{itemize}

Next we show that the contribution from \eqref{eq:xi_large1} is negligible. Assuming that $k_{31}$  is maximal among the components of $(k_1,k_3)$, we integrate $e^{i\xi\Omega}$ by parts in the variable $k_{11}$. As in the case of \eqref{eq:xi_large2}, the worst terms come from the derivative hitting $\widehat{h}$ and $\chi_0 (\nu^{-\beta}\, \Omega)$. We deal with these using \ref{k:R0} together with $|y|>\nu^{\theta}$, i.e.
\[
\begin{split}
|\eqref{eq:xi_large1}| & = \Big | \int_{\R^{2d}} \int_{\R} \frac{e^{i\xi\Omega}}{i\xi k_{31}}\, \frac{\pa}{\pa k_{11}} \left[ \widehat{h}(\nu\xi, y)\, \varphi(k_1,k_3)\, \chi_0 (\nu^{-\beta}\, \Omega) \,\chi_1 (\nu^{-\theta}y)\, \right]\, \chi_1 (\xi/R) \, d\xi\, dk_1 dk_3\Big | \\
& \lesssim  \nu^{-\theta}\, \int_{\R^{2d}\cap \{ |\Omega|\leq \nu^{\beta}\}\cap \{|y|>\nu^{\theta}\}\cap\{ |k_j|\leq \nu^{-\theta}\}} \int_{|\xi|>R} \frac{1}{|\xi| |k_{31}|} \left(  |(\pa_{y} \widehat{h})(\nu \xi, y) |  + \nu^{-\beta}  |\widehat{h}(\nu \xi, y) |\right) d\xi\, dk_1 dk_3\\
& \lesssim  \nu^{-\theta}\,  \int_{\R^{2d}\cap \{ |\Omega|\leq \nu^{\beta}\}\cap \{|y|>\nu^{\theta}\}\cap\{ |k_j|\leq \nu^{-\theta}\}} |y|^{-1}\, e^{- \nu |y| (R-R_0)} \, \left(1+ \nu^{-\beta-1} \, R^{-1}   \right)\, |k_{31}|^{-1} dk_1 dk_3 \\
& \lesssim \nu^{-2\theta}\, e^{- \nu^{1+\theta} \,(R-R_0)} \left(  \nu^{\beta} + \nu^{-1} R^{-1} \right) \lesssim_N \nu^{\theta N},
\end{split}
\]
by fixing $R$ sufficiently large, i.e.
\begin{equation}\label{eq:choose_R_2}
R > R_0 + \nu^{-1-2\theta}.
\end{equation}

All in all, we have proved that 
\[
\eqref{eq:S1_error_new}  =  \int_{\R^{2d}} \int_{\R} (\widehat{h}(\nu \xi, y) - \widehat{h}(0, y) )\, e^{i\xi\Omega}\, \varphi(k_1,k_3)\, \chi_0 (\nu^{-\beta}\, \Omega)\,\chi_1 (\nu^{-\theta}\xi)\, \chi_0 (\xi/R)\, \chi_1 (\nu^{-\theta} y)\,d\xi\, dk_1 dk_3  + \O (\nu^{\theta}).
\]

Our final goal is to show that this last term also has size $\O (\nu^{\theta})$.

\medskip

{\bf Step 3. Main terms.} Without loss of generality, let us assume that we are in the region $\Rc_{31}$ where $|k_{31}|$ is maximal, and let us integrate by parts the term $e^{i\xi\Omega}$  in the variable $k_{11}$. Thanks to the fast decay of $\varphi$ we may assume that $|(k_1,k_3)|\leq \nu^{-\theta}$ and thus $|y|\leq \nu^{-2\theta}$. We find that
\begin{multline}\label{eq:final_IBP}
\int_{\R^{2d}} \int_{\R}  e^{i\xi\Omega}\, (\widehat{h}(\nu \xi, y) - \widehat{h}(0, y))\, \varphi(k_1,k_3)\, \chi_0 (\nu^{-\beta}\Omega)\,\chi_0 (\frac{\xi}{R})\, \chi_1 (\nu^{-\theta}\xi)\, \chi_1 (\nu^{-\theta} y )\, d\xi \, dk_1\, dk_3\\
= \int_{\R^{2d}} \int_{\R} \frac{e^{i\xi\Omega}}{(i\xi \, k_{31})^2}  \frac{\pa^2}{\pa k_{11}^2} \left[ (\widehat{h}(\nu \xi, y) - \widehat{h}(0, y))\, \varphi(k_1,k_3)\, \chi_0 (\nu^{-\beta}\Omega)\,\chi_1 (\nu^{-\theta} y ) \right]\, \chi_0 (\frac{\xi}{R})\, \chi_1 (\nu^{-\theta}\xi)\,d\xi \, dk_1\, dk_3.
\end{multline}
We describe the estimates depending on which term is hit by $\pa_{k_{11}}$.

\

\begin{itemize}[nosep]
\item If both derivatives hit any of the terms $\varphi(k_1,k_3)$, $\chi_0 (\nu^{-\beta}\Omega)$ or $\chi_1 (\nu^{-\theta} y )$ we proceed in the same manner. Let us show how to bound the worst case scenario, which is both derivatives hitting $\chi_0 (\nu^{-\beta}\Omega)$. Using \ref{k:bounded} and \ref{k:lips} to estimate the kernel, we may bound the contribution of these terms by
\[
\nu^{-2\beta} \int_{\{ |\Omega|\leq \nu^{\beta}\} \cap \{ |k_j|\leq \nu^{-\theta}\}} \int_{\nu^{\theta}<|\xi|< R} \frac{\min \{ \nu^{1-3\theta} \, |\xi| ,1 \} }{|\xi|^2} \,  d\xi \, dk_1\, dk_3 \lesssim \nu^{1-\beta-6\theta} .
\]
The argument for the terms $\varphi(k_1,k_3)$ and $\chi_1 (\nu^{-\theta} y )$ is analogous, only one does not have a $\nu^{-2\beta}$ loss, and there is the additional term $|k_{31}|^{-2}$ which is integrable. As a result these terms are $\O (\nu^{1+\beta-6\theta} )$.
\item Consider the term where both derivatives in \eqref{eq:final_IBP} hit $\widehat{h}(\nu \xi, y) - \widehat{h}(0, y)$. These derivatives give rise to 
\[
(\pa_{k_{11}} y)^2\, \pa_y^2 \left[ \widehat{h}(\nu \xi, y) - \widehat{h}(0, y)\right] + (\pa_{k_{11}}^2 y)\, \pa_y \left[ \widehat{h}(\nu \xi, y) - \widehat{h}(0, y)\right] .
\]
The derivatives of $y$ may be controlled with $\nu^{-2\theta}$ in view of \eqref{eq:deriv_y}. Then we use \ref{k:y0} and \ref{k:R0} with $R=R_0$ to finish the proof.
\begin{align*}
& \int_{\{ |\Omega|\leq \nu^{\beta}\} \cap \{ |k_j|\leq \nu^{-\theta}\}}\frac{ \chi_1 (\nu^{-\theta} y )}{|k_{31}|^2}  \int_{\nu^{\theta}<|\xi|< R}  \frac{\nu^{-2\theta}}{|\xi|^2}\, \sum_{j=1}^2 \Big | \pa_y^j  \left[ \widehat{h}(\nu \xi, y) - \widehat{h}(0, y)\right]\Big | \, d\xi \, dk_1\, dk_3\\
&\ \lesssim \int_{\{ |\Omega|\leq \nu^{\beta}\} \cap \{ |k_j|\leq \nu^{-\theta}\}}\frac{ \chi_1 (\nu^{-\theta} y )}{|k_{31}|^2}  \int_{\nu^{\theta}<|\xi|< R_0}  \frac{\nu^{-2\theta}}{|\xi|^2}\, \sum_{j=1}^2\Big | \pa^j_y \left[ \widehat{h}(\nu \xi, y) - \widehat{h}(0, y)\right]\Big | \, d\xi \, dk_1\, dk_3
\\
&\ \ + \int_{\{ |\Omega|\leq \nu^{\beta}\} \cap \{ |k_j|\leq \nu^{-\theta}\}}\frac{ \chi_1 (\nu^{-\theta} y )}{|k_{31}|^2}  \int_{|\xi|>\max\{R_0,\nu^{\theta}\}}   \frac{\nu^{-2\theta}}{|\xi|^2}\, \sum_{j=1}^2 \Big | \pa^j_y \widehat{h}(\nu \xi, y) \Big | \, d\xi \, dk_1\, dk_3\\
&\ \ + \int_{\{ |\Omega|\leq \nu^{\beta}\} \cap \{ |k_j|\leq \nu^{-\theta}\}}\frac{ \chi_1 (\nu^{-\theta} y )}{|k_{31}|^2}  \int_{\max\{R_0,\nu^{\theta}\}<|\xi|< R}   \frac{\nu^{-2\theta}}{|\xi|^2}\, \sum_{j=1}^2 \Big | \pa^j_y \widehat{h}(0, y)\Big | \, d\xi \, dk_1\, dk_3\\
&\ \lesssim  \nu^{1+\beta-3\theta} \,  \log( \nu^{-\theta} \langle R_0\rangle)+ \nu^{\beta-\theta}\lesssim \nu^{\theta}
\end{align*}
after imposing that $\beta>2\theta$ and
\begin{equation}\label{eq:choose_R_3}
R_0 \leq \exp (\nu^{-1-\beta+4\theta} ).
\end{equation}

\item Let us show one example on how to handle cross-terms. Suppose that one derivative hits $\widehat{h}(\nu\xi,y)$ and the other one hits, say, $\chi_0 (\nu^{-\beta} \Omega)$. One may bound such a contribution using \ref{k:lips_y}
\begin{align*}
& \nu^{-\beta-2\theta}\, \int_{\{ |\Omega|\leq \nu^{\beta}\} \cap \{ |k_j|\leq \nu^{-\theta}\}} \int_{\nu^{-\theta}<|\xi|< R} \frac{1}{|\xi|^2} \, \Big | \pa_y \left[ \widehat{h}(\nu \xi, y) - \widehat{h}(0, y)\right]\Big | \, d\xi \, dk_1\, dk_3\\
&\ \lesssim \nu^{-\beta-2\theta}\, \int_{\{ |\Omega|\leq \nu^{\beta}\} \cap \{ |k_j|\leq \nu^{-\theta}\}} \int_{\nu^{\theta} <|\xi|< R} \frac{1}{|\xi|^2} \, \nu |\xi| \, d\xi \, dk_1\, dk_3
\lesssim \nu^{1-2\theta} \, \log (R\, \nu^{-\theta}) \lesssim \nu^{\theta}
\end{align*}
after imposing
\begin{equation}\label{eq:choose_R_4}
R \leq \exp (\nu^{-1+3\theta}).
\end{equation}
\end{itemize}

\medskip

{\bf Step 4. Choice of $R$.} In order to finish the proof, we show that a choice of $(R_0,R)$ satisfying \eqref{eq:choose_R}, \eqref{eq:choose_R_1}, \eqref{eq:choose_R_2}, \eqref{eq:choose_R_3} and \eqref{eq:choose_R_4} is possible.

In the case of $h_0$, we fix $R_0=\nu^{-1+\theta}$ and $R=\nu^{-1-2\theta}$ which satisfy all requirements. In the case of $h_1$ and $h_2$, we have that $R_0 = \max\{ T_{\mathrm{kin}}\tau,\nu^{-1}\}$ and $R= T_{\mathrm{kin}}\tau + \nu^{-2\theta-1}$ satisfy all requirements. In particular, the upper bounds \eqref{eq:choose_R_3}-\eqref{eq:choose_R_4} hold since $T_{\mathrm{kin}}$ and $\nu$ depend on $L$ polynomially.

\medskip

{\bf Step 5. Small $\tau$.} In the case of $\widehat{h}_2$, property \ref{k:lips} only holds when for $\varrho \tau \geq \nu^{\theta}$. However, the rest of the properties \ref{k:diff}, \ref{k:bounded} and \ref{k:lips_xi}-\ref{k:R0} hold even when $\varrho \tau < \nu^{\theta}$. 

In the latter regime, we want to prove that $\Ss_2(k,t,\tau)=\O(|\log \nu|)$. In order to do so, it suffices to show that \eqref{eq:S2_error2} is $\O(|\log \nu|)$. This bound follows from Steps 0-4 with the choice $\theta=\beta=0$ with minor modifications, which we detail next.

In Step 0, the integration around the singularities in $\xi$ admits the same bound. The integration around $y=0$ in \eqref{eq:y_small_1} requires the following modification:
\begin{align}
\int_{\R^{2d}} \int_{\R} & e^{i\xi\Omega}\, [\widehat{h}(\nu\xi,y)-\widehat{h}(0,y)]\, \chi_0 (2y)\,  \varphi(k_1,k_3) \widetilde{\chi}_1 (\xi)\, d\xi\, dk_1 dk_3 \nonumber \\
= &\ \int_{\R^{2d}} \int_{\R} \frac{1}{i\Omega} \, e^{i\xi\Omega}   \frac{\pa}{\pa \xi}\left[  \left( \widehat{h}(\nu\xi,y)-\widehat{h}(0,y) \right)\, \widetilde{\chi}_1 (\xi)  \right] \, \varphi(k_1,k_3)\, \chi_0 ( 2y)\, d\xi\, dk_1 dk_3, \label{eq:y_small_1_alt}
\end{align}
where $\widetilde{\chi}_1 (\xi)= \chi_1 (\xi) \chi_1 (\xi-T\varrho)\, \chi_1 (\xi+T\varrho)$. 
\begin{itemize}
\item If the derivative hits $\widetilde{\chi}_1 (\xi)$, we use \ref{k:bounded} and the localization in $\xi$ allows us to bound the $\xi$-integral.  Since $|y|<1/2$, \Cref{VIP_claim} guarantees\footnote{Once again, \Cref{VIP_claim}  actually yields $|\Omega|\gtrsim 1$, and thus the contribution of such a term is only $\O(1)$. However, the argument we present allows us to carry out the proof using only with the weaker bound $y^2 + (\nu^{-1}\Omega)^2\geq 1$, which is satisfied by many more dissipation choices $\gamma_k$.}
that $|\Omega|\gtrsim \nu$. Integrating $|\Omega|^{-1}$ then yields a bound $\O(|\log \nu |)$ for this contribution.
\item If the derivative hits $\widehat{h}(\nu\xi,y)-\widehat{h}(0,y)$, we use \ref{k:int_xi} to integrate in $\xi$. Then  \Cref{VIP_claim} allows us to integrate in $\Omega$ and obtain a contribution of size $\O(|\log \nu|)$.
\end{itemize}

In Step 1, it suffices to replace the usage of \ref{k:lips} by \ref{k:bounded}. All in all, this yields a contribution of size $\O(1)$ since $\theta=\beta=0$. 
Step 2 holds unchanged, yielding $\O(1)$ contributions.
In Step 3, all arguments hold with $\theta=\beta=0$ except one uses  \ref{k:bounded} instead of \ref{k:lips} in the first case. As a result, such terms are $\O (1)$ without gain in $\nu$.
\end{proof}

\begin{proof}[Proof of \Cref{thm:accept_kernel}]

\

{\underline {\it Kernel $h_0$.}} First consider $\widehat{h_0}(\xi,y)=2\pi\, e^{-|y\xi|}$. Properties \ref{k:diff}-\ref{k:lips_xi} are trivial, and so is \ref{k:lips_y}. Note that 
\[
|\pa_{\xi}^j \widehat{h_0}(\nu \xi,y)| = 4\pi \, (\nu |y|)^j \, e^{-\nu |y| \, |\xi|}, \qquad j=1,2,
\]
and thus 
\[
\int_{\R} |\pa_{\xi}^j \widehat{h_0}(\nu \xi,y)|  \, d\xi = 4\pi\, (\nu \, |y|)^{j-1}, \qquad j=1,2,
\]
which yields \ref{k:int_xi}.

Since $R_0=\nu^{-1}$ and $\Ac (h_0)=[-\nu^{\theta},\nu^{\theta}]$, \eqref{k:R_small} is trivially true. For any $R>0$ we have that 
\[
\int_{|\xi|>R}\frac{1}{|\xi|^j}\,  (\nu |\xi|)^j e^{-\nu |\xi|\, |y|}\, d\xi \lesssim \nu^{j-1} |y|^{-1} \, e^{-\nu |y| \, R},
\]
which in particular implies \eqref{k:R_large} and thus \ref{k:R0}.

\

{\underline {\it Kernels $h_1$ and $h_2$.}} By inspection of \eqref{eq:h}, one verifies that \ref{k:diff} holds with the singular lines being $y=0$, $\xi=0, \pm \varrho \tau$. \ref{k:bounded} also follows easily from the expressions in \eqref{eq:h}. 

Consider next \ref{k:lips} for $h_1$. We write
\[
\frac{1}{2\pi}\, \left[ \widehat{h_1} (\nu \xi, y) -\widehat{h_1} (0, y)\right] = \left( e^{-\nu |y|\,|\xi -T_{\mathrm{kin}} \tau |} - e^{-|y\,\varrho \tau |} \right) + \left( e^{-\nu |y|\,|\xi + T_{\mathrm{kin}} \tau |}- e^{-|y\,\varrho \tau |} \right).
\]
We explain how to handle the first summand, the second being analogous. It suffices to note that 
\[
|e^{-\nu |y|\,|\xi -T_{\mathrm{kin}} \tau|} - e^{-|y\,\varrho \tau|}| \leq  \Big | \int_0^{\xi} \pa_{\eta} e^{-\nu |y|\,|\eta -T_{\mathrm{kin}} \tau|} \, d\eta\Big |  =  \int_0^{|\xi|} \nu |y|\, d\eta= \nu |y| |\xi|.
\]
Since $h_1$ is a linear combination of such exponential terms, \ref{k:lips} follows.

Let us consider \ref{k:lips} for $h_2$ when $\varrho \tau>\nu^{\theta}$. Assume that $y\neq 0$. Given that $\widehat{h}_2 (\nu\xi,y)$ has a jump at $\nu\xi = \varrho \tau$, one cannot expect the Lipschitz property to hold everywhere. If $\nu\xi \in (-\varrho \tau,\varrho \tau)$, then
\[
|\widehat{h}_2 (\nu\xi,y)- \widehat{h}_2 (0,y)| \lesssim |\nu \xi| \, \sup_{\eta  \in (-\varrho \tau,\varrho \tau)} |\pa_{\eta} e^{-|y|\, |\eta -\varrho\tau|}| \lesssim \nu \, |\xi| |y|.
\]
In particular, the above bound holds whenever $|\xi|\leq \nu^{-\theta}$, since $\nu |\xi| \leq \nu^{1-\theta}\leq \nu^{\theta}<\varrho\tau$. 

Suppose next that $|\nu\xi|> \varrho\tau$. Then we use \ref{k:bounded} to obtain
\[
|\widehat{h}_2 (\nu \xi,y)- \widehat{h}_2 (0,y)| \lesssim 1\lesssim \frac{\nu |\xi|}{\varrho\tau} \leq \nu^{1-\theta} |\xi|
\]
since $\varrho\tau\geq \nu^{\theta}$. Finally, if $y=0$, $\widehat{h}_2 (\nu \xi,y)- \widehat{h}_2 (0,y)=0$ whenever $\nu|\xi|<\varrho\tau$ (in particular, this is the case if $|\xi|\leq \nu^{-\theta}$). When $\nu|\xi|>\varrho\tau$, 
\[
|\widehat{h}_2 (\nu \xi,0)- \widehat{h}_2 (0,0)|=4\pi \lesssim \frac{\nu |\xi|}{\varrho\tau} \leq \nu^{1-\theta} \, |\xi|,
\]
which concludes the proof of \ref{k:lips}. 

Next we study \ref{k:lips_xi} for $h_1$ and $h_2$. If $y=0$, away from singularities in $\xi$, we notice that $\widehat{h}_1$ and $\widehat{h}_2$ are constant in $\xi$, and thus \ref{k:lips_xi} is trivial. Let us therefore assume that $y\neq 0$.
Away from $\eta = T_{\mathrm{kin}}\tau$, we have that 
\begin{equation}\label{eq:block}
|\pa_{\eta}^j ( e^{-\nu |y|\,|\eta -T_{\mathrm{kin}} \tau|}) |= (\nu |y|)^j e^{-\nu |y|\,|\eta -T_{\mathrm{kin}} \tau|}
\end{equation}
Since $h_1$ and $h_2$ are a linear combination of similar exponential terms, \ref{k:lips_xi} follows.

Consider next \ref{k:int_xi}. As before, we may assume $y\neq 0$ since the derivative vanishes when $y=0$.
Using \eqref{eq:block} we have that 
\[
\int_{\eta>T_{\mathrm{kin}}\tau +\nu^{\theta}} (\nu |y|)^j \, e^{-\nu |y|\,(\eta -T_{\mathrm{kin}} \tau )} \, d\eta + \int_{\eta<T_{\mathrm{kin}}\tau-\nu^{\theta}} (\nu |y|)^j \, e^{-\nu |y|\,(T_{\mathrm{kin}}\tau-\eta)} \, d\eta =  2 \,(\nu |y|)^{j-1} \, e^{-\nu^{1+\theta} |y|} \lesssim (\nu\, |y|)^{j-1},
\]
which implies \ref{k:int_xi}.

We now turn to \ref{k:lips_y}. Consider $h_1$ first:
\begin{multline}\label{eq:der_h1}
\frac{\pa}{\pa y} \left[ e^{-\nu |y|\,|\xi -T_{\mathrm{kin}} \tau|}+  e^{-\nu |y|\,|\xi +T_{\mathrm{kin}} \tau |} - 2e^{-|y\,\varrho \tau |}\right] \\
= \mbox{sign}(y)\, \left( - |\nu \xi -\varrho \tau | e^{-\nu |y|\,|\xi -T_{\mathrm{kin}} \tau |} - |\nu \xi + \varrho \tau | e^{-\nu |y|\,|\xi +T_{\mathrm{kin}} \tau |}+ 2\varrho \tau e^{-|y\,\varrho \tau |} \right).
\end{multline}
We separate this expression into two summands and exploit the following cancellation between each of them:
\begin{equation}\label{eq:block2}
| \varrho\tau e^{-|y\,\varrho \tau |} - |\nu \xi -\varrho \tau | e^{-\nu |y|\,|\xi -T_{\mathrm{kin}} \tau |}| \lesssim \nu |\xi| .
\end{equation}
The above bound follows immediately if $|\xi|>T_{\mathrm{kin}}\tau$ (since it implies $\nu |\xi| > \varrho \tau$). If $\xi \in (-T_{\mathrm{kin}}\tau,T_{\mathrm{kin}}\tau)$ we write 
\[
 \varrho \tau e^{-|y\,\varrho \tau |} - (\varrho \tau  - \nu \xi ) \,  e^{-\nu |y|\,(T_{\mathrm{kin}} \tau -\xi)}  = \frac{1}{|y|}\, [F(\eta_1) -F(\eta_1-\eta_2)] 
\]
where $F(z) = z e^{-z}$, $\eta_1= |y\,\varrho \tau |>\eta_2= \nu |y \xi|$. Note that $F$ is bounded for all $z\geq 0$, and thus 
\begin{equation}\label{eq:block2_trick}
\frac{1}{|y|} \, | F(\eta_1) -F(\eta_1-\eta_2)| \leq \frac{1}{|y|} \,\left| \int_{\eta_1-\eta_2}^{\eta_1} (e^{-z} -F(z) )\, dz \right|\lesssim \frac{|\eta_2|}{|y|} = \nu |\xi|.
\end{equation}
It follows that \eqref{eq:block2} holds for all $\xi$. Since $\widehat{h_1}(\nu\xi ,y) - \widehat{h_1}(0 ,y)$ is a sum of two such terms, \ref{k:lips_y} holds for the kernel $h_1$.

Let us now prove \ref{k:lips_y} for $h_2$. Without loss of generality, suppose that $y>0$. Then 
\begin{multline}\label{eq:der_h2}
 \frac{1}{2\pi}\, \frac{\pa}{\pa y} \left[ \widehat{h_2}(\nu\xi ,y) - \widehat{h_2}(0 ,y)\right] =\frac{\pa}{\pa y} \left[  - e^{-\nu y\, |\xi -T_{\mathrm{kin}} \tau |}  \, \mbox{sign}(\xi -T_{\mathrm{kin}}\tau ) + e^{-\nu y\, |\xi +T_{\mathrm{kin}} \tau |}\, \mbox{sign}(\xi+T_{\mathrm{kin}} \tau  )  - 2 e^{-\varrho \tau y}\right] \\
  =  \nu \, (\xi -T_{\mathrm{kin}} \tau )\, e^{-\nu y \, |\xi -T_{\mathrm{kin}} \tau |} - \nu\,  (\xi+T_{\mathrm{kin}} \tau )\, e^{-\nu y\, |\xi +T_{\mathrm{kin}} \tau |} + 2\varrho \tau \, e^{-\varrho \tau  y} .
\end{multline}
The result then follows from applying \eqref{eq:block2} twice.

Let us now show \ref{k:y0} for both $h_1$ and $h_2$. For $y\neq 0$,  
\[
\widehat{h_1}(0 ,y)= 4\pi \, e^{- |y| \varrho \tau } = \mbox{sign}(y) \, \widehat{h_2}(0 ,y), 
\]
thus it suffices to note that
\[
|\pa_y^{j} e^{- |y| \varrho \tau }| = (\varrho \tau )^j \, e^{- |y| \varrho \tau } = |y|^{-j}\, \left[ (\varrho \tau |y|)^j \, e^{- |y| \varrho \tau } \right]\lesssim |y|^{-j}
\]
given that $\eta^j e^{-\eta} = \O_j (1)$.

Finally, we set out to prove \ref{k:R0}. In order to prove \eqref{k:R_large}, note that  for any $R> R_0 =\max\{T_{\mathrm{kin}}\tau,\nu^{-1}\}$,
\[
\int_{|\xi|>R} \frac{1}{|\xi|^j} \, (\nu \, |\xi - T_{\mathrm{kin}}\tau |)^j e^{-\nu |y| \, |\xi - T_{\mathrm{kin}}\tau |}\, d\xi \lesssim  \nu^j\, \int_{|\xi|>R} e^{-\nu |y| \, (\xi - T_{\mathrm{kin}}\tau)}\, d\xi = |y|^{-1} \, \nu^{j-1}\, e^{-\nu |y| \, (R- T_{\mathrm{kin}}\tau )}.
\]
Both $h_1$ and $h_2$ are a linear combination of such exponential terms, and thus \eqref{k:R_large} follows. 

Let us turn to \eqref{k:R_small}. Recall that $R_0 =\max\{T_{\mathrm{kin}}\tau,\nu^{-1}\}$. In view of the cancellation exploited in \eqref{eq:der_h1} and \eqref{eq:der_h2}, it suffices to study the $y$-derivative of the expression in \eqref{eq:block2}. Assuming, without loss of generality, that $y>0$ 
\[
\pa_y \left( \varrho \tau e^{-|y\,\varrho \tau |} - |\nu \xi -\varrho \tau| e^{-\nu |y|\,|\xi -T_{\mathrm{kin}} \tau |}\right) = -(\varrho \tau)^2 e^{-|y\,\varrho \tau|} + |\nu \xi -\varrho \tau|^2 e^{-\nu |y|\,|\xi -T_{\mathrm{kin}} \tau|}.
\]
We want to integrate this quantity in the region where $|\xi|\geq \nu^\theta$, $|\nu \xi-\varrho \tau|\geq \nu^\theta$ and $|\nu \xi |\leq \max(\rho \tau, 1)$. We claim that 
\begin{equation}\label{eq:block3}
\Big | |\nu \xi -\varrho \tau|^2 e^{-\nu |y|\,|\xi -T_{\mathrm{kin}} \tau|} -(\varrho \tau)^2 e^{-|y\,\varrho \tau|} \Big |\lesssim 
\begin{cases}
 \frac{\nu |\xi|}{|y|} \quad \mbox{when}\ |\nu \xi| <\varrho \tau,\\
 \nu |\xi| \quad \mbox{when}\ 1>|\nu \xi| \geq \varrho \tau.
 \end{cases}
\end{equation}
The first inequality follows from an argument such as \eqref{eq:block2_trick} with $F(z)=z^2 e^{-z}$, while the second inequality follows from a direct bound.

Using the fact that $h_1$ and $h_2$ are linear combinations of such terms, we find for $h\in \{h_1,h_2\}$ and $y\neq 0$ that 
\[
\begin{split}
\int_{\{\nu^{\theta} <|\xi|<R_0\}-\nu^{-1}\Ac(h)}  \frac{1}{|\xi|^2} \, |\pa_{y}^{2}[\widehat{h}(\nu \xi, y)- \widehat{h}(0, y) ]| d\xi & \lesssim
\max\{ 1, |y|^{-1}\}\, \int_{\nu^{\theta}<|\xi|<R_0}  \frac{\nu}{|\xi|} \, d\xi \\
 & \lesssim \nu\, \max\{ 1, |y|^{-1}\}\, \log (\nu^{-\theta}R_0).
 \end{split}
\]
Using \eqref{eq:block2} also yields:
\[
\int_{\{\nu^{\theta} <|\xi|<R_0\}-\nu^{-1}\Ac(h)}  \frac{1}{|\xi|^2} \, |\pa_{y}[\widehat{h}(\nu \xi, y)- \widehat{h}(0, y) ]| d\xi \lesssim 
 \int_{\nu^{\theta}<|\xi|<R_0}  \frac{\nu}{|\xi|} \, d\xi \lesssim \nu\, \log (\nu^{-\theta}R_0).
\]
This completes the proof \ref{k:R0} and the claim.
\end{proof}

We are ready to give the main convergence results in the case $\varrho\gtrsim 1$.

\begin{cor} Let $T,T_{\mathrm{kin}}$ be as in \Cref{thm:main2} and let $0 \leq t T_{\mathrm{kin}} \leq T$. Then

\begin{enumerate}[label=(\roman*)]
\item If $\varrho \in (0,\infty)$, then there exists some $0<\theta\ll 1$ such that 
\begin{equation}\label{eq:towards_forcing}
\E |u_k (tT_{\mathrm{kin}})|^2 = f(t,k)+ \int_0^t  e^{-2\varrho \gamma_k(t-t')}\, \Kc (f(t',\cdot))(k)\, dt' +\O_{\ell^{\infty}_k} ( L^{-\theta}\, t ),
\end{equation}
where  $\Kc$ is as in \eqref{eq:def_K} and $f$ is defined in \eqref{eq:def_f}. As a result, $\E |u(k,tT_{\mathrm{kin}})|^2$ is well-approximated by the solution $n(t,k)$ to
\begin{equation}\label{eq:damped_forced_WKE}
\pa_t n = \Kc (n) - 2\varrho \gamma_k \, n + 2\varrho \, b_k^2, \qquad n |_{t=0}=c_k^2.
\end{equation}

\item  If $\varrho \sim L^{\kappa}$ for some $\kappa> 0$, then there exists some $0<\theta\ll 1$ such that 
\begin{equation}\label{eq:overforced_WKE}
\E |u(k,tT_{\mathrm{kin}})|^2 = 
  c_k^2 \, e^{-2\varrho\gamma_k t}+\frac{b_k^2}{\gamma_k} \left( 1- e^{-2\varrho \gamma_k t}\right) + \O_{\ell^{\infty}_k} ( L^{-\kappa}+L^{-\theta}\, t ),
\end{equation}
which is the exact solution to
\begin{equation}
\pa_t n = - 2\varrho \gamma_k \, n + 2\varrho \, b_k^2, \qquad n |_{t=0}=c_k^2.
\end{equation}
\end{enumerate}
\end{cor}
\begin{proof}
We recall that $f$ defined in \eqref{eq:def_f} is given by
\[
\E |w_k^{(0)}(t)|^2 = c_k^2 \, e^{-2\varrho\gamma_k t}+\frac{b_k^2}{\gamma_k} \left( 1- e^{-2\varrho \gamma_k t}\right).
\]

\emph{(i)} By \Cref{lem:towardsWKE}, \eqref{eq:fromItoSigma}, \Cref{thm:sum_to_integral} and \Cref{thm:integral_to_delta}, there exists some $0<\theta\ll 1$ such that 
\[
\E|w_k^{(0)}(t) + w_k^{(1)} (t)+w_k^{(2)} (t)|^2 = 
f(t,k) + \int_0^t  e^{-2\varrho \gamma_k(t-t')}\, \Kc (f(t',\cdot))(k)\, dt'  + \O_{\ell^{\infty}_k} \left(L^{-\theta} \,t\right).
\]
By \Cref{thm:restricted_expectation} and \eqref{eq:expectation_Picard_2}, we have that if $0\leq t T_{\mathrm{kin}}\leq T$,
\[
\begin{split}
\E|u_k(tT_{\mathrm{kin}})|^2 & = f(t,k) + \int_0^t  e^{-2\varrho \gamma_k(t-t')}\, \Kc (f(t',\cdot))(k)\, dt'  + \O_{\ell^{\infty}_k} \left(L^{-\theta} t\right)
\end{split}
\]
Notice that the right-hand side is the first Picard iterate of the kinetic equation \eqref{eq:damped_forced_WKE}, and thus the result follows.

\emph{(ii)} If $\varrho=L^{\kappa}\rightarrow\infty$, we have that 
\begin{equation}
f(t,k)=c_k^2 \, e^{-2\varrho\gamma_k t}+\frac{b_k^2}{\gamma_k} \left( 1- e^{-2\varrho \gamma_k t}\right)=\begin{cases}
 c_k^2 + \O_{\ell^{\infty}_k} (L^{-N})& \mbox{if}\ t\ll \varrho^{-1},\\
  c_k^2 \, e^{-2\varrho\gamma_k t}+\frac{b_k^2}{\gamma_k} \left( 1- e^{-2\varrho \gamma_k t}\right)& \mbox{if}\ t\varrho \sim 1\\
  \frac{b_k^2}{\gamma_k} + \O_{\ell^{\infty}_k} (L^{-N}) & \mbox{if}\ t\gg \varrho^{-1}.
  \end{cases}
\end{equation}
for $N$ as large as desired. Since $\Kc (f)=\O_{\ell^{\infty}_k} (1)$, it follows that
\[
\int_0^t  e^{-2\varrho \gamma_k(t-t')}\, \Kc (f(t',\cdot))(k)\, dt' =\O_{\ell^{\infty}_k}  (L^{-\kappa}).
\]

Combining \Cref{lem:towardsWKE}, \eqref{eq:fromItoSigma}, \Cref{thm:sum_to_integral} and \Cref{thm:integral_to_delta}, we find that there exists some $0<\theta\ll 1$ such that 
\[
\E|w_k^{(0)}(t) + w_k^{(1)} (t)+  w_k^{(2)} (t)|^2  =  f(t,k) + \O_{\ell^{\infty}_k} \left(L^{-\kappa}+L^{-\theta} \, t\right).
\]
Finally, we use \Cref{thm:restricted_expectation} and \eqref{eq:expectation_Picard_2} finish the proof of \eqref{eq:overforced_WKE}.
\end{proof}

\subsection{Admissible dissipations}

A key tool in the proof of \Cref{thm:sum_to_integral} was the fact that $\Gamma_{-}^2 + (\nu^{-1} \Omega)^2\geq 1$ for our choice of dissipation $\gamma_k=(1+|k|_{\zeta}^2)^r$, $r\in (0,1]$. In particular, this means that there are no points in the intersection of a neighborhood of the manifolds $\Gamma_{-}=0$ and $\Omega=0$. This requirement may be weakened provided this intersection remains small:

\begin{lem}\label{lem:degen_diss}
Let $d\geq 2$, $T\leq L^2$, $L^{-2}<\alpha<1$ and fix a small $\theta \in (0,1)$. Consider the dissipation $\gamma_k=|k|_{\zeta}^2$. Let $k\in \Z_L^d$ satisfy $|k|\leq L^{\theta}$ and consider the set:
\[
\mathfrak{S}(k)= \{ (k_1,k_2,k_3)\in(\Z_L^d)^3 \ |\ k=k_1-k_2+k_3,\ k\neq k_1,k_3,\  |k_j|\leq L^{\theta},\ |\Omega_{13}^{2k}|<T^{-1},\ \Big |\sum_{j=1}^{3} \gamma_{k_j} - \gamma_k\Big |\leq \alpha\}.
\]
Then the number of points in $\mathfrak{S}(k)$ satisfies the following bound
\[
|\mathfrak{S}(k)|\lesssim L^{d\theta}\, L^{2d}T^{-1}\, (\alpha+T^{-1})^{\frac{d-1}{2}}.
\]
\end{lem}
\begin{proof}
Given that $|\Omega_{13}^{2k}|<T^{-1}$ and $\gamma_k=|k|_{\zeta}^2$,
\[
2|k_2|_{\zeta}^2 - T^{-1} \leq  \sum_{j=1}^{3} \gamma_{k_j} - \gamma_k = \Omega_{13}^{2k} + 2 |k_2|_{\zeta}^2 .
\]
This implies that $\mathfrak{S}(k) \subseteq \widetilde{\mathfrak{S}}(k)$, where
\[
\widetilde{\mathfrak{S}}(k)=  \{ (k_1,k_2,k_3)\in(\Z_L^d)^3 \ |\ k=k_1-k_2+k_3,\ k\neq k_1,k_3,\  |k_j|\leq L^{\theta},\ |\Omega_{13}^{2k}|<T^{-1},\ |k_2|_{\zeta}^2 \leq \frac{T^{-1} + \alpha}{2} \} \ .
\]
It thus suffices to estimate $|\widetilde{\mathfrak{S}}(k)|$. In view of the equality $\Omega_{13}^{2k}= -2 (k_1-k)\cdot (k_3-k)$, we introduce the variables  $p=L\, (k_1-k)\in \Z^d$ and $q=L\, (k_3-k)\in \Z^d$. Suppose that $p_iq_i=0$ for $1\leq i\leq j$ and $p_iq_i\neq 0$ for $j+1\leq i\leq d$, where $j\in \{0,\ldots, d\}$. The condition $|\Omega_{13}^{2k}|<T^{-1}$ implies that
\begin{equation}\label{eq:last_component}
\zeta_d p_d q_d = -\sum_{i=j+1}^{d-1} \zeta_i p_i\, q_i + \O (L^2 T^{-1}).
\end{equation}
The conditions $|k_1|, |k_3|\leq L^{\theta}$ imply that each component of $p$ and $q$ belong to an interval of size $L^{1+\theta}$. Moreover, the condition on $|k_2|_{\zeta}$ implies that
\begin{equation}\label{eq:gamma_cond}
|p+q-k|_{\zeta} \lesssim L\, \sqrt{\alpha+T^{-1}}.
\end{equation}
In particular, there are $\O( \min\{L^{(1+\theta)\, j}, L^{j}\, (\alpha+T^{-1})^{j/2}\})$ choices for the first $j$ components of $p$ and $q$ (since $p_iq_i=0$). If $j=d$, we are done since $T\leq L^2$. 

If $j=d-1$, \eqref{eq:last_component} yields $\O (L^2 T^{-1})$ choices for the product $p_d q_d$. It is well known that the number of integer pairs $(p_d,q_d)$, each smaller than $n$, such that an $p_dq_d$ is fixed is bounded by $n^{\theta}$ for $\theta\in (0,1)$ as small as desired {(this is a consequence of the divisor bound in $\mathbb Z$)}. As a result, there are $\O (L^{2+\theta} T^{-1})$ choices for the pair $(p_d,q_d)$. All in all, $|\widetilde{\mathfrak{S}}(k)|=\O( L^{d-1} (\alpha+T^{-1})^{\frac{d-1}{2}} L^{2+\theta} T^{-1})$ which concludes the proof in the case $j=d-1$.

If $j\leq d-2$, let us consider the remaining $d-j$ components of $p$ and $q$. For $j+1\leq i\leq d-1$, we fix $p_i$ freely (there are $\O ( L^{(1+\theta)(d-j-1)})$ such choices). By \eqref{eq:gamma_cond}, $q_i$ lives in an interval of size $ L\, \sqrt{\alpha+T^{-1}}$ and thus there are $\O (L^{d-j-1}\, (\alpha+T^{-1})^{\frac{d-j-1}{2}})$ choices for the remaining components of $q$. Finally, there are $\O (L^{2+\theta} T^{-1})$ choices for $(p_d,q_d)$ in view of \eqref{eq:last_component}. All in all,
\[
\begin{split}
|\widetilde{\mathfrak{S}}(k)| & \lesssim L^{j}\, (\alpha+T^{-1})^{j/2} \, L^{(1+\theta)(d-j-1)}\, L^{d-j-1}\, (\alpha+T^{-1})^{\frac{d-j-1}{2}}\, L^{2+\theta} T^{-1} \\
& \lesssim L^{d\theta}\,  L^{2d-j} \, T^{-1}\,  (\alpha+T^{-1})^{\frac{d-1}{2}}
\end{split}
\]
which is maximal when $j=0$, and yields the desired estimate.
\end{proof}

\begin{rk}\label{rk:degen_diss}
Using this result, we may extend cases (i) and (ii) of \Cref{thm:main} to the dissipation given by $\gamma_k=|k|^2_{\zeta}$ with forcing $b_k = \O (|k|_{\zeta})$ at $k=0$. Indeed, by choosing $\alpha=L^{-\mu}$ with $0<6\theta<\mu\ll 1$, one may show that the contribution of the sum over $\mathfrak{S}(k)$ towards, say, $\E |w^{(1)}(k,t)|^2$ in \eqref{eq:upper_triangle} has size 
\[
\O\left( \frac{T_{\mathrm{kin}}}{L^{2d}} \, |\mathfrak{S}(k)|\right)=\O ( L^{d\theta- \frac{d-1}{2}\, \mu})=\O(L^{-\theta}).
\]
In particular, the leading terms towards the kinetic equation are given by \eqref{eq:Sigmas} in the region $\mathfrak{S}(k)^c$. This is equivalent to adding a smooth cut-off $\chi(k_1,k_2,k_3)$ supported in $\mathfrak{S}(k)^c$ in all remaining calculations. Finally, one needs to verify that the term \eqref{eq:S2_delta} vanishes with this choice of $\gamma_k$. This easily follows from the fact that $\Gamma_{-} = \Omega + 2|k_2|_{\zeta}^2$, which implies that $\Gamma_{-}>0$ a.e.\ in the support of $\delta(\Omega)$. We remark that the condition $b_k = \O (|k|_{\zeta})$ at $k=0$ guarantees that the quantity $b_k^2/\gamma_k$ is well-defined at $k=0$. Similar conditions regarding the vanishing of the forcing around $k=0$ are not uncommon in the turbulence literature, see for instance \cite{ABCGG,BBCGR}.
\end{rk}

\section{Analysis of the remainder terms}\label{sec:error}

In this section we prove \Cref{thm:error}. As derived in \eqref{eq:erroreq}, the remainder satisfies the following equation:
\begin{equation}\label{eq:erroreq2}
 \bm{\mathcal{R}}_{N+1} = \mathcal{L}(\bm{\mathcal{R}}_{N+1}) + \mathcal{Q}(\bm{\mathcal{R}}_{N+1}) +\mathcal{C}(\bm{\mathcal{R}}_{N+1})+ \sum_{\substack{N\leq n_1+n_2+n_3 \\ n_1,n_2,n_3\leq N}}\mathcal{IW}(\bm{w}^{(n_1)}, \bm{w}^{(n_2)},\bm{w}^{(n_3)})
\end{equation}
The proof of \Cref{thm:error} is an application of the contraction mapping theorem.
The worst terms on the right-hand side of \eqref{eq:erroreq2} are the linear ones, which are essentially given by the operators
\[ \mathcal{IW}(\bm{\mathcal{R}}_{N+1}, \bm{w}^{(n_1)},\bm{w}^{(n_2)}), \quad \mbox{and}\quad \mathcal{IW}(\bm{w}^{(n_1)},\bm{\mathcal{R}}_{N+1}, \bm{w}^{(n_2)}),\]
where $n_1,n_2\leq N$. In particular, we would like to estimate the operator norm of
\[ 
\bm{v}\in h^{s,b}\ \longmapsto \mathcal{IW}(\bm{v}, \bm{w}^{(n_1)},\bm{w}^{(n_2)})\in h^{s,b} \ .
\]
By \Cref{thm:FT_higher_iterates}, it suffices to study the operators with $w^{(n_1)}$ and $w^{(n_2)}$ substituted by $\Jc_{\Tc_1}$ and $\Jc_{\Tc_2}$, where $\Tc_j$ is a ternary tree of order $n_j$, $j=1,2$. We have the following result:

\begin{prop}\label{thm:linear_error} Suppose that $T$ and $T_{\mathrm{kin}}$ are as in \eqref{eq:main2_T}- \eqref{eq:main2_Tkin}. Let $0\leq n_1,n_2\leq N$  and $\Tc_j$ be a ternary tree of order $n_j$ satisfying \eqref{eq:tree1}-\eqref{eq:tree3}. Then the operators
\begin{equation}\label{eq:linear_error}
\Pc_{+}: \bm{v}\mapsto \mathcal{IW}(\Jc_{\Tc_1},\Jc_{\Tc_2},\bm{v}), \qquad \Pc_{-}: \bm{v}\mapsto \mathcal{IW}(\Jc_{\Tc_1},\bm{v}, \Jc_{\Tc_2})
\end{equation}
satisfy the following bounds $L$-certainly
\begin{equation}
\norm{\Pc_{\pm}}_{h^{s,b}\rightarrow h^{s,b}} \leq L^{\theta}\, L^{-\delta (n_1+n_2+\frac{1}{2})}
\end{equation}
where $0<\theta\ll 1$ is as small as desired.
\end{prop}

Let us now prove \Cref{thm:error} assuming \Cref{thm:linear_error}.

\begin{proof}[Proof of \Cref{thm:error}]
Based on \eqref{eq:erroreq2}, it suffices to prove that the mapping
\begin{equation}\label{eq:contraction}
\bm{v}\longmapsto \mathcal{L}(\bm{v}) + \mathcal{Q}(\bm{v}) +\mathcal{C}(\bm{v})+ \sum_{\substack{N\leq n_1+n_2+n_3 \\ n_1,n_2,n_3\leq N}}\mathcal{IW}(\bm{w}^{(n_1)}, \bm{w}^{(n_2)},\bm{w}^{(n_3)})
\end{equation}
is a contraction from the set 
\[ 
\left\lbrace \bm{v}\in h^{s,b} : \norm{\bm{v}}_{h^{s,b}} < L^{-\delta N} \right\rbrace 
\]
to itself. To handle the last term in \eqref{eq:contraction}, we note that we can replace $\bm{w}^{(n_j)}$ by $\Jc_{\Tc_j}$ for trees $\Tc_j$ of order $n_j$. These trees uniquely determine a larger tree $\Tc$ of order $n_1+n_2+n_3+1$, thus \Cref{thm:n_iterate} yields:
\[
\norm{\mathcal{IW}(\Jc_{\Tc_1},\Jc_{\Tc_2},\Jc_{\Tc_3})}_{h^{s,b}}=\norm{\Jc_{\Tc}}_{h^{s,b}}\lesssim L^{\theta+c(b-1/2)} L^{-\delta (n_1+n_2+n_3+1)}
\]
$L$-certainly. As a consequence, the $h^{s,b}$-norm of the rightmost term in \eqref{eq:contraction} is $L$-certainly bounded by 
\[
\sum_{\substack{N\leq n_1+n_2+n_3 \\ n_1,n_2,n_3\leq N}} L^{\theta+c(b-1/2)} L^{-\delta (n_1+n_2+n_3+1)} \lesssim L^{\theta+c(b-1/2)} L^{-\delta (N+1)}\ll L^{-\delta N}.
\]
By \Cref{thm:linear_error}, the linear term admits the following bound $L$-certainly:
\[
\norm{\mathcal{L}(\bm{v})}_{h^{s,b}} \lesssim \sum_{0\leq n_1 , n_2 \leq N} L^{\theta}\, L^{-\delta (n_1+n_2+\frac{1}{2})} \norm{\bm{v}}_{h^{s,b}} \lesssim L^{\theta}\, L^{-\delta/2}\, L^{-\delta N} \ll  L^{-\delta N} .
\]
Finally, the quadratic and cubic terms are easy to estimate using the simple bound:
\[
\norm{\mathcal{IW}(\bm{f}_1,\bm{f}_2,\bm{f}_3)}_{h^{s,b}}\lesssim L^{\theta} \lambda T \, \prod_{j=1}^{3}\norm{\bm{f}_j}_{h^{s,b}}
\]
$L$-certainly for $\bm{f}_j\in \{ \bm{v},\, \bm{w}^{(n)}\, :\, n\in \N_0\}$, $j=1,2,3$. Using this bound together with the fact that $\lambda T\leq L^{d}$, we have that, $L$-certainly,
\[
\norm{\mathcal{Q}(\bm{v})}_{h^{s,b}} + \norm{\mathcal{C}(\bm{v})}_{h^{s,b}}\lesssim L^{\theta+c(b-1/2) + d} L^{-2\delta N} \ll L^{-\delta N}
\]
provided we choose $N>d/\delta$. 

This shows that the operator \eqref{eq:contraction} $L$-certainly maps the set 
\[ 
\left\lbrace \bm{v}\in h^{s,b}: \norm{\bm{v}}_{h^{s,b}} < L^{-\delta N} \right\rbrace 
\]
to itself. The proof that it is a contraction is analogous.
\end{proof}

\begin{proof}[Proof of \Cref{thm:linear_error}]

The proof is an adaptation of \cite[Proposition~2.6]{DengHani} to our setting, so we only give a rough sketch of the main ideas. In order to show that $\Pc_{\pm}$ maps $h^{s,b}$ to $h^{s,b}$, one interpolates between a simple bound with a $L^{3d}$ loss
\begin{equation}\label{eq:rough_L2bound}
\norm{\Pc_{+}\bm{v}}_{h^{s,1}} \lesssim  L^{\theta + c (b-1/2)+3d} \, L^{-\delta (n_1+n_2)}\, \norm{\bm{v}}_{h^{s,0}},
\end{equation}
and a sharper bound
\begin{equation}\label{eq:sharp_L2bound}
\begin{split}
\norm{\Pc_{+}\bm{v}}_{h^{s,1-b}} & \lesssim L^{\theta} \, L^{-(n_1+n_2+1)\delta} \norm{\bm{v}}_{h^{s,b}} .
\end{split}
\end{equation}

The proof of \eqref{eq:sharp_L2bound} requires a $TT^{\ast}$ argument applied $D\gg 1$ times, which leads to the analysis of the kernel of the operator $(\Pc_{+}\Pc_{+}^{\ast})^D$. 
This kernel may be written in the form of $\Jc_{\Tc}$ in \eqref{eq:FT_higher_iterates} for some large auxiliary tree $\Tc=\Tc(D)$ which may be constructed by attaching $2D$ copies of the trees $\Tc_1,\Tc_2$ in \eqref{eq:linear_error} successively. 
In particular, giving sharp estimates of the kernel of the operator $(\Pc_{+}\Pc_{+}^{\ast})^D$ admits the same strategy as estimating the Picard iterates of our equation, see \Cref{thm:n_iterate}  and \Cref{sec:iterates}. 
Choosing large enough $D\gg 1$ allows for $\theta$ in \eqref{eq:sharp_L2bound} to be as small as desired. The full details may be found in \cite[Proposition~2.6]{DengHani}.
\end{proof}

\appendix

\section{Some useful lemmas}\label{sec:appendix}

Let $H$ be a Gaussian Hilbert space where our standard complex Wiener processes, $\beta_k (t)$, and the Gaussian variables $\eta_k$ live. We define
\[ 
\mathcal{P}_n (H) = \{ p(X_1, \ldots , X_m) \mid \mbox{$p$ is a polynomial of degree} \leq n;\ X_1,\ldots,X_m\in H; m\in\N\}.
\]
We will often consider the closure $\overline{\mathcal{P}_n (H)}$ with respect to the topology given by the inner product in $L^2 (\Omega,\mathcal{F},\P)$.

We start with a corollary of Nelson's hypercontractivity estimate (see Theorem 5.10 and Remark 5.11 in \cite{Janson}).

\begin{lem}\label{thm:hypercontractivity}
For any $1\leq p,q<\infty$ we have that
\begin{equation}\label{eq:hypercontractivity}
 \left(\E |X|^q\right)^{1/q} \leq c(p,q)^{n} \, \left(\E |X|^p \right)^{1/p}
 \end{equation}
for all $X\in \overline{\mathcal{P}_n (H)}$, $n\geq 0$.
\end{lem}
\begin{rk}
It can be shown that $c(2,q)\leq (q-1)^{1/2}$ and $c(1,2)\leq e$, see \cite{Janson}. We will often combine these and use the estimate
\begin{equation}\label{eq:hypercontractivity2}
 \left(\E |X|^q\right)^{1/q} \leq (q-1)^{n/2} \, \left(\E |X|^2 \right)^{1/2} \leq  (q-1)^{n/2} \, e^{n}\,\E |X| .
 \end{equation}
\end{rk}

Another important result is Isserlis' theorem (also known as Wick's theorem). We record it below as stated in \cite{Janson}, Theorem 1.28.

\begin{thm}\label{thm:Iserlis}
Let $X_1,\ldots,X_n$ be centred jointly normal random variables. Then
\begin{equation}
\E \left[ X_1\cdots X_n\right] = \sum \prod_k \E [ X_{i_k} X_{j_k}],
\end{equation}
where we sum over all partitions of $\{ 1,\ldots, n\}$ into disjoint pairs $\{i_k,j_k\}$.
\end{thm}

Beyond the combinatorial objects introduced in \Cref{sec:higher_iterates}, we need some additional definitions that will allow us to count the number of decorations.

First consider a tree $\Tc$ of order $n$ with a set of pairings $\Pc$ between the leaves. Suppose we fix some of the nodes of the tree in the counting process. These nodes are ``colored'' red and form the set $\Rc$. Then we have the following definition:

\begin{defn} Suppose that we fix $n_{\sub{m}}\in\Z_L^d$ for each $\sub{m}\in\Rc\subset \Tc$ (the set of ``red'' nodes in the tree). A decoration $(k_{\sub{n}}:\sub{n}\in\Tc)$ is called \emph{strongly admissible} with respect to the pairing $\Pc$, coloring $\Rc$, and $(n_{\sub{m}} :\sub{m}\in\Rc)$ if it satisfies \eqref{eq:tree1} and 
\[ 
k_{\sub{m}}=n_{\sub{m}}\ \ \forall\sub{m}\in\Rc;\qquad |k_{\sub{l}}|\leq L^{\theta}\ \ \forall\sub{l}\in\Lc; \qquad k_{\sub{l}}=k_{\sub{l}'}\ \ \forall \ \{\sub{l},\sub{l}'\}\in \Pc .
\]
\end{defn}

The main combinatorial result is the following.

\begin{prop}\label{thm:main_counting}
Let $0\leq \theta\ll 1$. Suppose that $\Tc$ is a paired and colored ternary tree such that $\Rc\neq\emptyset$, and let $(n_{\sub{m}} : \sub{m}\in\Rc)$ be fixed. We also fix $\sigma_{\sub{n}}\in \R$ for each $\sub{n}\in\Nc$. Let $\Sc$ be the set of leaves which are not paired. Let $|\Lc|=l$ be the total number of leaves, $p=|\Pc|$ be the number of pairs, and $r=|\Rc|$ be the number of red nodes. Then the number of strongly admissible decorations $(k_{\sub{n}}:\sub{n}\in\Tc)$ that also satisfy 
\[ |\Omega_{\sub{n}} - \sigma_{\sub{n}}|\leq T^{-1} \ \ \forall\sub{n}\in\Nc,\]
is bounded by 
\begin{equation}\label{eq:main_counting}
M\leq 
\begin{cases} 
L^{\theta}\, (L^d T^{-1} \rho)^{l-p-r}  & \quad \mbox{if}\ \Rc\neq\Sc \cup \{ \sub{r}\},\\
L^{\theta}\, (L^d T^{-1} \rho)^{l-p-r+1} & \quad \mbox{if}\ \Rc=\Sc \cup \{ \sub{r}\},
\end{cases}
\end{equation}
where $\rho$ is defined in \eqref{eq:def_rho}.
\end{prop}
\begin{proof}
The proof is a description of a counting algorithm, and it may be found as Proposition 3.5 in \cite{DengHani}.
\end{proof}

The following result is Corollary 3.6 in \cite{DengHani}. The proof is based on a slight modification of the main counting algorithm.
\begin{cor}\label{thm:special_counting}
 Suppose that $\Rc=\{\sub{r}\}$  in the setting of \Cref{thm:main_counting}. Then \eqref{eq:main_counting} can be improved to
 \[
M\leq L^{\theta}\, (L^d T^{-1} \rho)^{l-p-3} \, L^{2d} T^{-1}.
\]
\end{cor}

\noindent \textbf{Acknowledgements}

\noindent  The first author thanks the University of Michigan for hosting him while part of this work was conducted. 

\noindent \textbf{Funding} 

\noindent The authors are partially supported by a Simons Collaboration Grant on Wave Turbulence. The second author is supported by NSF grant DMS-2350242 and NSF CAREER Award DMS-1936640.

\noindent \textbf{Data Availibility}  

\noindent Data sharing is not applicable to this article as no new data were created or analysed in this study.

\noindent \textbf{Declarations}

\noindent \textbf{Conflict of interest} 

\noindent The authors have no relevant financial or non-financial interests to disclose. The authors
have no conflict of interest to declare that are relevant to the content of this article.

\bibliographystyle{hsiam}

\end{document}